\documentclass[11pt]{amsart}
\usepackage{amssymb,bm,mathrsfs,xcolor}

\usepackage{tikz-cd}
\usepackage{hyperref}

\newcommand{\map}[1]{\xrightarrow{#1}}

\newcommand{\dlim}{\varinjlim}
\newcommand{\iso}{\cong}
\newcommand{\define}{\stackrel{\mathrm{def}}{=}}

\DeclareMathOperator{\Gal}{\mathrm{Gal}}

\DeclareMathOperator{\Aut}{\mathrm{Aut}}
\DeclareMathOperator{\End}{\mathrm{End}}
\newcommand{\Q}{\mathbb Q}
\newcommand{\Z}{\mathbb Z}
\newcommand{\R}{\mathbb R}
\newcommand{\C}{\mathbb C}
\newcommand{\F}{\mathbb F}
\newcommand{\A}{\mathbb A}
\newcommand{\co}{\mathcal O}

\newcommand{\sweet}{\mathrm{Sweet}}

\newcommand{\FirstIndex}{\prime}
\newcommand{\SecondIndex}{\prime\prime}

\DeclareMathOperator{\SL}{\mathrm{SL}}
\DeclareMathOperator{\GL}{\mathrm{GL}}
\DeclareMathOperator{\Sp}{\mathrm{Sp}}
\DeclareMathOperator{\Mp}{\mathrm{Mp}}

\DeclareMathOperator{\Sym}{\mathrm{Sym}}
\DeclareMathOperator{\Herm}{\mathrm{Herm}}

\begin{document}
\author{Benjamin Howard}
\title{Theta functions after Weil}
\date{\today}


\address{Department of Mathematics\\Boston College\\ 140 Commonwealth Ave. \\Chestnut Hill, MA 02467, USA}
\email{howardbe@bc.edu}

\begin{abstract}
This is a mostly expository paper on the construction of  theta functions on the Siegel and Hermitian half-spaces using the machinery of the Weil representation, especially  theta functions weighted by  polynomials.  
When the polynomial in question satisfies suitable homogeneity and harmonicity conditions, the resulting theta function  is  holomorphic.
Non-holomorphic Siegel and Hermitian theta functions can be constructed in a similar way, using a more general class of polynomials.
\end{abstract}

\maketitle
\setcounter{tocdepth}{1}
\tableofcontents

\theoremstyle{plain}
\newtheorem{theorem}{Theorem}[subsection]
\newtheorem{bigtheorem}{Theorem}[section]
\newtheorem{proposition}[theorem]{Proposition}
\newtheorem{lemma}[theorem]{Lemma}
\newtheorem{corollary}[theorem]{Corollary}
\newtheorem{conjecture}[theorem]{Conjecture}

\theoremstyle{definition}
\newtheorem{definition}[theorem]{Definition}
\newtheorem{hypothesis}[theorem]{Hypothesis}

\theoremstyle{remark}
\newtheorem{remark}[theorem]{Remark}
\newtheorem{example}[theorem]{Example}
\newtheorem{question}[theorem]{Question}

\numberwithin{equation}{subsection}
\renewcommand{\thebigtheorem}{\Alph{bigtheorem}}


\section{Introduction}


In his landmark \emph{Acta} paper \cite{Weil1}, Weil recast the classical theory of theta functions  in representation-theoretic terms, using what is now called the  \emph{Weil representation} (except when it is called the \emph{Segal-Shale-Weil representation} or the  \emph{oscillator representation}).

The  goal of this expository paper is to explain  Weil's representation-theoretic approach to the theory, and use it to  construct   interesting, and sometimes  non-holomorphic,  theta functions weighted by  polynomials  on the Siegel and Hermitian half-spaces.  

There is no shortage of literature explaining  the connection between the Weil representation and the classical theta functions  of Jacobi and Siegel.
In addition to Weil \cite{Weil1}, we refer the reader to  \cite{Igusa},   \cite{LV}, and the appendices to \cite{ShimuraEisenstein} and \cite{Shimura2000}.
Mumford's \cite{Mumford-3} is particularly thorough in exploring theta functions from several different points of view, including Weil's. 
 We hope that what we have written  differs from those works in enough ways that it is not entirely redundant.  
The main differences are that we include  non-holomorphic theta functions as objects of interest, and we systematically use the Lie algebra form of the Weil representation as the primary tool for explicit calculations over the real numbers.
Both of these choices are motivated  in part  by a desire to present the theory from a point of view consistent with that of  Kudla-Millson \cite{KM1}  and Funke-Millson  \cite{FunkeMillson}  in their construction of theta functions valued in the cohomology of symmetric spaces.

The main result on Siegel  theta functions (Theorem \ref{thm:intro main orthogonal})   already appears in the literature, but we  have given a proof  that centers the  representation theory.
To the best of our knowledge, the main results on  Hermitian theta functions do not appear in the existing literature.
In this technical sense they  may qualify as new, but the statements and proofs run  parallel to the Siegel case.

The reader looking for an exposition of  Howe duality,  the  theta lift of  automorphic forms between different groups, and connections with Langlands functoriality  should look instead to  \cite{Dprasad} and \cite{Gan}.  
Our aim is  very  specific: to prove that certain $q$-expansions define  modular forms, with the intention of applying  these results (but not in this paper)  to prove the modularity of  generating series of cycles on algebraic varieties; see \S \ref{ss:intro hermitian} below for a few more words about this last point.


\subsection{Theta functions on the upper half-plane}


Suppose $(V,Q)$ is a positive definite quadratic space over $\Q$, and  $L \subset V$ is a $\Z$-lattice on which $Q$ takes integer values.   
It is a classical theorem of Jacobi (for lattices of rank $1$) and Siegel (for higher rank),  found in most introductory courses on modular forms, that the theta function
\[
 \vartheta( z ,L ) = \sum_{ m \in \Z }   \# \{ v \in L : Q(v) = m \}    \cdot  e^{2\pi i m z} 
\]
 is a modular form of half-integral weight $2^{-1} \dim(V)$ in the variable $z\in \mathcal{H}$ of the   complex upper half-plane.
The level of this  modular form  is some congruence subgroup of $\SL_2(\Z)$ that  depends  on the data $(L,Q)$.

To generalize this classical result,  denote by 
\[
b(v_1,v_2)= Q(v_1+v_2)-Q(v_1) -Q(v_2)
\]
 the symmetric bilinear form on $V$ associated to the quadratic form, extend $b$ to an $\R$-bilinear form  on $V_\R=V\otimes_\Q\R$, and choose a basis of  $V_\R$ that is orthonormal with respect to $b$.  
Use this basis to regard functions on  $V_\R \iso \R^m$ as functions of    real variables $u_1,\ldots, u_m$, and  define the   Laplace differential operator 
\[
\Delta =    \frac{\partial^2}{ \partial u_1^2}  + \cdots +  \frac{\partial^2}{ \partial u_m^2} 
\]
on smooth functions on $V_\R$.  This operator does not depend on the choice of orthonormal basis used in its definition, and as usual a smooth function $P$ on $V_\R$ is  \emph{harmonic} if $\Delta P=0$.

Any  polynomial $P : V_\R \to \R$ is annihilated by   some power of the Laplace operator, allowing us to define a new polynomial
\[
P^\sharp = \exp\left( -\frac{\Delta}{8\pi} \right)  P  = 
\left(1 -  \frac{\Delta}{8\pi} + \frac{\Delta^2}{  2! ( 8\pi)^2  } - \frac{\Delta^3}{  3! ( 8\pi)^3 }  + \cdots   \right) P .
\]
The following result is implicit in work of Vign\'{e}ras \cite{Vigneras}, and explicit in work of Borcherds \cite[\S 4]{Borcherds}.  Both Vign\'{e}ras and Borcherds work in greater generality than we state here, proving similar modularity results for indefinite quadratic spaces.

\begin{bigtheorem}[Vign\'{e}ras, Borcherds] \label{thm:intro general theta}
Suppose $P : V_\R \to \R$ is any homogeneous polynomial of degree $N$.
The theta function
\[
 \vartheta_P( z ,L ) = y^{-N/2}
  \sum_{ m \in \Z }   \left(  \sum_{ \substack{   v \in L \\ Q(v) = m  }  }  P^\sharp\big(  \sqrt{y} \cdot  v   \big) \right)    \cdot  e^{2\pi i m z} 
\]
of the variable $z=x + iy \in \mathcal{H}$ transforms like a  modular form of weight $N+ 2^{-1} \dim(V)$, under the action of  some congruence subgroup of $\SL_2(\Z)$. 
The congruence subgroup depends on  $(L,Q)$, but not on $P$.
\end{bigtheorem}

The theta function in Theorem \ref{thm:intro general theta} is typically non-holomorphic, which is why we use the clunky phrase ``transforms like a modular form'' rather than simply calling it a modular form.
However,  in the special case where   the polynomial $P$ is  homogeneous \emph{and  harmonic}, so that $P^\sharp = P$,   Theorem \ref{thm:intro general theta} simplifies to the well-known  modularity of the holomorphic  theta function
\[
 \vartheta_P( z ,L ) = \sum_{ m \in \Z }   \left(  \sum_{ \substack{   v \in L \\ Q(v) = m  }  }  P(v)  \right)   \cdot  e^{2\pi i m z} .
\]
See Miyake's book \cite[Theorem 4.9.3]{Miyake}, for example.
Further specializing to $P=1$ recovers the classical result of Jacobi and Siegel  on the modularity of $\vartheta(z,L)$.

We will prove Theorem \ref{thm:intro general theta} by proving a more general modularity theorem for theta functions on the Siegel half-space, explained in the next subsection.


\subsection{Theta functions on the Siegel half-space}


Keep the positive definite quadratic space $(V,Q)$ and the lattice $L \subset V$ as above, and now also fix an integer $d \ge 1$.  Denote by $\mathcal{H}_d$ the Siegel half-space of symmetric $d \times d$ matrices $\mathtt{z} = \mathtt{x} + i \mathtt{y} \in \Sym_d(\C)$ whose imaginary part $\mathtt{y}\in \Sym_d(\R)$ is positive definite.

Siegel  proved that the function 
\[
 \vartheta( \mathtt{z}  ,L ) = \sum_{ T  \in \Sym_d(\Q)  }   \# \{ v \in L^d : Q(v) = T \}    \cdot  e^{2\pi i  \mathrm{Tr}( T  \mathtt{z} ) } 
\]
of the variable $\mathtt{z} \in \mathcal{H}_d$ is a Siegel modular form of weight $2^{-1} \dim(V)$, for some congruence subgroup of the Siegel modular group $\Sp_{2d}(\Z)$.  Here for any $d$-tuple $v= (v_1,\ldots, v_d)\in L^d$ of lattice vectors, we denote by
\[
Q(v) = \left( \frac{ b(v_j,v_k)  }{2}  \right)_{ 1 \le j,k \le d } \in \Sym_d(\Q)
\]
one-half of  the symmetric matrix of inner products of the components of $v$.

To generalize Theorem \ref{thm:intro general theta} to Siegel modular forms,  for each $1 \le j \le d$ denote by $\Delta_j$ the Laplace operator on the $j^\mathrm{th}$ copy of $V_\R$ in $V_\R^d$.  The differential operator 
$
\Delta = \Delta_1 + \cdots + \Delta_d
$
 is then just the Laplace operator on the space of smooth functions on $V_\R^d$, regarded as a quadratic space by taking the orthogonal direct sum of $d$ copies of $V_\R$.  
As in the previous subsection, for any polynomial function $P : V_\R^d \to \R$ we may define a new polynomial
\[
P^\sharp = \exp\left( -\frac{\Delta}{8\pi} \right)  P .
\]

The following generalization of Theorem \ref{thm:intro general theta}, which we will prove in  \S \ref{ss:orthogonal poly weights}, is due to  Roehrig \cite{Roehrig}.
The proof we give, based on calculations with the infinitesimal form of the Weil representation, is different from Roehrig's, although in some sense the main ingredients are the same.
Roehrig, like Vign\'{e}ras and Borcherds,  proves a more general theorem that includes theta functions for  indefinite quadratic spaces.
We restrict to the simpler positive definite case in keeping with the expository nature of the article.

\begin{bigtheorem}[Roehrig]\label{thm:intro main orthogonal}
Fix  an integer $N \in \Z$, and  suppose $P: V_\R^d \to \R$ is a polynomial function  satisfying the homogeneity condition 
\[
P( v g ) = \det(g)^N P(v) 
\]
 for  all $v\in V_\R^d$ and $g \in \GL_d(\R)$.    
 The  function 
\[
 \vartheta_P (\mathtt{z} ,  L ) = 
\det(\mathtt{y})^{-N/2} 
\sum_{ T \in \Sym_d( \Q)}  \left(   \sum_{ \substack{ v \in L^d \\ Q(v) = T } }   P^\sharp(v \alpha )    \right)   \, e^{2 \pi i \mathrm{Tr}( T \mathtt{z} ) } 
\]
of the variable $\mathtt{z} =\mathtt{x} + i \mathtt{y}\in  \mathcal{H}_d$ transforms like a Siegel modular form of weight $N+2^{-1} \dim(V)$, under the action of some congruence subgroup of $\Sp_{2d}(\Z)$.
Here  $\alpha \in \GL_d(\R)$  is any matrix of positive determinant such that  $\mathtt{y}=\alpha \cdot {}^t \alpha$. 
\end{bigtheorem}

In Theorem \ref{thm:intro main orthogonal}, the notations $v g$ and $v\alpha$ mean that we are viewing $v \in V_\R^d$ as a row vector with components in $V_\R$, and multiplying it on the right by a $d\times d$ matrix  in the usual way.  
 The matrix $\alpha$ in the theorem is not uniquely determined, but the value $P^\sharp(v\alpha)$ does not depend on the choice.

Just as with  Theorem \ref{thm:intro general theta}, the theta function of Theorem \ref{thm:intro main orthogonal} is typically non-holomorphic.  However, if   $\Delta P =0$ then the theta series of the theorem simplifies to  the holomorphic Siegel modular form
\[
 \vartheta_P (\mathtt{z} , L ) = 
\sum_{ T \in \Sym_d( \Q)}  \left(   \sum_{ \substack{ v \in L^d \\ Q(v) = T } }   P(v )    \right)   \, e^{2 \pi i \mathrm{Tr}( T \mathtt{z} ) } .
\]
The modularity of this theta series can be deduced from the more general result \cite[Chapter II, Theorem 7.4]{Mumford-1}, which Mumford attributes to   Freitag  and Oda, independently.
Taking  $P=1$ recovers Siegel's theta function $\vartheta(\mathtt{z},L)$ defined above.


\subsection{Theta functions on the Hermitian half-space}
\label{ss:intro hermitian}


We also prove a variant of  Theorem \ref{thm:intro main orthogonal}  in which the quadratic lattice $L$ is replaced by a Hermitian lattice, and the resulting theta series is a (typically  non-holomorphic) modular form on the Hermitian half-space.  
This is the main result of \S \ref{ss:unitary poly weights}, and  we refer the reader to Theorem \ref{thm:main unitary} for the precise statement.

There is not much  literature on Hermitian theta functions weighted by polynomials. 
As far as we know, such objects appear only in  work of Shimura, in the appendices to 
\cite{ShimuraEisenstein} and \cite{Shimura2000}.
But as  we explain in Remark \ref{rem:shimura polynomials},  Shimura works with a different class of polynomials than ours.  Outside of theta functions weighted by a \emph{constant} polynomial, there is no overlap 
 between our  Hermitian theta functions and Shimura's.

The desire to have the Hermitian analogue of  Theorem \ref{thm:intro main orthogonal} appear in the literature is among the author's  motivations for writing this expository paper.  
This theorem is used by Greer and Tayou \cite{GT}  to construct non-holomorphic Hermitian modular forms whose coefficients are classes in the cohomology of  certain CM abelian varieties
(more precisely,  tensor products of  CM elliptic curves with positive definite Hermitian lattices).
This construction can be upgraded to construct Hermitian modular forms whose coefficients lie in the Chow groups of those same CM abelian varieties.

The CM abelian varieties in question appear as boundary divisors of toroidal compactifications of Shimura varieties for unitary groups of signature $(n,1)$, and  these modular generating series of cycles on the boundary divisors are intimately related to  the  modular generating series of cycles on the interior studied  by Kudla-Millson \cite{KM1,KM2,KM3,KM4}  and Raum \cite{Raum}, and are  also closely related to work of Funke-Millson \cite{FunkeMillson,FunkeMillson2}.

The forthcoming paper  \cite{GHT}  will  extend Kudla-Millson theory to these toroidal compactifications of unitary Shimura varieties,  and  the Hermitian analogue of Theorem \ref{thm:intro main orthogonal} plays an important role in this.

\subsection{Acknowledgments}


In the summer of 2026, the author gave a series of three lectures at the research school  \emph{Classical and p-adic aspects of the Kudla program} at the CIRM in Luminy.
The first lecture was on Kudla-Millson theory \cite{KM1,KM2,KM3,KM4}.
The second was on the arithmetic Siegel-Weil formulas of Kudla-Rapoport \cite{KR14} and Li-Zhang \cite{LiZhang}.
 The third was on  the higher derivative arithmetic Siegel-Weil formula in the function field setting proved by  Feng-Yun-Zhang \cite{FYZ1},  and an extension of it proved in the author's joint work with Feng and Mkrtchyan \cite{FHM}.

Most of the content of the first two lectures can already be found in the  survey articles \cite{Kudla04} and \cite{Li}, and much of  the content of the third is included in the survey articles \cite{FengHarris} and \cite{Yun}.
Rather than  trying to improve on the  excellent existing  expositions, the author has chosen a different route.
Because Weil's representation-theoretic interpretation of the classical theory of theta functions is an essential prerequisite for making sense of any of the works cited above, and because the discussion of it in the lecture series was much too brief, the author is taking this  opportunity to discuss it in greater detail.

The author sincerely thanks Fabrizio Andreatta, Giada Grossi, Adrian Iovita, Marc-Hubert Nicole, and Joaquin Rodrigues Jacinto for  organizing a stimulating and enjoyable research school, and the CIRM for hosting it.


\section{The Weil representation}
\label{s:weil representation}


In this section we define the metaplectic cover of a symplectic group over a local field, and its Weil representation.   Then we construct the adelic metaplectic cover of a symplectic group over a global field, and define the Weil representation of this adelic group.

Essentially everything  here can be found in Weil \cite{Weil1}, but Weil's proofs often proceed through indirect and inexplicit routes, which don't lend themselves well to our ends. 
In the local setting, much of what Weil did was made more explicit   by  Rao \cite{Rao}.
Detailed expositions of the archimedean local theory can be found in \cite{Igusa} and \cite{Folland}.  
The nonarchimedean local theory is explained at length in  \cite{ThetaBook}, whose terminology and conventions we follow.   
There are aspects of the  adelic theory that  are more subtle  than one might naively expect, but details can be found in the unpublished thesis of Sweet \cite{Sweet}.


\subsection{The Stone--von-Neumann theorem}


Let $F$ be a local field of characteristic $\mathrm{char}(F) \neq 2$.  Thus $F$ is  one of $\R$, $\C$, a finite extension of $\Q_p$, or a finite extension of the field $\F_p((x))$ of Laurent series with coefficients in  the field with $p>2$ elements.  Fix a nontrivial additive character
\[
\psi : F \to \C^1 ,
\]
where $\C^1 = \{ z \in \C : | z| =1 \}$ is the circle group.

Let $W$ be a  finite dimensional vector space over $F$, equipped with a nondegenerate symplectic  form 
\[
\langle - , - \rangle  : W \times W \to F,
\]
and denote by $\Sp(W)$ the topological group (as opposed to the  reductive algebraic group) of   symplectic automorphisms of $W$.

\begin{definition}
The  \emph{Heisenberg group} is the topological group whose underlying topological space is 
\[
H(W) = W \times F , 
\]
endowed with the group law
\[
( w_1 , t_1 ) \cdot ( w_2 , t_2 ) = \left( w_1 + w_2 , t_1 + t_2 + \frac{\langle w_1 , w_2 \rangle}{2}  \right).
\]
Note that the center of the Heisenberg group is $F\iso \{ 0 \} \times F \subset H(W)$. 
\end{definition}

A short proof of the following theorem can be found in  \cite{Aprasad}.
Proofs in the archimedean and nonarchimedean cases can also be found in  \cite{Folland} and \cite{ThetaBook},  respectively.

\begin{theorem}[Stone--von-Neumann]\label{thm:SvN}
Up to isomorphism, there is a unique  continuous   irreducible  unitary Hilbert space representation $(\mathcal{S}_\psi , \rho_\psi)$ of $H(W)$ with central character $\psi$.
\end{theorem}

The Hilbert space representation $(\mathcal{S}_\psi , \rho_\psi)$  from  Theorem \ref{thm:SvN} is the \emph{Heisenberg representation} of the Heisenberg group.
Any $g\in \Sp(W)$  determines an automorphism 
$
 (w,t) \mapsto  ( gw,t)
$
of  the Heisenberg group,  fixing the center pointwise.  It follows that 
\[
\rho_\psi^g( w,t) \define  \rho_\psi ( gw,t) 
\]
defines a new representation of $H(W)$ on the same space $\mathcal{S}_\psi$, and this new representation is  again irreducible and unitary with central character $\psi$.  
By Theorem \ref{thm:SvN},  there is a unitary automorphism $M_\psi(g) : \mathcal{S}_\psi \to \mathcal{S}_\psi$ intertwining $\rho_\psi$ with $\rho_\psi^g$.    In other words, making the diagram 
\[
\begin{tikzcd}
{  \mathcal{S}_\psi    }    \ar[rr,   "{\rho_\psi(w,t)}"  ]  \ar[d , "  M_\psi(g) "  ']   &  & {  \mathcal{S}_\psi    }    \ar[d , "  M_\psi(g) " ]     \\
{  \mathcal{S}_\psi    }    \ar[rr, "{\rho^g_\psi(w,t)} " ' ]  &    & {  \mathcal{S}_\psi    }      
\end{tikzcd}
\]
commute for all $(w,t) \in H(W)$.  The unitary operator  $M_\psi(g)$ is only well-defined up to scaling by $\C^1$, and the equality $M_\psi(g_1 g_2)= M_\psi(g_1) \circ M_\psi(g_2)$  holds up to the same ambiguity.

\begin{definition}\label{def:projective weil}
The \emph{projective Weil representation} is the homomorphism 
\[
M_\psi : \Sp(W) \to \Aut ( \mathcal{S}_\psi) / \C^1 .
\]
Here $\Aut ( \mathcal{S}_\psi)$ is the group of Hilbert space (i.e.~unitary) automorphisms.
\end{definition}

\begin{definition}\label{def:abstract metaplectic} 
The \emph{metaplectic group}  is the subgroup 
\[
 \Mp(W)_\psi  \subset  \Sp(W) \times \Aut ( \mathcal{S}_\psi) 
\]
of pairs $(g,r)$ for which  $r = M_\psi(g)$ in the quotient  $\Aut ( \mathcal{S}_\psi) / \C^1$.
Projection to the first factor defines a short exact sequence
\[
1 \to \C^1 \to \Mp(W)_\psi \to \Sp(W ) \to 1 ,
\]
and projection to the second factor defines the \emph{Weil representation}
\[
\omega_\psi : \Mp(W)_\psi  \to \Aut ( \mathcal{S}_\psi) .
\]
\end{definition}


\subsection{The Schr\"{o}dinger model}
\label{ss:schrodinger model}


Keep the local field $F$, the additive character $\psi : F \to \C^1$, and the  symplectic space $W$ as in the previous subsection.   We now fix a  polarization  
\[
W = X \oplus Y  ,
\]
meaning a decomposition as a direct sum of two  totally isotropic subspaces.

Theorem \ref{thm:SvN} provides us with a distinguished isomorphism class of representations  of the Heisenberg group $H(W)$, but does not provide any concrete Hilbert space $\mathcal{S}_\psi$ on which $H(W)$ acts.
However, once we fix a model for the Heisenberg representation,    Definitions \ref{def:projective weil}  and \ref{def:abstract metaplectic}  provide  us with a  projective representation of $\Sp(W)$ and an honest representation of  $\Mp(W)_\psi$ on  the \emph{same} Hilbert space.
Let us therefore  construct a model of the Heisenberg representation.

Define a representation $\rho_\psi$ of $H(W)$ on the space of $\C$-valued  Schwartz functions $S(Y)$ by 
\begin{equation}\label{schrodinger heisenberg}
 \big(  \rho_\psi( w , t ) \varphi   \big)   (y) = \psi \left( \langle y,x_w \rangle + \frac{ \langle y_w,x_w\rangle }{2} + t \right) \cdot \varphi(y_w+y) 
\end{equation}
for all $w = x_w + y_w \in X \oplus Y =W$,  $t\in F$, and $y\in Y$.

For a proof of the following result see \cite[Theorem 1.13]{ThetaBook}, for example.

\begin{proposition}\label{prop:basic heisenberg}
The representation \eqref{schrodinger heisenberg}  of $H(W)$  is unitary with respect to the $L^2$ norm on $S(Y)$, and its extension to a unitary representation of $H(W)$ on  $L^2(Y)$ is irreducible with central character $\psi$.  
In particular, by the Stone--von-Neumann theorem, the representation $(L^2(Y) , \rho_\psi)$  is a model of the  Heisenberg representation.
\end{proposition}

\begin{definition}
The representation  of $H(W)$ on $L^2(Y)$ from Proposition \ref{prop:basic heisenberg}  is the \emph{Schr\"odinger model} of the Heisenberg representation.  By the discussion surrounding Definition \ref{def:projective weil}, it determines  homomorphisms
\begin{equation}\label{M_XY}
M_\psi : \Sp(W) \to \Aut( L^2(Y) ) / \C^1 
\end{equation}
and 
\begin{equation*} 
\omega_\psi : \Mp(W)_\psi \to \Aut( L^2(Y) ) 
\end{equation*}
called the \emph{Schr\"odinger models} of the projective Weil representation and the (honest) Weil representation.
\end{definition}

Extending results of  Weil,  Rao  gave  an explicit integral formula for the homomorphism  \eqref{M_XY}.
Each $ g \in  \mathrm{Sp}(W)$ determines a $\C$-linear operator  
\begin{equation}\label{r_XY}
r_\psi(g)  :  S(Y) \to  S(Y)
\end{equation}
in the following way:  The inverse  $g^{-1}$ is encoded by four linear maps
\begin{align*}
a    : X \to X,   & \qquad    b   : Y \to X, \\
c  : X \to Y,   & \qquad    d  : Y \to Y,
\end{align*}
defined by  the relation  $g^{-1} (x+y)  = \big(  a(x) + b(y)    \big)  + \big(   c(x) + d(y)   \big)$ for all $x\in X$ and $y\in Y$.
Define a function $f_g : W \to \C^1$ by 
\[
f_g(x + y) = \psi\left(  \frac{\langle  c(x) ,   a(x) \rangle }{2}   +  \frac{\langle d(y) , b(y) \rangle }{2}    +   \langle   c (x) ,  b (y)  \rangle  \right) ,
\]
and then define \eqref{r_XY} by 
\begin{equation}\label{big weil}
 \big(  r_\psi(g) \varphi   \big)  ( y  ) = \int_{X / \ker(c)} f_g( x+y)  \cdot  \varphi\big(   c(x)  + d(y) \big) \, d_g x .
\end{equation}
The integral here is with  respect to  Haar measure on $X / \ker(c)$, which we  normalize  according to the first assertion  of Proposition \ref{prop:rao integral} below.

For the following proposition  see  {\cite{Rao}} or {\cite[Proposition 2.12]{ThetaBook}}.
We are following here the convention of \cite{ThetaBook} that the action of $\Sp(W)$ on $W$ is on the left, whereas  \cite{Rao} and many other references follow Weil \cite{Weil1} in writing the action  on the right.

\begin{proposition}\label{prop:rao integral}
For each $g \in \Sp(W)$,   there is a unique positive real-valued Haar measure $d_g x$ for which the operator  \eqref{big weil}  is unitary with respect to the $L^2$ norm on $S(Y)$.  
If we extend  \eqref{big weil}  to a unitary automorphism of  $L^2(Y)$,   the composition 
\[
  \Sp(W)   \map{r_\psi   }  \Aut ( L^2(Y) )  \to  \Aut ( L^2(Y) ) / \C^1
\]
is  precisely the homomorphism \eqref{M_XY}.
\end{proposition}

As in \cite[Definition 2.6]{ThetaBook}, define the  \emph{Leray cocycle} 
\begin{equation}\label{leray cocycle}
c_\psi : \Sp(W) \times \Sp(W)  \to \C^1 
\end{equation}
by the relation 
\begin{equation}\label{leray relation}
 r_\psi(g_1) \circ  r_\psi(g_2) = c_\psi( g_1,g_2) \cdot  r_\psi(g_1 g_2 ).
 \end{equation}
Rao  \cite[Theorem 4.1]{Rao}  proved that the Leray cocycle takes values in $8^\mathrm{th}$ roots of unity; see also \cite[Theorem 5.2]{ThetaBook}.
We will use it to make the abstract metaplectic group 
\[
\Mp(W)_\psi  \subset \Sp(W) \times \Aut( L^2 (Y)  ) 
\]
of Definition \ref{def:abstract metaplectic} more concrete.

\begin{definition}
We call the  bijection of sets
\begin{equation}\label{leray coords}
\Sp(W) \times \C^1   \iso \Mp(W)_\psi 
\end{equation}
 sending $( g ,t ) \mapsto  (g , t  \cdot r_\psi(g)  )$  the \emph{Leray coordinates} on the metaplectic group.
Under the Leray coordinates,  the group law on the right translates to the group law
\begin{equation}\label{leray multiplication}
( g_1 , t_1 ) \cdot (g_2, t_2 ) = \big(g_1 g_2 , t_1 t_2 c_\psi(g_1,g_2) \big)
\end{equation}
on the left, and the Weil representation on $L^2(Y)$ is given by  
\begin{equation}\label{leray weil}
\omega_\psi(  g , t ) = t \cdot r_\psi(g) .
\end{equation}
\end{definition}

One thing that  the above explicit formulas make clear is that the subspace $S(Y) \subset L^2(Y)$ of Schwartz functions  is stable under the action of the Weil representation, and we usually  regard the Weil representation as a representation of $\Mp(W)_\psi$ on $S(Y)$, unitary with respect to the $L^2$-norm.

The metaplectic group $\Mp(W)_\psi$ has a natural topology on it, making it  a locally compact group for which the map to $\Sp(W)$ is continuous and open.  Because the Leray cocycle is   discontinuous, this natural topology is    \emph{not}  the  product topology on the left hand side of \eqref{leray coords}.

If $F$ is archimedean, the natural topology on the metaplectic group is the unique one making it a connected Lie group for which the  map to $\Sp(W)$ is a smooth $\C^1$-bundle.

If $F$ is nonarchimedean, the natural topology on the metaplectic group is discussed at length in \cite[Chapter 8]{ThetaBook}.
The Weil representation on $S(Y)$ is  smooth and admissible, if one interprets these words appropriately; 
because  the topology on $\Mp(W)_\psi$ is not totally disconnected,  it doesn't literally make sense to talk about its representations having these properties.
However, because the Leray cocycle $c_\psi$   takes values in $8^\mathrm{th}$ roots of unity, for  any $8 \mid m$ the subset (defined in  Leray coordinates)
\[
\Mp(W)_\psi^{(m)} =   \{ ( g ,t ) \in \Mp(W)_\psi  : t \in \mu_m \} 
\]
is a closed subgroup.    It is   locally compact and totally disconnected,
and the restriction of the Weil representation on $S(Y)$ to this subgroup is   smooth and admissible in the usual sense of groups with such a topology.  Proofs of these claims can be found in  \cite[Chapter 9.4]{ThetaBook}.


\subsection{The adelic Weil representation}
\label{ss:adelization}


Now let $F$ be a global field with $\mathrm{char}(F) \neq 2$.  
Thus $F$ is either a finite extension of $\Q$, or a finite extension of the field of rational functions $\F_p(T)$ over a finite field of odd cardinality. 
Abbreviate $\A =\A_F$ for its ring of adeles. 

Fix an additive character $\psi : F \backslash \A  \to \C^1$, and a polarized symplectic space $W= X \oplus Y$ over $F$.   The local constructions of the previous subsections can be assembled into  a central extension 
\begin{equation}\label{adelic cover}
1 \to \C^1 \to \Mp( W_\A)_\psi \to \Sp( W_\A ) \to 1 
\end{equation}
of the adelic symplectic group $\Sp( W_\A )$, and a representation of $\Mp( W_\A)_\psi$ on the adelic Schwartz space $S( Y_\A)$.

  Before we explain how these constructions work, let us explain what does  \emph{not} work. 
 For each place $v$ of $F$, recall  from \eqref{leray cocycle} the Leray cocycle
\[
c_{\psi , v} :  \Sp( W_{F_v} ) \times \Sp(  W_{F_v}) \to \C^1 , 
\]
and from  \eqref{leray weil}  the  operator
\[
r_{\psi , v }(g_v) : S(Y_{F_v}) \to S(Y_{F_v})
\]
associated to each  $g_v\in \Sp(W_{F_v})$.
One is tempted to define   $c_\psi(g_1,g_2)$  for  $ g_1,g_2\in \Sp(W_\A)$  as the product of the local cocycles, and  then define  $\Mp(W_\A)_\psi$ as the set $\Sp(W_\A) \times \C^1$  endowed with the group structure as in \eqref{leray multiplication}.  
One is then tempted to further define, for each $g\in \Sp(W_\A)$, an operator $r_\psi(g)$  on $S(Y_\A)$ as the infinite product of the local operators, and then use this to define a representation as in \eqref{leray weil}.  
Unfortunately neither of the infinite products involved is convergent, and some modifications to this approach are necessary.

 There is one important situation in which the product of the local Leray cocycles does converge.

\begin{proposition}\label{prop:leray reciprocity}
If $g_1,g_2 \in \Sp(W)$ are rational points, then 
\[
\prod_v c_{\psi, v}(g_1 , g_2 ) =1,
\]
where the product is over all places of $F$, and all but finitely many factors are equal to $1$.
\end{proposition}

\begin{proof}
This follows  from Rao's explicit formula,  \cite[Theorem 4.1]{Rao}  or \cite[Theorem 5.2]{ThetaBook},  for the local Leray cocycle  in terms of  Leray invariants and Weil indices, together with  Weil's   product formula \cite[\S 30]{Weil1} for local Weil indices.
\end{proof}

Fix  bases $e_1,\ldots, e_d \in X$ and $f_1,\ldots, f_d \in Y$ in such a way that $\langle e_j, f_k \rangle = \delta_{jk}$.  At every nonarchimedean place $v$ of $F$, let $K_v^\circ \subset \Sp(W_{F_v})$ be the stabilizer of the 
$\co_{F_v}$-lattice in $W_{F_v}$ generated by  $e_1,\ldots, e_d , f_1,\ldots, f_d$, and let $\varphi_v^\circ \in S(Y_{F_v})$ be the characteristic function of the $\co_{F_v}$-lattice in $Y_{F_v}$ generated by $f_1,\ldots, f_d$.
Thus 
\[
\Sp(W_\A) = \prod_v\nolimits^\prime \Sp(W_{F_v}) 
\qquad \mbox{and} \quad
S(Y_\A) = \bigotimes_v\nolimits^\prime  S(Y_{F_v}) ,
\]
where both products are over all places of $F$.  The first is the restricted topological product with respect to the collection of maximal compact open subgroups  $K^\circ_v $, and  the second is the restricted tensor product with respect to the collection of Schwartz functions $\varphi_v^\circ$.

The nontriviality of the function $\lambda_v$ of the following lemma is what causes the divergence of the infinite products mentioned above.

\begin{lemma}\label{lem:leray adelization}
If $v$ is a nonarchimedean place of odd residue characteristic for which $\psi_v$ is unramified, there is a function
$
\lambda_v : K_v^\circ \to \mu_8
$
satisfying 
\[
r_{\psi  , v}(k)   \varphi_v^\circ =  \lambda_v(k) \varphi_v^\circ
\]
for all $k\in K_v^\circ$.  
If $g \in \Sp(W)$ is a rational point, then  for all but finitely many places $v$ we have $g \in K_v^\circ$ and $\lambda_v(g) = 1$.
\end{lemma}

\begin{proof}
The first claim is  \cite[Lemma 7.10]{ThetaBook}.
One verifies the second claim about rational points  when  $g$ lies   in the Siegel parabolic of $\Sp(W)$ stabilizing $X$, and then also  when $g$ is the order $4$ element defined by 
$e_j \mapsto - f_j$ and $f_j \mapsto e_j$.  
In both cases  the claim follows from  the explicit formulas of \cite[Lemma 7.10]{ThetaBook}.
This proves the  claim about rational points for a set of generators of $\Sp(W)$, and the general case  follows for all elements using    \eqref{leray relation} and Proposition \ref{prop:leray reciprocity}. 
 \end{proof}

Let $\Sigma$ be a finite set of places of $F$.  For any  $g\in \Sp(W_\A)$,  define an operator 
\[
r_{ \psi, \Sigma } (g)   :  S(Y_\A) \to S(Y_\A)
\]
on pure tensors $\varphi = \otimes_v \varphi_v \in S(Y_\A)$ by 
\begin{equation}\label{Sigma operator}
r_{\psi ,\Sigma}  (g)   \varphi =  \big(  \otimes_{ v \in \Sigma} r_{\psi,v}( g_v) \varphi_v  \big) \otimes \big(  \otimes_{ v \not\in \Sigma}  \varphi_v   \big).  
\end{equation}
It follows from  the first claim of Lemma \ref{lem:leray adelization} that as one takes the  limit over  increasingly large $\Sigma$, the Schwartz function $r_{\psi ,\Sigma}  (g)   \varphi$ eventually stabilizes,  \emph{but only up to multiplication by an $8^\mathrm{th}$ root of unity}.

If $g\in \Sp(W)$, the second claim  of Lemma \ref{lem:leray adelization} implies   that  the $\mu_8$ ambiguity disappears, giving us  a well-defined operator   
\begin{equation}\label{splitting adelic operator}
r_\psi (g) = \lim_\Sigma r_{\psi , \Sigma}(g) \in  \Aut( S(Y_\A) ) .
\end{equation}
Here the limit just means $r_\psi(g) \varphi = r_{\psi,\Sigma}(g) \varphi$ for any sufficiently large (depending on $\varphi$ and $g$) finite set of places $\Sigma$.

For general $g\in \Sp(W_\A)$,  what is well-defined is only 
\begin{equation*}
r_\psi (g) = \lim_\Sigma r_{\psi , \Sigma}(g) \in  \Aut( S(Y_\A) ) / \mu_8 .
\end{equation*}
 To be completely precise about what this limit means,  fix arbitrarily a nonzero $\varphi_0 \in S( Y_\A)$.  
Then fix a  finite set of places  $\Sigma_0$ large  enough (a notion depending on $\varphi_0$ and $g$) that  for any  $\Sigma\supset \Sigma_0$ there exists a $\zeta_{\Sigma,\Sigma_0} \in \mu_8$ satisfying
\[
r_{\psi ,\Sigma }  (g)   \varphi_0 = \zeta_{\Sigma,\Sigma_0}  \cdot r_{\psi ,\Sigma_0 }  (g)   \varphi_0 .
\]
Now for any $\varphi \in S(Y_\A)$ the Schwartz function  
\[
r_\psi(g) \varphi \define \zeta_{\Sigma,\Sigma_0}^{-1}   \cdot r_{\psi ,\Sigma }  (g)  \varphi
\]
will be independent of   $\Sigma$ chosen sufficiently large.
This defines $r_\psi(g) \in \Aut(  S(Y_\A) )$, and  making different initial choices of $\varphi_0$ and $\Sigma_0$ has the effect of multiplying $r_\psi(g)$ by an $8^\mathrm{th}$ root of unity.

\begin{definition}\label{def:adelic metaplectic}
The \emph{adelic metaplectic group} is the subgroup  
\[
\Mp(W_\A)_\psi  \subset \Sp(W_\A) \times  \Aut( S(Y_\A) )
\]
of pairs $(g, r)$ such that $ r = r_\psi(g)$ in the quotient $\Aut( S(Y_\A) ) / \C^1$. 
Projection to the first factor defines the  short exact sequence \eqref{adelic cover}, 
and projection to the second factor defines the \emph{adelic Weil representation}
\[
\omega_\psi : \Mp(W_\A)_\psi  \to \Aut( S(Y_\A) ) .
\]
\end{definition}

Because the operator \eqref{splitting adelic operator} is defined  unambiguously for a rational point $g \in \Sp(W)$, there is a commutative diagram
\begin{equation}\label{canonical splitting}
\begin{tikzcd}
 & { \Mp(W_\A)_\psi  }  \ar[d]  \\
 {  \Sp(W) } \ar[ur] \ar[r]  & {   \Sp(W_\A)  } 
\end{tikzcd}
\end{equation}
in which the diagonal arrow is $g \mapsto  (g , r_\psi (g) )$.  
This  is a group homomorphism by \eqref{leray relation} and  Proposition \ref{prop:leray reciprocity},
and we may identify $\Sp(W)$ with its image in the adelic metaplectic group.

\begin{remark}\label{rem:adelic leray coordinates}
Unlike the local situation, the constructions we have made do not fix any bijection from  $\Sp(W_\A) \times \C^1$ to $\Mp(W_\A)_\psi$.  
What we actually  get is only a partial analogue of the Leray coordinates \eqref{leray coords}:  an injection  of sets
\begin{equation}\label{semi leray coordinates}
 \left( \prod_{v\in\Sigma} \Sp(W_{F_v})  \right) \times \C^1  \to \Mp(W_\A)_\psi 
\end{equation}
for every \emph{finite} set of places $\Sigma$   of $F$, sending 
$(g  , t)  \mapsto \big( g , t \cdot r_{\psi,\Sigma }(g) \big)$, 
where we view  $g \in \Sp(W_\A)$  with trivial components away from  $\Sigma $.
\end{remark}

One now topologizes $\Mp(W_\A)_\psi$ in the following way:
 Lemma  \ref{lem:leray adelization} provides us with a splitting homomorphism
\begin{equation*}
K_v^\circ \to  \Mp(W_{F_v})_\psi  \subset \Sp(W_{F_v} ) \times \Aut( S(Y_{F_v} ) ) ,
\end{equation*}
for every  nonarchimedean place $v$ of odd residue characteristic for which $\psi_v$ is unramified, 
defined by 
$k \mapsto  ( k , \lambda_v(k)^{-1}  r_{\psi,v}(k) )$. 
If we use these to regard  each $K_v^\circ$ as a compact open subgroup of $\Mp(W_{F_v})_\psi$, and then form the restricted topological product of all $\Mp(W_{F_v})_\psi$ with respect to these subgroups,   there is a canonical  surjective homomorphism 
\begin{equation}\label{metaplectic product}
\prod\nolimits_v^\prime \Mp(W_{F_v})_\psi \to \Mp(W_\A)_\psi .
\end{equation}
Endow $\Mp(W_\A)_\psi$ with the quotient topology.

The following is the main result of Weil's \emph{Acta} paper \cite[\S 41]{Weil1}, and is the cornerstone of his representation-theoretic approach to  theta functions.

\begin{theorem}[Weil]\label{thm:big theta}
For any $\varphi \in S( Y_\A)$, the sum 
\[
\Theta_\psi( g ,\varphi)  \define  \sum_{ y \in Y}  ( \omega_\psi(g) \varphi )(y) 
\]
converges to a continuous function of $g \in \Mp( W_\A)_\psi$.
This function is left-invariant under the subgroup $\Sp(W) \subset \Mp( W_\A)_\psi$.
\end{theorem}

Readers seeking to understand the proof of convergence should consult Weil's paper. 
Readers  seeking to understand the left  invariance under  $\Sp(W)$ are encouraged to work out the proof themselves.
  The point is that the left invariance immediately reduces to verifying the equality
\begin{equation}\label{distribution invariance}
\sum_{ y \in Y}  ( r_\psi(g) \varphi )(y)  = \sum_{ y \in Y}   \varphi (y) 
\end{equation}
for any $g\in \Sp(W)$ and  $\varphi \in S(Y_\A)$, and it suffices to check this as $g$ runs over a set of generators for the group.  
As a set of generators,  take the elements of the Siegel parabolic stabilizing $X \subset W$, along with the order $4$ element defined by $e_j \mapsto  - f_j$ and $f_j \mapsto e_j$. 
If $g$ is in the Siegel parabolic then \eqref{distribution invariance} follows directly from the definition \eqref{big weil}.  
For the order $4$ element, which acts on $S(Y_\A)$ by the Fourier transform, use Poisson summation.


\subsection{The metaplectic double cover}


We now discuss the  metaplectic  double cover and its Weil representation, in both the local and adelic contexts.

Begin with the local situation.  Return to the  polarized symplectic space $W=X\oplus Y$ over a local field $F$ from  \S \ref{ss:schrodinger model}.

As was already noted,  the Leray cocycle $c_\psi$ from \eqref{leray cocycle} actually   takes values in $\mu_8$.
Rao \cite[\S 5]{Rao} proves this, see also   \cite[Chapter 5.2]{ThetaBook},  by constructing  an explicit  function 
\[
m_\psi : \Sp(W) \to \mu_8 ,
\]
called \emph{Rao's normalizing factor},   in such a way that the \emph{Rao cocycle} 
\begin{equation}\label{rao cocycle}
c_\psi^{(2)}(g_1,g_2) \define c_\psi (g_1,g_2) \cdot \frac{ m_\psi(g_1) m_\psi (g_2) }{  m_\psi(g_1g_2)  }  
\end{equation}
lies in $\mu_2$.
For any $g \in \Sp(W)$, we  renormalize the operator \eqref{big weil} by
\begin{equation}\label{rao correction}
r_\psi^{(2)}(g) \define  m_\psi(g) \cdot r_\psi(g)  : S(Y) \to S(Y) ,
\end{equation}
so that  $c_\psi^{(2)}$ and $r_\psi^{(2)}$ are related as in \eqref{leray relation}.

\begin{definition}
The \emph{metaplectic double cover} of $\Sp(W)$ is the subgroup 
\[
\Mp(W)_\psi^{(2)}   \subset \Sp(W)  \times \Aut ( S(Y) )  
\]
of pairs $(g, r)$ for which $r =  \pm r^{(2)}_\psi(g)$.  
\end{definition}

It is clear from the construction that the inclusion $\Mp(W)_\psi^{(2)}   \subset \Mp(W)_\psi$ holds as subsets of $\Sp(W)  \times \Aut ( S(Y) )$, and  we have a commutative diagram
\begin{equation}\label{double inclusion}
\begin{tikzcd}
{ 1 }  \ar[r]   &  {  \mu_2 } \ar[r] \ar[d]  & {  \Mp(W)_\psi^{(2)} } \ar[r] \ar[d] & {  \Sp(W)  } \ar[d, equal] \ar[r] & {  1 } \\ 
{ 1 }  \ar[r]   &  {  \C^1 } \ar[r]   & {  \Mp(W)_\psi } \ar[r]  & {  \Sp(W)  }   \ar[r] & {  1 } 
\end{tikzcd}
\end{equation}
with exact rows.

\begin{definition}
The bijection of sets
\begin{equation}\label{rao coords}
\Sp(W) \times \mu_2 \iso \Mp(W)_\psi^{(2)} 
\end{equation}
sending $(g, \epsilon ) \mapsto \big( g , \epsilon  \cdot r_\psi^{(2)}(g)\big)$  is the \emph{Rao coordinates}. 
The group law on the left is given by the obvious analogue of \eqref{leray multiplication}, with $c_\psi$ replaced by $c_\psi^{(2)}$.
\end{definition}

If we now switch to the global setting of \S \ref{ss:adelization},  the adelization of the local metaplectic double covers goes through essentially verbatim.  The essential points are
\begin{itemize}
\item Proposition \ref{prop:leray reciprocity} holds with every $c_{\psi,v}$ replaced by $c_{\psi,v}^{(2)}$.
This is because for any rational point $g \in \Sp(W)$ we have $\prod_v m_{\psi,v}(g)=1$.
\item
Lemma \ref{lem:leray adelization} holds  with $r_{\psi,v}$ replaced by $r_{\psi , v}^{(2)}$, and with an improvement: one can replace the $\mu_8$ appearing there with $\mu_2$.  
This is because the product of  Rao's  local normalizing factor $m_{\psi,v}(k)$  with the $\lambda_v(k)$ of Lemma \ref{lem:leray adelization} lies in $\mu_2$ for any $k \in K_v^\circ$, which is the content of the proof of \cite[Theorem 7.13]{ThetaBook}.
\end{itemize}
As a consequence, we may imitate \eqref{Sigma operator} to define, for any $g \in \Sp(W_\A)$  and any finite set of places $\Sigma$ of $F$, an operator 
\[
r^{(2)}_{ \psi, \Sigma } (g)   :  S(Y_\A) \to S(Y_\A).
\]
For a fixed $\varphi \in S(Y_\A)$,  the limit of $r^{(2)}_{ \psi, \Sigma } (g)  \varphi$ over increasing $\Sigma$  stabilizes, but only  up to multiplication by $\mu_2$.
As before, this allows us to define
\begin{equation}\label{double r limit}
r^{(2)}_\psi (g) = \dlim_\Sigma r^{(2)}_{\psi , \Sigma}(g) \in  \Aut( S(Y_\A) ) / \mu_2 .
\end{equation}
Directly imitating Definition \ref{def:adelic metaplectic}, but with $r_\psi$ replaced by $r_\psi^{(2)}$ and $\C^1$ replaced by $\mu_2$, we obtain  the top row  in the  commutative diagram
\[
\begin{tikzcd}
{ 1 }  \ar[r]   &  {  \mu_2 } \ar[r] \ar[d]  & {  \Mp(W_\A)_\psi^{(2)} } \ar[r] \ar[d] & {  \Sp(W_\A)  } \ar[d, equal] \ar[r] & {  1 } \\ 
{ 1 }  \ar[r]   &  {  \C^1 } \ar[r]   & {  \Mp(W_\A)_\psi } \ar[r]  & {  \Sp(W_\A)  }   \ar[r] & {  1 } .
\end{tikzcd}
\]

As in  \eqref{splitting adelic operator},  for a rational point $g \in \Sp(W)$ the $\mu_2$-ambiguity disappears in the definition of \eqref{double r limit}, giving a well-defined   operator
\[
r^{(2)}_\psi (g)  \in  \Aut( S(Y_\A) ) .
\]
Because the local Rao factors satisfy $\prod_v m_{\psi,v}(g_v) =1$, this actually agrees with 
 \eqref{splitting adelic operator}, and so  \eqref{canonical splitting}  can be upgraded to 
 \begin{equation}\label{double splitting}
\begin{tikzcd}
 & { \Mp(W_\A)_\psi^{(2)}   }  \ar[d]   \ar[r]  &  {   \Mp(W_\A)_\psi }  \ar[d]   \\
 {  \Sp(W) } \ar[ur] \ar[r]  & {   \Sp(W_\A)  }  \ar[r, equal] & { \Sp(W_\A)   }  
\end{tikzcd}
\end{equation}
in which the diagonal arrow is  $g \mapsto ( g , r^{(2)}_\psi (g) ) = ( g , r_\psi (g) )$.

\begin{remark}\label{rem:adelic rao coordinates}
As in Remark \ref{rem:adelic leray coordinates},  the above constructions do not fix any bijection from  $\Sp(W_\A) \times \mu_2$ to $\Mp(W_\A)_\psi^{(2)}$.   What we actually have is only a partial analogue  of the Rao coordinates \eqref{rao coords}:  an injection  
\begin{equation}\label{semi rao coordinates}
 \left( \prod_{v\in\Sigma} \Sp(W_{F_v})  \right) \times \mu_2  \to \Mp(W_\A)_\psi ^{(2)} 
\end{equation}
for every \emph{finite} set of places $\Sigma$  of $F$, sending 
$(g  , t)  \mapsto \big( g , t \cdot r^{(2)} _{\psi,\Sigma }(g) \big) $.
\end{remark}


\subsection{Sweet's renormalization}
\label{ss:sweet coordinates}


The contents of this subsection will   be mentioned later in passing, but will not play a significant role in what follows.

Return to the setting of \S \ref{ss:adelization}, so that $F$ is a global field.
Fix a finite set $\Sigma$ of places of $F$ containing all archimedean places, all places at which the additive character $\psi$ is ramified, and all nonarchimedean places with residue characteristic $2$.

We have  pointed out in Remarks \ref{rem:adelic leray coordinates} and \ref{rem:adelic rao coordinates} that we have not fixed any bijections from  the adelic metaplectic groups
$\Mp(W_\A)_\psi$ and $\Mp(W_\A)_\psi^{(2)}$ to  the sets $\Sp(W_\A) \times \C^1$ and $\Sp(W_\A) \times \mu_2$, and as a result we don't have names for most elements of these   groups.  This is  more of an aesthetic shortcoming than an actual problem, but we explain in this subsection how Sweet  \cite{Sweet} addresses the issue.

As noted in the previous subsection, there is a version of Lemma \ref{lem:leray adelization} that holds  with $r_{\psi,v}$ replaced by $r_{\psi , v}^{(2)}$.  
 That is to say, for every place $v\not\in  \Sigma$, there is a function
 $\lambda_v^{(2)} : K_v^\circ \to \mu_2$ such that 
 \[
 r^{(2)}_{\psi,v} (k) \varphi_v^\circ = \lambda_v^{(2)}(k) \varphi_v^\circ
 \]
 for all $k \in K_v^\circ$.    Sweet \cite[Proposition 1.6.6]{Sweet} derives an explicit formula for this function, and uses his  formula to  extend $\lambda_v^{(2)}$ to a function
 \[
 \lambda_v^\sweet : \Sp(W_{F_v})  \to \mu_2 .
 \]
 For $v \in \Sigma$, set  $ \lambda_v^\sweet =1$.

 For any place $v$ and any $g_v\in \Sp(W_{F_v})$, we obtain a   renormalized operator 
 \[
 r^\sweet_{\psi,v} (g_v) = \lambda_v^\sweet (g_v)^{-1}  \cdot   r^{(2)}_{\psi,v} (g_v)  \in \Aut ( S(Y_{F_v}  )) .
 \]
By construction,  this operator  fixes the vector $\varphi_v^\circ$ whenever   $g_v\in K_v^\circ$.
 This last observation allows us to form, for an adelic point $g \in \Sp(W_\A)$,  the tensor product  
 \[
 r^\sweet_{\psi } (g)   = \otimes_v   r^\sweet_{\psi ,v} (g_v )    \in \Aut( S(Y_\A)) 
 \]
 over all places of $F$.   From this we may define a  bijection  of sets
 \[
 \Sp(W_\A) \times \mu_2  \iso  \Mp(W_\A)_\psi^{(2)}  
 \]
 by $(g,\epsilon ) \mapsto \big( g , \epsilon \cdot   r^\sweet_{\psi } (g) \big)$, and the same formula defines 
  \[
 \Sp(W_\A) \times \C^1  \iso  \Mp(W_\A)_\psi .
 \]
We call both bijections the  \emph{Sweet coordinates}.
From  the collection of operators $r^\sweet_\psi(g)$ we  can imitate  the definition of \eqref{leray cocycle}  to define the corresponding \emph{Sweet cocycle}  
\[
c^\sweet_\psi  : \Sp(W_\A) \times  \Sp(W_\A) \to \mu_2 ,
\]
which factors as a product of local cocycles.  
 
The above renormalization has the advantage that all elements of the adelic metaplectic groups now have names, via the Sweet coordinates.  
Moreover, the  warnings at the beginning of \S \ref{ss:adelization}  do not apply to Sweet's renormalized cocycles and operators.  After this renormalization, the adelic theory really is just the product of the everywhere-local theories.

There are also disadvantages to this  renormalization:
\begin{itemize}
\item The  extra  renormalization factors make it harder to keep track of explicit formulas for the Weil representation. 
If we use only the Leray and Rao coordinates, all the formulas we need are readily available in \cite{ThetaBook}.
  (Sweet's  normalizing factors $\lambda_v^\sweet$ are identically $1$ on the Siegel parabolic stabilizing $X \subset W$, so at least some formulas are unaffected.)  
\item 
The collection of normalizing factors $\{  \lambda_v ^\sweet \}_v $ is not really canonical.  
It genuinely depends not only on the choice of $\Sigma$, but also on the choice of basis $e_1,\ldots, e_d,f_1,\ldots, f_d \in W$ used in \S \ref{ss:adelization} to define the  compact open subgroups  $\{ K_v^\circ \}_v$.
Making a different choice of basis  will change finitely many members of the family $\{   \lambda_v^\sweet \}_v $, resulting in different Sweet coordinates.
This is in contrast to the partial Leray \eqref{semi leray coordinates} and partial Rao \eqref{semi rao coordinates} coordinates, which are  basis-independent.
\end{itemize}

For these reasons,  we have chosen to avoid Sweet's renormalization in what follows.  The penalty for this choice is that some adelic statements, most notably Proposition \ref{prop:so adelization}, are not as clean as one might like.
 

\section{The infinitesimal Weil representation}
\label{s:infinitesimal weil}


In this section we work with a fixed  polarized  symplectic space $W=X\oplus Y$ over the field $F=\R$.
The symplectic form is denoted $\langle - , - \rangle$ as before, and we  take the usual 
\[
\psi(x) = e^{2 \pi i x}
\]
for our additive character $\psi : \R \to \C^1$.

Our goal is to describe the infinitesimal form of the Weil representation of $\Mp(W)_\psi$ on $S(Y)$, that is,  the induced action of the real Lie algebra on the same space.
Although \cite{Folland} is the most complete reference, our exposition is closer to   \cite{KM4},  \cite[Appendix A]{FunkeMillson}, and \cite{Adams}.  See also \cite{Howe-remarks} and \cite{Howe-trans}.


\subsection{Modules over the Weyl algebra}


There is a general method of constructing representations of the symplectic Lie algebra $\mathfrak{sp}(W)$ from modules over an associative $\C$-algebra, called the Weyl algebra (or quantum algebra).

\begin{definition}\label{def:weyl algebra}
The \emph{Weyl algebra}  $\mathcal{W}$  is the quotient of the complex tensor algebra 
\[
W_\C^{\otimes}  \define  \bigoplus_{\ell \ge 0} W_\C^{\otimes \ell} 
\]
 by the two-sided ideal generated by the subset 
 \[
\{  w_1 \otimes w_2 - w_2 \otimes w_1 - 2\pi i  \langle w_1,w_2\rangle  : w_1,w_2 \in W_\C \} .
 \]
\end{definition}

Suppose $M$ is a $\C$-vector space, and we are given an $\R$-linear map
\[
\rho : W \to \End_\C(M) 
\]
satisfying the commutator relations
\begin{equation}\label{Weyl module relations}
\rho(w_1) \circ \rho(w_2) - \rho(w_2) \circ \rho(w_1) = 2\pi i  \langle w_1,w_2\rangle .
\end{equation}
The function  $\rho$ extends uniquely to a $\C$-algebra homomorphism $W_\C^\otimes \to \End_\C(M)$, which then factors through a homomorphism  
\[
\rho : \mathcal{W} \to \End_\C(M) .
\]
This makes  $M$ into a left $\mathcal{W}$-module.

We  make the space of Schwartz functions $S(Y)$ into a left $\mathcal{W}$-module as follows:   For each $e \in X$, define an  endomorphism of $S(Y)$ by 
\[
 \rho (e)  \varphi   =  2\pi i  \langle e , \cdot  \rangle \,  \varphi .
\]
 For each  $ f \in Y$, define an  endomorphism of $S(Y)$ by
 \[
 \rho  (f)  \varphi = - D_f\,  \varphi ,
 \]
where $D_f$ is  the  usual directional derivative  in the direction $f$.  
 If we extend these definitions  $\R$-linearly  to  $\rho: W \to \End_\C(S(Y))$, the commutator relations  \eqref{Weyl module relations} are satisfied, and  we obtain a $\C$-algebra map
 \begin{equation}\label{Weyl module}
\rho :  \mathcal{W} \to \End_\C (S( Y ) ) .
 \end{equation}

 \begin{remark}\label{rem:weyl action}
 To be completely explicit,  choose    bases  
 \[
 e_1,\ldots, e_d \in X  \qquad \mbox{and} \qquad f_1 ,\ldots, f_d \in Y
 \]
 in such a way that  $\langle e_j,f_k\rangle = \delta_{jk}$, and  identify  $\R^d \iso Y$ by 
 $(y_1,\ldots, y_d) \mapsto y_1 f_1+\cdots+ y_d f_d$. 
 Identifying vectors in $W$ with their images in the Weyl algebra $\mathcal{W}$, our basis elements act on 
  $S(\R^d) \iso S(Y)$   through the differential operators
 \[
 \rho(e_j) =  2 \pi i  y_j  \qquad \mbox{and} \qquad \rho (f_j) =-  \frac{ \partial }{ \partial y_j }.
 \]
 \end{remark}

Denote by  $\mathfrak{sp}(W)$  the real Lie algebra of $\Sp(W)$.
We claim that there is a natural  injection
   \begin{equation}\label{quantum iota}
\iota :  \mathfrak{sp}( W )  \to  \mathcal{W} 
 \end{equation}
 as Lie algebras; that is to say, satisfying
\[
 \iota ( [ x,y ] ) = \iota(x) \iota(y) -  \iota(y) \iota(x) 
 \]
 for all $x,y\in  \mathfrak{sp}(W)$.
To construct \eqref{quantum iota}, first define an $\R$-linear bijection
\begin{equation}\label{sym2sp}
 \Sym^2( W) \iso \mathfrak{sp}(W) ,
 \end{equation}
 by sending a symmetric tensor $w_1  w_2 \in  \Sym^2( W)$  to the endomorphism of $W$ defined by 
$
w \mapsto  \langle w_1,w\rangle  w_2 + \langle w_2 , w\rangle  w_1 .
$
Next,  define an $\R$-linear injection  
\begin{equation}\label{Sym2Weyl}
  \Sym^2( W ) \to   W_\C^\otimes  \to  \mathcal{W}
 \end{equation}
    by 
 \[
  w_1 w_2  \mapsto   \frac{ 1 }{4\pi i } ( w_1 \otimes w_2 + w_2\otimes w_1 )
  = \frac{1}{2\pi i}  w_1 \otimes w_2 - \frac{1}{2}  \langle  w_1 , w_2  \rangle  .
\]
 The Lie algebra map \eqref{quantum iota} is  defined by composing  \eqref{sym2sp} and \eqref{Sym2Weyl}.

\begin{remark}
The injection \eqref{quantum iota} induces a homomorphism from the universal enveloping algebra of $\mathfrak{sp}(W)$ to the Weyl algebra $\mathcal{W}$,  but this is neither injective nor surjective.
The Weyl algebra could be viewed as an approximation to  the universal enveloping algebra,  particularly well-suited to studying the Weil representation.
\end{remark}

 Using  \eqref{quantum iota},  any left $\mathcal{W}$-module can  be viewed as  a Lie algebra representation of  $\mathfrak{sp}(W)$.  In particular, this applies  to the $\mathcal{W}$-module \eqref{Weyl module}.

\begin{definition}\label{def:Lie Schrodinger}
The  \emph{infinitesimal  Weil representation} is the Lie algebra representation 
 \[
 \omega^\mathrm{inf} = \rho \circ \iota :  \mathfrak{sp}(W)   \to \End_\C(S( Y ) ) 
 \]
 obtained by composing \eqref{Weyl module} and \eqref{quantum iota}.
 \end{definition}

Recall  the metaplectic cover 
\[
1 \to \C^1 \to \Mp(W)_\psi \to  \Sp(W ) \to 1 
\]
from \S \ref{s:weil representation}.     Its real Lie algebra 
\begin{equation}\label{mp lie}
\mathfrak{mp}(W) =  \mathfrak{sp}(W) \oplus i\R 
\end{equation}
is reductive with center $i\R$.  

The appeal  of Definition  \ref{def:Lie Schrodinger} is that the action of $ \mathfrak{sp}(W)$ on $S(Y)$ is through very explicit differential operators, in contrast to the slightly unwieldy  integral formula \eqref{big weil} underlying the action of the metaplectic group $\Mp(W)_\psi$ on the same space via the Weil representation $\omega_\psi$ of 
\S \ref{s:weil representation}.  The following theorem explains the precise connection between the two.

\begin{theorem} \label{thm:inf compatibility}
Under the Weil representation $\omega_\psi$  of $\Mp(W)_\psi$ on $L^2(Y)$,  the  Schwartz functions $S(Y)$ are  precisely the $C^\infty$ vectors; that is to say, those vectors $\varphi \in L^2(Y)$ for which the orbit map 
\[
\Mp(W)_\psi \map{  g  \to \omega_\psi(  g ) \varphi }  L^2(Y)
\]
 is infinitely differentiable.
Under the induced differentiated action of the Lie algebra \eqref{mp lie} on $S(Y)$, 
 the  first  summand in \eqref{mp lie} acts through the infinitesimal Weil representation 
$\omega^\mathrm{inf}$ from 
Definition \ref{def:Lie Schrodinger}, and 
the second summand acts through the inclusion  $i\R \to \C$.
\end{theorem}

The first claim of Theorem \ref{thm:inf compatibility} is \cite[VIII.1.11]{BorelWallach}. 
The second is a restatement of \cite[(4.45)]{Folland} in different language.  
A form of the theorem  appears  in work of Howe,  spread across   \cite{Howe-theta},  \cite{Howe-remarks}, and  \cite{Howe-trans}.


\subsection{The vacuum vector}
\label{ss:vacuum vector}


Fix  $J \in \Sp(W)$ satisfying $J^2=-1$.   
Denote by  $W_J$ the complex vector space whose underlying real vector space is $W$,  with $i \in \C$ acting via $i\cdot w=J w$.
It is easy to check that  
\[
H_J( w_1 , w_2 ) =   \langle w_1 , J w_2 \rangle - i \langle w_1,w_2 \rangle
\]
defines   a nondegenerate Hermitian form on $W_J$, and we  assume that $J$ is chosen so that  it is positive definite.  This ensures that the  centralizer 
\[
K_J = \{ g \in \Sp(W) : g  J = J g \} \iso \mathrm{U}( W_J ) 
\]
is a maximal compact subgroup of $\Sp(W)$.  
The centralizer has a distinguished character
\begin{equation}\label{chi_J}
\chi_J : K_J \to \C^1,
\end{equation}
defined as the  \emph{inverse} of the determinant of the action  of $K_J$ on $W_J$.

\begin{definition}\label{def:vacuum}
The \emph{vacuum vector} $\varphi^\circ_J \in S(Y)$ is 
\[
\varphi^\circ_J(y) =  e^{- \pi  \langle y , J y \rangle } .
\]
In other words, it is the restriction to $Y \subset W $ of the Schwartz function 
\[
\varphi^\circ_J (w) =  e^{   - \pi   H_J(w,w) } .
\]
\end{definition}


Denote by  
 $
 \mathfrak{k}_J \subset \mathfrak{sp}(W)
 $
  the real Lie algebra of $K_J$. 
Differentiating $\chi_J$  yields a linear functional $ \chi_J^{\mathrm{inf}}: \mathfrak{k}_J \to i\R$, defined by the commutativity of 
\begin{equation}\label{symplectic trace}
\begin{tikzcd}
{  \mathfrak{k}_J   }  \ar[r , " \chi_J^{\mathrm{inf}}"  ]   \ar[d, " \exp" ' ]  & { i \R  } \ar[d, " \exp"  ]  \\
{  K_J  }  \ar[r , " \chi_J " ]  & { \C^1 }  .
\end{tikzcd}
\end{equation}

Denote by $\tilde{K}_J$ the preimage of $K_J$ under the covering map 
 $\Mp(W)_\psi \to \Sp(W)$.
Recalling \eqref{mp lie}, the  Lie algebra of  $\tilde{K}_J$  is  the subalgebra 
\[
 \mathfrak{k}_J\oplus i\R  \subset \mathfrak{sp}(W) \oplus i\R .
 \]

\begin{proposition}\label{prop:vacuum eigenvector} 
Under the restriction of the Weil representation,  the subgroup  $\tilde{K}_J\subset \Mp(W)_\psi$  acts on the vacuum vector $\varphi_J^\circ$ through the unique character whose differential
$
\mathfrak{k}_J  \oplus i\R  \to i \R 
$
restricts to $\frac{1}{2} \chi_J^\mathrm{inf}$ on the first  summand, and to  the identity on the second.
\end{proposition}

Proposition \ref{prop:vacuum eigenvector} can be proved  by directly computing the action of the integral operators \eqref{big weil} on the vacuum vector. 
 This is the approach taken in \cite[Chapter I.9, Lemma 11]{Igusa}.
We will spend the rest of this subsection giving a different proof, using Theorem \ref{thm:inf compatibility} to reduce the problem to showing that the vacuum vector is a solution to a concrete system of differential equations.
The argument serves  as a prototype for the more involved calculations lying ahead in \S \ref{ss:orthogonal poly weights}  and \S \ref{ss:unitary poly weights}.

\begin{remark}
As soon as $\dim(W) > 2$,  the span of the vacuum vector is the only line in  $L^2(Y)$ that is stable under the action of  $\tilde{K}_J$.  See \cite[Chapter I.9, Theorem 8]{Igusa}.
When  $\dim(W) =2$ there are many lines in  $L^2(Y)$ that are stable under the action of $\tilde{K}_J$, but the span of the vacuum vector is the unique one on which $\tilde{K}_J$ acts through the character specified in Proposition \ref{prop:vacuum eigenvector}.    In either case, Proposition \ref{prop:vacuum eigenvector} characterizes the vacuum vector up to scaling. 
\end{remark}

All possible choices of $J \in \Sp(W)$ are conjugate, and so in order to prove Proposition \ref{prop:vacuum eigenvector} for all of them, it suffices to do so for one.
Choose a basis  
$
e_1,\ldots, e_d , f_1,\ldots, f_d\in W
$
  as in Remark \ref{rem:weyl action}, and take the complex structure $J$ defined by 
\begin{equation}\label{standard J}
J e_j =f_j  \qquad \mbox{and} \qquad  J f_j = - e_j .
\end{equation}
Under the identification   $Y \iso \R^d$  of Remark \ref{rem:weyl action},
the vacuum vector $\varphi^\circ = \varphi_J^\circ \in S(Y)$  of Definition \ref{def:vacuum} is identified with  the Gaussian
\begin{equation}\label{coordinate vacuum}
\varphi^\circ(y_1 ,\ldots, y_d ) = e^{- \pi   ( y_1^2 + \cdots + y_d^2)  } .
\end{equation}
Abbreviate $K=K_J$ for the centralizer of $J$ in $\Sp(W)$, and $\chi=\chi_J$ for the character  \eqref{chi_J}.

\begin{remark}
Our basis identifies  $W = \R^{2d}$, and so  identifies   $\Sp(W)$ with the matrix group
\[
\Sp_{2d}(\R) =  \left\{ 
g \in \GL_{2d}(\R) :  {}^t g  \begin{pmatrix}  & I_d \\  - I_d \end{pmatrix}  g = \begin{pmatrix}  & I_d \\  - I_d \end{pmatrix} 
\right\} .
\]
Depending on one's preference, one can  regard elements of $W$ as column vectors, in which case $g \in \Sp_{2d}(\R)$ corresponds to the symplectic automorphism $g(w) = g\cdot w$, or regard elements of $W$ as row vectors, in which case $g \in \Sp_{2d}(\R)$ corresponds to the symplectic automorphism $g(w) = w\cdot {}^t g$.  (In both cases the $\cdot$ is matrix multiplication).
Under either convention, the stabilizer  $K \subset \Sp(W)$ of $J$ is identified with the maximal compact subgroup
 \begin{equation}\label{standard K}
K = \left\{  \begin{pmatrix} A & B \\ -B & A \end{pmatrix} \in  \GL_{2d}(\R)   : A+iB \in \mathrm{U}(d)    \right\}  ,
\end{equation}
where $\mathrm{U}(d) = \{ g \in  \GL_d(\C) : \overline{g} = {}^t g ^{-1}  \}$ is the usual unitary group.
The  character  $\chi$ is given by 
 \begin{equation}\label{det chi}
\chi  \left( \begin{pmatrix} A & B \\ -B & A \end{pmatrix}  \right) = \det(A+iB). 
\end{equation}
These explicit descriptions will be needed later, but not immediately.
\end{remark}

Let $\mathfrak{k} \subset \mathfrak{sp}(W)$ be the real Lie algebra of $K$, and  let $\chi^\mathrm{inf} : \mathfrak{k} \to i \R$ be the  linear map defined by differentiating $\chi$, as in \eqref{symplectic trace}.
Denote in the same way its $\C$-linear extension to $\mathfrak{k}_\C \to \C$.

Using the compatibility (Theorem  \ref{thm:inf compatibility}) of the differentiated action of $\omega_\psi$ with the infinitesimal Weil representation $\omega^\mathrm{inf}$ of Definition \ref{def:Lie Schrodinger},  Proposition \ref{prop:vacuum eigenvector} 
 is equivalent to the equality 
\begin{equation}\label{Zeigenvacuum}
\omega^\mathrm{inf} ( Z ) \varphi^\circ = \frac{1}{2}  \cdot \chi^\mathrm{inf} ( Z ) \cdot \varphi^\circ 
\end{equation}
for all $Z \in \mathfrak{k}$.
This is what we will actually verify.

The path to doing this is clear.  We will  choose a  finite set of vectors that span $\mathfrak{k}_\C$.  
Each of our  chosen vectors $Z \in \mathfrak{k}_\C$  determines a differential operator $\omega^\mathrm{inf}(Z)$ on $S(\R^d)$, and a scalar $\chi^\mathrm{inf}(Z)$.  Once we have computed these for each $Z$ in our finite spanning set, \eqref{Zeigenvacuum} becomes a concrete system of  differential equations, and checking that $\varphi^\circ$ satisfies them is a calculus exercise.

\begin{lemma}\label{lem:basic Z_jk action}
As  $j,k \in \{1,\ldots, d\}$ vary, the vectors 
\begin{equation}\label{KZs}
Z_{jk} \define  (e_j  -  i  f_j  )  \cdot ( e_k  +  i f_k   )  \in \Sym^2(W_\C) 
\stackrel{ \eqref{sym2sp}}{\iso}\mathfrak{sp}(W_\C) 
\end{equation}
span the subalgebra $\mathfrak{k}_\C \subset  \mathfrak{sp}(W_\C)$, and satisfy
\[
\chi^\mathrm{inf} (Z_{jk})   = 2 i \delta_{jk} .
\]
\end{lemma}

\begin{proof}
Extend $J$ to a $\C$-linear automorphism of $W_\C$, and decompose
\[
W_\C  =  W_\C^{J=-i}  \oplus W_\C^{J = i } 
  \]
into its $J$-eigenspaces. 
 The two summands  are spanned, as $\ell \in \{1,\ldots, d\}$ varies,  by the vectors $e_\ell + i f_\ell$ and $e_\ell - i f_\ell$, respectively.   

Our maximal compact subgroup $K \subset \Sp(W)$ is precisely the subgroup of elements that preserve the direct sum decomposition, and the character $\chi$ is the composition
\[
K \to \GL( W_\C^{J = i } ) \map{ \det^{-1} } \C^\times.
\]
Differentiating this, the linear functional $\chi^\mathrm{inf}$ is 
\[
\mathfrak{k}_\C \iso \mathfrak{gl}( W_\C^{J = i }  ) \map{  - \mathrm{Tr} } \C .
\]

The lemma  follows   by directly calculating the action of each $Z_{jk}$ on $W_\C$, which is by
\begin{align*}
Z_{jk} \cdot (  e_\ell + i f_\ell ) & =    2 i \delta_{j \ell }  \cdot (  e_k + i f_k ) \\ 
Z_{jk} \cdot (  e_\ell - i f_\ell ) & =  -2 i \delta_{k \ell }  \cdot (  e_j - i f_j )  .
\end{align*}
From these formulas it is clear that $Z_{jk}$ preserves the  decomposition of $W_\C$, so lies in $\mathfrak{k}_\C$, and acts on the summand $W_\C^{J=i}$ with trace $-2 i \delta_{jk}$.
\end{proof}

\begin{lemma}\label{lem:basic Kinf calc}
Under the infinitesimal Weil representation, the action of  \eqref{KZs}  on $S(\R^d) \iso S(Y)$ is as follows:
If $j \neq k$,  then 
\[
\omega^\mathrm{inf}( Z_{jk} )   = 
  2\pi i   \left(    y_j  +   \frac{1}{2 \pi}   \frac{\partial}{ \partial y_j }   \right)
 \left(   y_k -   \frac{1}{2\pi}  \frac{\partial}{ \partial y_k}   \right)   .
\]
If $j=k$,  we instead have 
\[
\omega^\mathrm{inf}( Z_{kk} )   = 
2 \pi i  y_k^2 + \frac{1}{ 2 \pi i }      \frac{\partial^2 }{ \partial y_k^2 }    .
\]
\end{lemma}

\begin{proof}
We   compute the action of each $Z_{jk}$ on $S(\R^d)$  by first  applying  the injection  \eqref{Sym2Weyl}  to obtain  $\iota(Z_{jk}) \in \mathcal{W}$, and then letting this element act on $S(\R^d) \iso S(Y)$ via the map $\rho$ of  \eqref{Weyl module}.

 If $j \neq k$,   then
 \[
\iota( Z_{jk} ) =   \frac{1}{ 2 \pi i } \cdot (e_j  -  i  f_j  )  \otimes  ( e_k  +  i f_k   )  \in \mathcal{W} .
\]
If  $j=k$, we instead  rewrite \eqref{KZs} as  $Z_{kk} = e_k^2 + f_k^2 \in \Sym^2(W)_\C$, and find 
\[
\iota( Z_{k k} ) =   \frac{1}{ 2 \pi i } \cdot ( e_k \otimes e_k  + f_k \otimes f_k   )  \in \mathcal{W}.
\]
In either case,  apply $\rho$ to both sides and  use   Remark \ref{rem:weyl action}.
\end{proof}

\begin{proof}[Proof of Proposition \ref{prop:vacuum eigenvector}]
Using  Lemma \ref{lem:basic Kinf calc},  a calculus exercise shows that the Gaussian \eqref{coordinate vacuum} satisfies 
\[
\omega^\mathrm{inf}( Z_{jk} ) \varphi^\circ = i \delta_{jk} \cdot  \varphi^\circ .
\]
On the other hand, Lemma \ref{lem:basic Z_jk action} implies 
\[
\frac{1}{2}  \cdot \chi^\mathrm{inf} ( Z_{jk} ) \cdot \varphi^\circ =
 i \delta_{jk} \cdot  \varphi^\circ  .
\] 
Varying $j$ and $k$  proves that \eqref{Zeigenvacuum} holds for all $Z\in \mathfrak{k}_\C$,   and we have already explained why this implies Proposition \ref{prop:vacuum eigenvector}.
\end{proof}


\subsection{The Fock model}


We end \S \ref{s:infinitesimal weil}  with a few words about the Fock model of the infinitesimal Weil representation. 
Although we will not use it directly, it  plays a central role in the calculations of Kudla-Millson \cite{KM4} and Funke-Millson \cite{FunkeMillson}.

Return to the setting of \S \ref{ss:vacuum vector},  fix a complex structure $J \in \Sp(W)$ of the type considered there, and decompose 
\begin{equation}\label{J-decomp}
W_\C  =  W_\C^{J=-i}  \oplus W_\C^{J = i }  
\end{equation}
as in the proof of Lemma \ref{lem:basic Z_jk action}.

\begin{definition}
Let $\mathcal{J} \subset \mathcal{W}$ be the left ideal generated by the image of $W_\C^{J = i }$ under the composition $W_\C \to W_\C^\otimes \to \mathcal{W}$.  
The \emph{Fock model} of the infinitesimal Weil representation is the left $\mathcal{W}$-module 
\[
M_J^\mathrm{Fock} = \mathcal{W}/ \mathcal{J},
\]
regarded as a representation of $\mathfrak{sp}(W)$ using \eqref{quantum iota}.
\end{definition}

Recalling the action \eqref{Weyl module},  one can show that the vacuum vector  in $S(Y)$ satisfies
$
\rho( \mathcal{J} ) \varphi_J^\circ =0,
$
and that  this relation determines the vacuum vector uniquely up to scaling.
 It follows  that there is a unique $\mathcal{W}$-module map 
 \[
M_J^\mathrm{Fock} \to S(Y)
\]
sending $1\mapsto \varphi_J^\circ$. As in \cite[Appendix A]{FunkeMillson}, this map is injective with image 
\begin{equation}\label{K-finite}
 \{ P\cdot \varphi^\circ_J : P \mbox{ is a polynomial function on } Y \}  .
\end{equation}
More or less tautologically,  it intertwines the action of $\mathfrak{sp}(W)$ on the Fock model  with its action on $S(Y)$ via the infinitesimal Weil representation $\omega^\mathrm{inf}$.

\begin{remark}
One usually identifies the Fock model with a space of polynomials, but \emph{not}  by identifying it with \eqref{K-finite} and dividing by the vacuum vector.
Each summand in \eqref{J-decomp} is totally isotropic with respect to the symplectic form on $W_\C$, and it follows that the composition 
\[
 W_\C^{J=-i} \hookrightarrow W_\C \to \mathcal{W}
 \]
  extends to a $\C$-algebra homomorphism from the full symmetric algebra 
$\Sym(  W_\C^{J=-i})$ to the Weyl algebra.
The composition 
\[
\Sym(  W_\C^{J=-i}) \to \mathcal{W} \to  \mathcal{W} / \mathcal{J} = M_J^\mathrm{Fock}
\]
is a bijection, and one uses this to identify   $M_J^\mathrm{Fock}$ with the space of polynomial functions on the $\C$-linear dual of $W_\C^{J=-i}$.  
\end{remark}


\section{Symplectic-orthogonal dual pairs}
\label{s:symplectic-orthogonal}


In this section we use the theory of symplectic-orthogonal reductive dual pairs in the sense of Howe \cite{Howe-theta} to embed a smaller metaplectic group into a larger one.   Using the restriction of   the Weil representation to the embedded subgroup, we  construct theta functions on the smaller adelic metaplectic group, and then convert them to classical theta functions on the Siegel half-space.


\subsection{The setup}
\label{ss:so setup}


Throughout \S \ref{s:symplectic-orthogonal} we work over  a global field $F$ of  characteristic $\mathrm{char}(F) \neq 2$, and abbreviate  $\A=\A_F$ for its ring of adeles.  
Fix  an additive character  $\psi : F \backslash \A \to \C^1$.

 Let $V$ be a finite dimensional vector space over $F$, endowed with a nondegenerate  symmetric bilinear form 
  $
  b:V\times V \to F . 
  $   
For a $d$-tuple $v=(v_1,\ldots, v_d) \in V^d$, write
\[
Q(v) = \left( \frac{ b(v_j,v_k)  }{2}  \right)_{ 1 \le j,k \le d } \in \Sym_d(F)
\]
for one-half of the matrix of  inner products of the components of $v$.

Let $W$ be a vector space over $F$ of dimension $2d$, endowed with a symplectic form $\langle - , -\rangle : W \times W \to F$.
Choose a basis $e_1,\ldots, e_d, f_1,\ldots, f_d \in W$ in such a way that $\langle e_j,f_k\rangle = \delta_{jk}$, and the subspaces 
\[
X=\mathrm{Span} \{ e_1,\ldots, e_d \}   \qquad \mbox{and} \qquad Y=\mathrm{Span} \{ f_1,\ldots, f_d \} 
\]
are totally isotropic.  We identify  $\Sp(W)$ with the  matrix group 
  \[
\Sp_{2d}(F) = \left\{ 
g \in \GL_{2d}(F) :  {}^t g  \begin{pmatrix}  & I_d \\  - I_d \end{pmatrix}  g = \begin{pmatrix}  & I_d \\  - I_d \end{pmatrix} 
\right\} 
\]
in the following way:
Use the chosen basis of $W \iso F^{2d}$ to write its elements as \emph{row}  vectors, and identify  a matrix $g\in \Sp_{2d} (F)$ with the symplectic automorphism  $g(w) = w \cdot {}^tg$ of $W$ (the  $\cdot$ on the right is usual matrix multiplication).

Combine $V$ and $W$ into a new   polarized symplectic space,  by endowing $\bm{W} = W \otimes V$
with the  symplectic form  
\[
\langle w_1 \otimes v_1 , w_2 \otimes v_2 \rangle_{\bm{W}}  
= \langle w_1,w_2 \rangle  \cdot b(v_1,v_2) 
\]
 and  the polarization $\bm{W} = \bm{X} \oplus \bm{Y}$ with  $ \bm{X} = X \otimes V$ and  $\bm{Y} = Y \otimes V$.  Identify 
 \[
 V^d \iso  \bm{Y} 
 \]
by  $(v_1,\ldots, v_d ) \mapsto f_1 \otimes v_1 + \cdots + f_d \otimes v_d$.
  There is a natural map 
 $\Sp(W) \times  \mathrm{O}(V)  \to\Sp(\bm{W})$,
and in particular an injective homomorphism
\begin{equation}\label{symplectic i_V}
i_V : \Sp(W) \to \Sp(\bm{W}) .
\end{equation}


\subsection{Lifting the embedding}

 
 Following Kudla \cite{KudlaSplitting}, but using the notation and conventions of \cite{ThetaBook}, 
we will upgrade    \eqref{symplectic i_V}  to a homomorphism  between metaplectic groups.
We will then  restrict  the Weil representation of the larger group to a representation of the smaller one, and use this to construct theta functions on the smaller metaplectic group. 
The first step is described in this subsection, and the second is described in the next.

For any place $v$ of $F$, we have 
\begin{itemize}
\item the Leray cocycle  $\bm{c}_{\psi,v} : \Sp(\bm{W}_{F_v}) \times \Sp(\bm{W}_{F_v}) \to \C^1$ from \eqref{leray cocycle}
\item
 the Rao cocycle $c_{\psi,v}^{(2)} : \Sp(W_{F_v}) \times \Sp(W_{F_v})  \to \mu_2$ from \eqref{rao cocycle}.
\end{itemize}
To relate them,   Kudla \cite{KudlaSplitting}  writes  down an explicit function 
 \begin{equation*}
 \beta_v=  \beta_{V,\psi , v }  : \Sp(W_{F_v} ) \to \mu_8 
 \end{equation*}
 in such a way that 
 \begin{equation}\label{orthogonal-symplectic splitting cocycle}
 \bm{c}_{\psi,v}   \big(  i_V(g_1)  ,  i_V(g_2) \big)  
  \cdot \frac{  \beta_v (g_1)  \beta_v (g_2) }{  \beta_v ( g_1g_2) } =
 c_{\psi,v} ^{(2)} ( g_1,g_2)^{ \dim(V) } 
 \end{equation}
 for all $g_1,g_2 \in \Sp(W_{F_v})$.
 See Case $1_{+}$ of   \cite[(11.3)]{ThetaBook} for the precise definition.
 
 \begin{remark}\label{rem:ThetaBook caveat}
Although \cite{ThetaBook} works over a  nonarchimedean local field,  the definition of $\beta_v$ given there makes sense for all  places $v$, and is the  correct one in both the archimedean  and  nonarchimedean  cases.
  Indeed, the original source \cite{KudlaSplitting} treats the archimedean and nonarchimedean cases simultaneously, and the formulas found there are agnostic to the distinction.  Similar remarks apply to \cite{Rao}.  We will generally cite  \cite{ThetaBook} for explicit local formulas, even at archimedean places,  because we follow its conventions and normalizations, which are different from those of  \cite{KudlaSplitting} and \cite{Rao}.
   \end{remark}

There is a commutative diagram
  \begin{equation}\label{symplectic-orthogonal local splitting diagram}
\begin{tikzcd}
{ \Sp(W_{F_v})  \times \mu_2  }  \ar[r]   \ar[d,  " \eqref{rao coords} "  ' ]     &  {  \Sp(W_{F_v})  \times \C^1 }  \ar[d,  " \eqref{leray coords} " ]    \\ 
 {  \Mp(W_{F_v})_\psi^{(2)} }  \ar[d] \ar[r, " i_{V,\psi} "]&  {  \Mp(\bm{W}_{F_v})_\psi  }  \ar[d] \\
{   \Sp(W_{F_v}) }  \ar[r , "i_V" ]   &  {   \Sp(\bm{W}_{F_v}) , } 
\end{tikzcd}
\end{equation}
in which the top horizontal arrow is 
\[
(g,\epsilon) \mapsto (  i_V( g ) , \epsilon^{\dim(V)} \beta_v(g) ) . 
\]
The arrow labeled  $i_{V,\psi}$ is defined by the commutativity of the diagram,  and  \eqref{orthogonal-symplectic splitting cocycle}  is precisely the condition that makes it a group homomorphism.

 These local homomorphisms can be assembled, as $v$ varies, into a homomorphism
 \begin{equation}\label{orthogonal-symplectic adelic lift}
i_{V,\psi} :  \Mp( W_\A )_\psi^{(2)}   \to   \Mp(\bm{W}_\A)_\psi ,
 \end{equation}
 but this requires some care.
 As  pointed out in  Remarks \ref{rem:adelic leray coordinates} and \ref{rem:adelic rao coordinates}, we have not chosen  any explicit coordinates on the source or target of this map, and so we do not have names for most of their elements.
 This forces us to characterize \eqref{orthogonal-symplectic adelic lift} in a slightly indirect way: by describing what it does to a dense subset of elements of the domain for which  we do have names.

 \begin{proposition} \label{prop:so adelization}
 There is a unique continuous homomorphism \eqref{orthogonal-symplectic adelic lift}  such that, for every finite set of places $\Sigma$ of $F$, the diagram 
 \[
 \begin{tikzcd}
 { \left( \prod_{v\in\Sigma} \Sp(W_{F_v})  \right) \times \mu_2 }    \ar[d, " \eqref{semi rao coordinates}"   ' ]   \ar[rr]      &  &  { \left( \prod_{v\in\Sigma} \Sp( \bm{W}_{F_v})  \right) \times \C^1 }  \ar[d , " \eqref{semi leray coordinates} " ]  \\ 
 { \Mp( W_\A )_\psi^{(2)}  } \ar[rr ,  " i_{V,\psi} "  ' ]    &   & {    \Mp(\bm{W}_\A)_\psi   }
 \end{tikzcd}
 \]
 commutes.  Here the top horizontal arrow is defined by 
 \[
 ( g  , \epsilon ) \mapsto \Big(  i_V(g) , \epsilon^{\dim(V) }\prod_{v\in \Sigma} \beta_v (g_v)  \Big) .
 \]
 Moreover,  for this homomorphism we have a commutative diagram
 \[
 \begin{tikzcd}
 {  \Sp(W) }    \ar[d, " \eqref{double splitting}"   ' ]   \ar[r , " i_V" ]        &  { \Sp( \bm{W}  )  }  \ar[d , " \eqref{canonical splitting}  " ]  \\ 
 { \Mp( W_\A )_\psi^{(2)}  } \ar[r ,  " i_{V,\psi} "  ' ]    &    {    \Mp(\bm{W}_\A)_\psi   .}
 \end{tikzcd}
 \]
 If $\dim(V)$ is even, then \eqref{orthogonal-symplectic adelic lift} factors through a continuous homomorphism
 \[
 \Sp( W_\A )  \to   \Mp(\bm{W}_\A)_\psi  .
 \]
 \end{proposition}

Proposition \ref{prop:so adelization} is a reformulation of \cite[Proposition 2.4.2]{Sweet}.
Sweet works systematically with the  alternate coordinates described in  \S \ref{ss:sweet coordinates}, and we have  translated his  statement into our preferred (partial) coordinates \eqref{semi leray coordinates} and \eqref{semi rao coordinates}. 
 Sweet's formulation makes it clear that the map \eqref{orthogonal-symplectic adelic lift} actually takes values in the subgroup 
  $\Mp(\bm{W}_\A)_\psi^{(2)}$, but we will never use this.


\subsection{Restriction of the Weil representation}

 
The metaplectic group $\Mp(\bm{W}_\A)_\psi$ acts on the Schwartz space $S(V^d_\A)= S(\bm{Y}_\A)$ via   the  Weil representation  $\omega_\psi$ of Definition \ref{def:adelic metaplectic}.
Pulling back   the Weil representation  along  \eqref{orthogonal-symplectic adelic lift},  we obtain a representation of $\Mp(W_\A)^{(2)}_\psi$ on $S(V^d_\A)$.   Let us record  this construction as a definition.

\begin{definition}\label{def:symplectic-orthogonal weil}
The \emph{Weil representation}  $\omega_{V,\psi}$  is the composition 
\[
\omega_{V,\psi}  \define  \omega_\psi  \circ  i_{V,\psi} : \Mp(W_\A)_\psi^{(2)} \to  \Aut \big(  S( V_\A^d)  \big) .
\]
The \emph{adelic theta function}  associated to a  Schwartz function $\varphi   \in S(V_\A^d)$ is the continuous function 
\begin{equation}\label{adelic siegel theta}
\theta_{V,\psi}( g ,  \varphi)  =  \sum_{ v \in V^d }  \left(  \omega_{V,\psi}(g) \varphi  \right) (v)  
\end{equation}
of the variable  $g \in \Sp(W) \backslash \Mp(W_\A)_\psi^{(2)}$.
\end{definition}

\begin{remark}
The convergence and left  $\Sp(W)$-invariance of \eqref{adelic siegel theta} are immediate from  Theorem \ref{thm:big theta}.
The point is that the function $\Theta_\psi(   -   , \varphi   )$ on $\Sp( \bm{W} )  \backslash \Mp(\bm{W}_\A)_\psi$ from that theorem  satisfies
\[
\theta_{V,\psi}( g ,  \varphi)  =  \Theta_\psi(    i_{V,\psi} (g)    , \varphi   ) .
\]
\end{remark}

We now make the Weil representation $\omega_{V,\psi}$  more explicit.  
Suppose   $\Sigma$ is any finite set of places of $F$.
If $g \in \Sp(W_\A)$ is an adelic point such that $g_v=1$ for all $v$ outside $\Sigma$,  and $\epsilon \in \mu_2$, the partial Rao coordinates  \eqref{semi rao coordinates} determine a point 
$
(g , \epsilon ) \in \Mp(W_\A)_\psi^{(2)} .
$
Unpacking the definitions, we find that
 \begin{equation}\label{orthogonal symplectic general weil}
\omega_{V,\psi} (g , \epsilon ) =  \epsilon^{\dim(V)} \cdot   \beta_\Sigma   (g)  \cdot 
\bm{r}_{\psi , \Sigma}  ( i_V(g) ),
\end{equation}
where  
$
\bm{r}_{\psi , \Sigma} : \Sp(\bm{W}_\A)  \to \Aut ( S(\bm{Y}_\A) ) 
$
is the operator   \eqref{Sigma operator}, and 
\[
\beta_\Sigma(g) = \prod_{v \in \Sigma}  \beta_v   (g_v) \in \mu_8.
\]

\begin{proposition}\label{prop:symplectic-orthogonal weil formulas}
Fix $A \in \GL_d(\A)$ and $B \in \Sym_d(\A)$, and define elements of $\Sp_{2d}(\A) \iso \Sp(W_\A)$ by 
\[
m(A)  = \begin{pmatrix}  A  &  \\  & {}^t A ^{-1}  \end{pmatrix} 
\qquad \mbox{and} \qquad
n(B) = \begin{pmatrix}   I_d  & B  \\ &  I_d  \end{pmatrix} .
\]
Assume that the local components of both $A$ and $B$ are trivial at all places outside some finite set of places $\Sigma$ of $F$, so that we may use the partial  Rao coordinates  \eqref{semi rao coordinates} to form 
\[
( m(A) , \epsilon ) \in \Mp(W_\A)^{(2)}_\psi \quad \mbox{and}\quad ( n(B) , \epsilon ) \in \Mp(W_\A)^{(2)}_\psi 
\]
for any $\epsilon \in \mu_2$.    Under the Weil representation on $S(V_\A^d)$, these satisfy
\[
\omega_{V,\psi} (m(A), \epsilon )   \varphi (v) 
=   \epsilon^{\dim(V)}  \cdot   \beta_\Sigma   ( m(A) )  
\cdot   |  \det(A)  |^{ \frac{1}{2} \dim_F(V) }   \cdot  \varphi(   v  A ) 
\]
and 
\[
\omega_{V,\psi} (n(B) , \epsilon) \varphi (v) 
=  \epsilon^{\dim(V)}  \cdot    \psi(   \mathrm{Tr}( Q(v)  B )  ) \cdot  \varphi(  v )  . 
\]
In the first formula the tuple $v A \in V_\A^d$ is computed by viewing $v \in V_\A^d$ as a row vector, and multiplying it by the matrix $A$ in the usual way. 
The modulus function $| \cdot |$ on $\A^\times$  is normalized by 
\[
\int_\A f(ax) \, dx = |a| ^{-1}  \int_\A f(x)\, dx.
\]
\end{proposition}

\begin{proof}
Write $g \in \Sp(W_\A)$ for the symplectic automorphism corresponding to the matrix $m(A)$, under the convention explained in \S \ref{ss:so setup}.
  As in the definition of \eqref{r_XY}, its  inverse $g^{-1}$ is encoded  by  linear maps $a : X_\A \to X_\A$ and $d : Y_\A\to Y_\A$. 
Using the bases $e_1,\ldots , e_d$ and $f_1,\ldots, f_d$, we identify elements of $X_\A \iso \A^d$ and $Y_\A \iso \A^d$ as row vectors, so that $a$ and $d$ are expressed in terms of matrix multiplication by
\begin{align*}
a(x) & = g^{-1}(x)  = x\cdot {}^t A^{-1}  \\
d(y) & = g^{-1} (y) = y \cdot A.
\end{align*}

Now consider  $i_V(g) \in \Sp(\bm{W}_\A)$. Its inverse is encoded by the linear maps
\begin{align*}
  \bm{X} _\A =  X_\A \otimes V_\A   &  \map{  a \otimes \mathrm{id}  }    X_\A \otimes V_\A  \iso \bm{X}_\A  \\
    \bm{Y} _\A =  Y_\A \otimes V_\A   &  \map{  d \otimes \mathrm{id}  }    Y_\A \otimes V_\A  \iso \bm{Y}_\A  .
  \end{align*} 

Tracing the definition of $\bm{r}_{\psi , \Sigma}$ all the way back to the integral  \eqref{big weil}, we find that for any $y \otimes v \in  Y_\A \otimes V_\A$ we have
\[
\bm{r}_{ \psi , \Sigma } (  i_V(g)  )  \varphi     ( y \otimes v  ) = 
 |  \det(A)  |^{ \frac{1}{2} \dim_F(V) } 
 \varphi\big(   (y\cdot A) \otimes v \big) .
 \]
The normalizing factor $ |  \det(A)  |^{ \frac{1}{2} \dim_F(V) }$ is  included to make the operator unitary, as per Proposition \ref{prop:rao integral}.
 When we identify a tuple  
 $v \in V_\A^d = Y_\A \otimes V_\A$ with a row of vectors in $V_\A$, this becomes 
 \[
\bm{r}_{ \psi , \Sigma } (  i_V(g)  )   \varphi     ( v ) = 
 |  \det(A)  |^{ \frac{1}{2} \dim_F(V) }  \cdot 
 \varphi\big(   v  A \big) .
 \]
Now use \eqref{orthogonal symplectic general weil} obtain the formula for the action of $( m(A) , \epsilon)$.

The proof for $(n(B),\epsilon)$ is similar.  
For any place $v$, the restriction of Kudla's function  $\beta_v$  to the unipotent radical of the Siegel parabolic is identically $1$, which is why it does not appear in the formula for $\omega_{V,\psi} (n(B) , \epsilon)$.
\end{proof}

\begin{remark}\label{rem:beta on parabolics}
Let $v$ be an archimedean place of $F$,  let $P_v \subset \Sp(W_{F_v})$ be the Siegel parabolic stabilizing $X_{F_v} \subset W_{F_v}$, and denote by $P_v^+ \subset P_v$ the connected component of the identity.  
 Kudla's function $\beta_v$  satisfies $\beta_v(p) =1$ for all $p \in P_v^+$.
 For  a real place $v$, this implies    $\beta_v( m(A) )=1$ whenever  $\det(A)>0$.
\end{remark}

\begin{remark}
If $\dim(V)=2r$ is even, then 
\[
\beta_v( m(A_v) ) = \big( \det(A_v) , (-1)^{ r}  \det(V)  \big)_v,
\]
where the right hand side is the Hilbert symbol, and $\det(V)$ is the determinant of the matrix of inner products $( b(x_j , x_k) )_{j,k} \in \Sym_{2r}(F)$ of a basis $x_1,\ldots, x_{2r} \in V$. 
When $\dim(V)$ is odd the formula for $\beta_v(m(A_v))$ involves a Weil index, and we refer the reader to 
Case $1_{+}$ of   \cite[(11.3)]{ThetaBook}.
\end{remark}

\begin{remark}
For an arbitrary $g\in \Sp(W_\A)$, the  product $\prod_v \beta_v(g)$ over all places is not defined, as there will typically be infinitely many factors different from $1$.  
However,  for any $A \in \GL_d(\A)$ the product 
$\beta( m(A))  =\prod_v \beta_v( m(A_v))$ is actually finite.
If we assume that $\dim(V)$ is even, then the Weil representation $\omega_{V,\psi}$ factors through $\Sp(W_\A)$, and satisfies the expected formulas 
\[
\omega_{V,\psi} (m(A)  )   \varphi (v) 
=     \beta    ( m(A) )  
\cdot   |  \det(A)  |^{ \frac{1}{2} \dim_F(V) }   \cdot  \varphi(   v  A ) 
\]
and 
\[
\omega_{V,\psi} (n(B)  ) \varphi (v) 
=    \psi(   \mathrm{Tr}( Q(v)  B )  ) \cdot  \varphi(  v )   
\]
for arbitrary  $A\in \GL_d(\A)$ and $B\in \Sym_d(\A)$.
\end{remark}


\subsection{Generalities on Siegel modular forms}


For the remainder of \S \ref{s:symplectic-orthogonal} we work over the global field $F=\Q$.  
Our goal in this subsection  is  to recall the general  process of de-adelizing    automorphic forms on  $\Sp(W_\A)$  to functions  on the Siegel half-space.
For more substantial expositions of the theory of Siegel modular forms, see \cite{vandergeer} and \cite{pitale}.

The matrix group $\Sp_{2d}(\R)$ acts transitively on the Siegel half-space
\[
\mathcal{H}_d 
= \left\{ \mathtt{z} =  \mathtt{x} + i \mathtt{y}  \in M_d(\C) :  
\begin{array}{c}    
\mathtt{x}, \mathtt{y} \in \Sym_d(\R) \\
 \mathtt{y} \mbox{ is positive definite}    \end{array} \right\}
\]
according to the usual rule: if 
\begin{equation*}
\gamma = \begin{pmatrix}  A & B \\ C & D  \end{pmatrix}  \in \Sp_{2d}(\R),
\end{equation*}
then  $\gamma \mathtt{z} = (A\mathtt{z}+B)(C\mathtt{z}+D)^{-1}.$  
The function 
\begin{equation}\label{integral j}
j( \gamma , \mathtt{z}) =  \det( C\mathtt{z} + D ) 
\end{equation}
of $\gamma \in \Sp_{2d}(\R)$ and $\mathtt{z} \in \mathcal{H}_d$ satisfies the familiar  cocycle relation
\begin{equation}\label{siegel automorphy}
 j( \gamma_1 \gamma_2 , \mathtt{z} )  = j( \gamma_1 ,  \gamma_2 \mathtt{z} ) \cdot j( \gamma_2 , \mathtt{z} ) .
\end{equation}

The stabilizer of $i I_d \in \mathcal{H}_d$ is the maximal compact subgroup $K   \subset \Sp_{2d}(\R)$
already encountered in   \eqref{standard K}, which has a distinguished character $\chi  : K  \to \C^1$ defined by   \eqref{det chi}.
Fix an integer  $m$, and  suppose $F$ is a function on  $\Sp(W) \backslash \Sp(W_\A)$  satisfying 
\begin{equation}\label{siegel weight}
F(g k ) = \chi(k)^m F(g) 
\end{equation}
for all $g \in \Sp(W_\A)$ and $k \in K$.
Assume that $F$ is right invariant under  some compact open subgroup   $U \subset \Sp(W_{\A_f})$,  and define   
\begin{equation*}
\Gamma =     \{ \gamma \in \Sp_{2d}(\R) :  \exists u \in U \mbox{ for which }
\gamma u \in \Sp_{2d}(\Q) \}  .
\end{equation*}
On the right hand side  the product $\gamma u$ is  understood inside  $\Sp_{2d}(\A)$.

Convert $F$ to a function on the Siegel half-space as follows: 
Given a  $\mathtt{z} = \mathtt{x} + i \mathtt{y}  \in \mathcal{H}_d$,  choose a    factorization  
\begin{equation*}
\mathtt{y} =\alpha \cdot {}^t \alpha
\end{equation*}
in such a way that  $\alpha \in \GL_d(\R)$ has positive determinant,   set 
\begin{equation}\label{natural g_z}
g_\mathtt{z}  =  \begin{pmatrix} I_d & \mathtt{x} \\ & I_d \end{pmatrix} 
\begin{pmatrix} \alpha &  \\ & {}^t\alpha^{-1}  \end{pmatrix}    \in \Sp_{2d}(\R)  ,
\end{equation}
and define 
  \begin{equation}\label{classical siegel}
F(\mathtt{z}) =    \det(\mathtt{y})^{- m/2 } F( g_\mathtt{z} ) .
\end{equation}
On  the right hand side,  we are viewing 
 $g_\mathtt{z} \in \Sp_{2d}(\R)$ as an element of  $\Sp_{2d}(\A)$ with trivial nonarchimedean components.

 \begin{remark}\label{rem:alpha ambiguity}
 Making a different choice of $\alpha$ amounts to multiplying the $\alpha$ already chosen  on the right by an element of $\mathrm{SO}(d)$, which multiplies $g_\mathtt{z}$ on the right  by an element of 
 \[
 \left\{  m(A)=\begin{pmatrix} A &  \\   & A \end{pmatrix}   : A  \in \mathrm{SO}(d)    \right\} \subset  \ker( \chi : K \to \C^1).
\]
Using  the assumption \eqref{siegel weight}, it follows that  $F(\mathtt{z})$ is independent of the choice of $\alpha$ used to define it.
\end{remark}

\begin{proposition}\label{prop:siegel transformation}
The function \eqref{classical siegel}  satisfies  the transformation law
\[
F(\gamma \mathtt{z} ) =  j( \gamma ,\mathtt{z})^m F(\mathtt{z}) 
\]
 for all $\gamma \in  \Gamma \subset  \Sp_{2d}(\R)$.
\end{proposition}

\begin{proof}
This is well-known \cite[\S 6.1]{pitale}, but we formulate the proof in a way that  points  toward the  correct definition  of  $j(\gamma,\mathtt{z})^m$ when $m$ is allowed to be a half-integer.

Under  the action of $\Sp_{2d}(\R)$ on $\mathcal{H}_d$,  we have   $g_\mathtt{z}\cdot i I_d = \mathtt{z}$.
Hence  $g_\mathtt{z}^{-1}  \gamma^{-1}  g _{\gamma \mathtt{z} }  \in K$ for all  $\gamma \in \Sp_{2d}(\R)$,  allowing us to define 
\begin{equation}\label{circle cocycle}
c( \gamma , \mathtt{z} ) =   \chi ( g_\mathtt{z}^{-1}  \gamma^{-1}  g _{\gamma \mathtt{z} }  )   .
\end{equation}

Now  suppose  $\gamma \in \Gamma$.
  By definition of $\Gamma$, there is a $u  \in U$ such that $\gamma  u $ lies in the image of the  natural embedding $\Sp(W) \to \Sp(W_\A)$.  
Because the function $F(g)$  is left invariant under this image, and right invariant under $U$, we have 
\[
F( g_\mathtt{z} )  =  F( \gamma   u  \cdot g_\mathtt{z} ) = F( \gamma  \cdot  g_\mathtt{z}  \cdot  u ) =  F( \gamma  g_\mathtt{z}  ) .
\]
On the other hand, because $g_\mathtt{z}^{-1}  \gamma^{-1}  g _{\gamma \mathtt{z} }  \in K$, the assumption \eqref{siegel weight} implies 
\[
c(\gamma, \mathtt{z})^m  \cdot  F( \gamma  g_\mathtt{z}  ) 
= \chi (  g_\mathtt{z}^{-1}  \gamma^{-1}  g _{\gamma \mathtt{z} }  )^m   \cdot  F( \gamma  g_\mathtt{z}  ) 
= F( g _{\gamma \mathtt{z} }  ) .
\]
This proves  the equality
$
F( g_{ \gamma \mathtt{z} } ) = c(\gamma,\mathtt{z})^{m} F(g_\mathtt{z}). 
$

The proposition now follows from 
\[
 j( \gamma , \mathtt{z} ) 
= 
c( \gamma , \mathtt{z} ) 
 \cdot 
 \left(   \frac{\det  \mathrm{Im}  ( \mathtt {z} )     }  {  \det  \mathrm{Im}(  \gamma \mathtt{z} )   } \right)^{1/2} .
\]
One verifies this last equality by direct calculation  for   $\mathtt{z} = i I_d$,  then  deduces  the general case  using the cocycle relation \eqref{siegel automorphy}, which is also satisfied  by   $c(\gamma, \mathtt{z})$. The details are left to the reader.
\end{proof}


\subsection{Siegel modular forms of half-integral weight}


We now extend the discussion of the previous subsection to Siegel modular forms of half-integral weight.
The challenge, of course, is to specify  square roots of the factors of automorphy $j( \gamma, \mathtt{z})$ from  \eqref{integral j} in some consistent way.   One approach to this is  \cite[Theorem 8.3]{Mumford-3}, but we will effectively define our square roots in such a way that the proof of Proposition \ref{prop:siegel transformation} goes through unchanged.

Identify $\Sp(W_\R) = \Sp_{2d}(\R)$ in the way explained in \S \ref{ss:so setup}.
Denote by $K^{(2)} \subset \Mp(W_\R)_\psi^{(2)}$ the preimage of $K$ under the metaplectic double cover 
\[
\Mp(W_\R)^{(2)}_\psi \map{ \eqref{double inclusion} }  \Sp(W_\R) .
\]
It is known that $K^{(2)}$ is the unique connected double cover of $K\iso \mathrm{U}(d)$,  and that the pullback of the character $\chi$  to this cover has a unique square root, denoted $\chi^{1/2}$.  
In fact, one has  a cartesian diagram
\begin{equation*}
\begin{tikzcd}
{ K^{(2)}  }  \ar[d] \ar[r, " \chi^{ 1/2  } " ] &  { \C^1  } \ar[d , " z \mapsto z^2" ] \\
{ K } \ar[r , " \chi " ]  & { \C^1 .}  
\end{tikzcd} 
\end{equation*}
For proofs of these claims see \cite[\S 7.1]{Sweet}.

Now fix  $m\in 2^{-1}\Z$, and abbreviate
\begin{equation}\label{half-character}
\chi^{ m  } = ( \chi^{1/2}  )^{(2m)} : K^{(2)} \to \C^1 .
\end{equation}
Imitating \eqref{circle cocycle},   define a function of  $\gamma \in \Mp(W_\R)_\psi^{(2)}$ and $\mathtt{z} \in \mathcal{H}_d$ by 
\begin{equation}\label{half circle cocycle}
c(  \gamma , \mathtt{z} )^m  =
 \chi^{ m } \big(  \tilde{g}_\mathtt{z} ^{-1}  \cdot  \gamma^{-1}  \cdot  \tilde{g} _{\gamma \mathtt{z}} \big),
\end{equation}
where we have used the Rao coordinates \eqref{rao coords} to define
\begin{equation}\label{g_z lift}
\tilde{g}_{\mathtt{z}} = ( g_{\mathtt{z}} , 1 ) \in \Sp(W_\R) \times \mu_2    \iso  \Mp(W_\R)_\psi^{(2)}.
\end{equation}
Note that $\tilde{g}_\mathtt{z} ^{-1}  \cdot  \gamma^{-1}  \cdot  \tilde{g} _{\gamma \mathtt{z}} \in K^{(2)}$, as in the  proof of Proposition \ref{prop:siegel transformation}.

\begin{proposition}\label{prop:good half cocycle}
The function 
\[
j(  \gamma , \mathtt{z} )^m  \define
c(   \gamma   , \mathtt{z} )^m
 \cdot 
 \left(   \frac{\det( \mathrm{Im}(   \mathtt{z} ) )    }  {  \det( \mathrm{Im}(  \gamma \mathtt{z} ) )  } \right)^{m/2} 
\]
of $\gamma \in \Mp(W_\R)_\psi^{(2)}$ and $\mathtt{z} \in \mathcal{H}_d$  satisfies the following properties: 
\begin{enumerate}
\item
It is well-defined, meaning independent of the choice of  $\alpha$ used in the definition of \eqref{natural g_z}.  
\item
It  is holomorphic as a function of $\mathtt{z} \in \mathcal{H}_d$.
\item
It satisfies the cocycle relation
$
 j(  \gamma_1  \gamma_2 , \mathtt{z} )^m  = j( \gamma_1 ,  \gamma_2 \mathtt{z} )^m  \cdot j(   \gamma_2 , \mathtt{z} )^m.
$
\item
If $m\in \Z$, then $j(  \gamma , \mathtt{z})^m$ depends only on the image of $\gamma$ in  $\Sp(W_\R)$, and agrees with the   $j(\gamma , \mathtt{z})^m$ already defined by \eqref{integral j}.
\end{enumerate}
\end{proposition}

\begin{proof}
If we can prove that $j(\gamma ,\mathtt{z})^m$ is well-defined and continuous as a function of $\mathtt{z}$, then everything else follows easily.
The cocycle relation asserted in (3) follows from the analogous relation for \eqref{half circle cocycle}, which is clear from the definition.  
If $m$ is an integer,  the  character \eqref{half-character} factors through $K^{(2)} \to K$, and agrees with the $m^\mathrm{th}$-power 
of the character $\chi : K \to \C^1$.  Hence the function $c(\gamma,\mathtt{z})^m$  defined by \eqref{half circle cocycle} only depends on the image of $\gamma$ in $\Sp(W_\R)$, and agrees with the $m^\mathrm{th}$ power of  \eqref{circle cocycle}.
Property (4) follows immediately.  
From property  (4) we deduce that the square of $j( \gamma , \mathtt{z})^m$ is holomorphic.  
If  $j( \gamma , \mathtt{z})^m$ is continuous, it must  be holomorphic as well.

The key to proving that $j(\gamma ,\mathtt{z})^m$ is well-defined and continuous is the following lemma.

\begin{lemma}\label{lem:siegel section}
As in Remark \ref{rem:beta on parabolics},  let $P^+_\R \subset \Sp(W_\R)$ denote the identity component of the Siegel parabolic.
The composition
\[
P_\R^+ \map{  p \mapsto \tilde{p}   }  \Sp(W_\R) \times \mu_2   \stackrel{ \eqref{rao coords}}{\iso}  \Mp(W_\R)_\psi^{(2)} ,
\]
where we have set $\tilde{p} = ( p,1)$,  is a continuous group homomorphism. 
\end{lemma}

\begin{proof}
To say that $p\mapsto \tilde{p}$ is a group homomorphism means precisely that  the archimedean Rao cocycle satisfies $c_{\psi}^{(2)}( p_1 , p_2) =1$ for all $p_1,p_2 \in P^+_\R$.  
This follows from the explicit formula of \cite[Corollary 5.16(c)]{ThetaBook}.
See also Rao \cite[Corollary 5.5(2)]{Rao}, as per the caveat of  Remark \ref{rem:ThetaBook caveat}.

For the continuity claim, unwinding the definitions shows that  $\tilde{p}$ is another name for the pair 
\[
\big( p ,  r^{(2)}_\psi(p)  \big)  \in  \Mp(W_\R)_\psi^{(2)}  \subset \Sp(W_\R) \times \Aut (L^2(Y_\R) ) ,
\]
where $r^{(2)}_\psi(p) = m_\psi(p) r_\psi(p) $ as in \eqref{rao correction}.   
For  $p \in P_\R^+$,  Rao's normalizing factor satisfies $m_\psi(p)=1$.
So  really the claim is that  $r_\psi(p)  \in \Aut (L^2(Y_\R) )$ varies continuously, in a suitable sense, with $p\in P_\R^+$. 
As per Weil's characterization \cite[\S 35]{Weil1} of the topology on the metaplectic group, 
the suitable sense is with respect to the strong operator topology on $\Aut (L^2(Y_\R) )$.
A proof of the continuity of $r_\psi(p)$  as a function of $p$ (on all of $P_\R$, not just $P_\R^+$)  can be found in \cite[Chapter I.7]{Igusa}.
\end{proof}

Let us now prove (1).   As in Remark \ref{rem:alpha ambiguity}, making a different choice of $\alpha$ amounts to multiplying it on the right by some $A \in \mathrm{SO}(d)$, which multiplies $g_\mathtt{z}$ on the right by the element 
\[
m(A) \in P_\R^+ \cap \ker(  \chi : K \to \C^1 )  \subset \Sp_{2d}(\R) .
\]
Denoting the continuous homomorphism of Lemma \ref{lem:siegel section}  by $p \mapsto \tilde{p}=(p,1)$, we see that 
making a different choice of $\alpha$ multiplies  $\tilde{g}_\mathtt{z}$ on the right by $\tilde{m}(A)$.
This multiplies   \eqref{half circle cocycle}  by the inverse of $\chi^m ( \tilde{m}(A) )$, which is $\pm 1$ because  
\[
\chi^m ( \tilde{m}(A) )^2 = \chi( m(A)  )^{2m} =1.
\]
The point is now simply that $A \mapsto \chi^m( \tilde{m}(A) )$ is a continuous homomorphism $\mathrm{SO}(d) \to \mu_2$, which must be the  trivial homomorphism because $\mathrm{SO}(d)$ is connected.  
This proves that   \eqref{half circle cocycle}  does not depend on the choice of $\alpha$ used to define $g_\mathtt{z}$.
By the same reasoning it is independent of the analogous choice made to define $g_{\gamma \mathtt{z}}$.
In other words $c(\gamma , \mathtt{z})^m$, hence also  $j( \gamma , \mathtt{z})^m$,  is well-defined.

To complete the proof of Proposition \ref{prop:good half cocycle},  choose the elements \eqref{natural g_z} to vary continuously with $\mathtt{z}$, for example by taking $\alpha$ to be the unique positive definite symmetric square root of $\mathtt{y}$.  Then use the  continuity claim of Lemma \ref{lem:siegel section} to see that \eqref{half circle cocycle} is continuous as a function of  $\mathtt{z}$, and so the same is true of $j( \gamma , \mathtt{z})^m$. 
\end{proof}

Recall that, exactly as in \eqref{metaplectic product}, we may realize our adelic metaplectic double cover as a quotient 
\[
\prod\nolimits_v^\prime \Mp(W_{\Q_v})^{(2)}_\psi \to \Mp(W_\A)^{(2)}_\psi .
\]
of a restricted topological product (over all places of $\Q$)  of local metaplectic double covers.

Suppose  $F$ is a  function on $\Sp(W) \backslash \Mp(W_\A)_\psi^{(2)}$  satisfying 
\begin{equation}\label{siegel 1/2 weight}
F(g k ) = \chi^m(k) F(g) 
\end{equation}
for all $k \in K^{(2)}$.  Here, by abuse of notation, we are using the same notation for  $K^{(2)}$ and for its image under
\[
\Mp(W_{\R})^{(2)}_\psi  \to  \Mp(W_\A)^{(2)}_\psi .
\]
Suppose also that there is  some  compact open subgroup
\[
U \subset \prod\nolimits_{v\neq \infty}^\prime \Mp(W_{\Q_v})^{(2)}_\psi 
\]
such that, using a similar abuse of notation,  $F(g u ) = F(g)$ for all $u \in U$.
Define 
\[
\Gamma =  \{ \gamma \in \Mp(W_\R)_\psi^{(2)} :  \exists u \in U \mbox{ for which }
\gamma u \in \Sp(W) 
\} .
\]
On the right hand side the product $\gamma u$ is understood in  $\Mp(W_\A)^{(2)}_\psi$, and we use the canonical splitting \eqref{double splitting} to regard $\Sp(W)$ as a subgroup of the same adelic metaplectic group.

\begin{proposition}\label{prop:half-integral de-adelization}
The function 
\begin{equation}\label{half-classical siegel}
F(\mathtt{z}) \define    \det(\mathtt{y})^{- m/2 } F( \tilde{g}_\mathtt{z} )  
\end{equation}
of $\mathtt{z} \in \mathcal{H}_d$ satisfies  the transformation law
\[
F(\gamma \mathtt{z} ) =  j(  \gamma   ,\mathtt{z})^m F(\mathtt{z}) 
\]
for all $\gamma \in  \Gamma \subset \Mp(W_\R)_\psi^{(2)}$.
\end{proposition}

\begin{proof}
We have formulated  our definitions in such a way that the proof of  Proposition \ref{prop:siegel transformation} goes through essentially word for word.
\end{proof}

The $\Gamma$ in Proposition \ref{prop:half-integral de-adelization}  is a subgroup of $\Mp(W_\R)_\psi^{(2)}$, but one can always rephrase the proposition as a statement about how $F(\mathtt{z})$ transforms under  a  subgroup of $\Sp(W)$. 
The key observation for this is that on the central $\mu_2 \subset \Mp(W_\R)_\psi^{(2)}$, 
which is contained in $K^{(2)}$, we have $\chi^m( \epsilon )=  ( \chi^{1/2}(\epsilon))^{2m} = \epsilon^{2m}$. 
\begin{itemize}
\item
If $m\in \Z$, then \eqref{siegel 1/2 weight} implies that $F$ is invariant under the central $\mu_2$, so descends to a function on $\Sp(W) \backslash \Sp(W_\A)$.   The function  $F(\mathtt{z})$ defined by  \eqref{half-classical siegel} agrees with the one defined by \eqref{classical siegel}, and we are back in the situation of Proposition \ref{prop:siegel transformation}.
\item
If $m \not\in \Z$,  then \eqref{siegel 1/2 weight} implies that $F$ is \emph{genuine}, in the sense that 
$F(\epsilon g) = \epsilon F(g)$ for all $\epsilon \in \mu_2$.  
Assuming $F \neq 0$, this forces  the subgroup $\Gamma$ to intersect the central $\mu_2$ trivially, and so the covering map  identifies $\Gamma$ with its image in $\Sp(W_\R)$.  This  image is contained in $\Sp(W)$, essentially by definition of $\Gamma$.
\end{itemize}
In either case, it makes sense to say that $F(\mathtt{z})$  \emph{transforms like a Siegel modular form of weight $m$ under the action of some arithmetic subgroup of $\Sp_{2d}(\Q)$}.


\subsection{Siegel theta functions}


We continue to work  over the global field $F=\Q$, and for our additive character
\[
\psi : \Q \backslash \A \to \C^1
\]
take the unique one with archimedean component  $\psi_\infty(x) = e^{2\pi i x}$.

Fix a half-integer $m\in 2^{-1}\Z$.
Recalling the notation \eqref{half-character}, suppose  
\[
\varphi = \varphi_\R \otimes \varphi_f \in S(V_\A^d)
\]
 is a Schwartz function whose archimedean component  transforms under the Weil representation of Definition \ref{def:symplectic-orthogonal weil}  as  
\begin{equation}\label{orthogonal symplectic K weight}
\omega_{V,\psi}(k ) \varphi_\R =   \chi^m(k) \cdot  \varphi_\R,
\end{equation}
for all $k$ in the compact subgroup  $K^{(2)} \subset \Mp(W_\R)_\psi^{(2)}$.   
The associated adelic theta function \eqref{adelic siegel theta} then satisfies 
\[
\theta_{V,\psi}( g k , \varphi) = \chi^m(k)  \cdot \theta_{V,\psi}( g   ,  \varphi)
\]
for all $k \in K^{(2)}$, so we may apply the de-adelization construction \eqref{half-classical siegel} to obtain a function 
\begin{equation}\label{classical theta}
\theta_{V,\psi}(\mathtt{z} , \varphi) \define    \det(\mathtt{y})^{- m/2 }  \theta_{V,\psi}( \tilde{g}_\mathtt{z} , \varphi) 
\end{equation}
of the variable $\mathtt{z}  = \mathtt{x} + i \mathtt{y} \in \mathcal{H}_d$.
By Proposition \ref{prop:half-integral de-adelization}, this function  transforms like a Siegel modular form of half-integral weight $m$, under the action of some  arithmetic   subgroup of $\Sp_{2d}(\Q)$.

\begin{proposition}\label{prop:classical theta expansion}
The Fourier expansion of  \eqref{classical theta}  is
\[
\theta_{V,\psi} (\mathtt{z} , \varphi ) = \sum_{ T \in \Sym_d( \Q)} R_T( \mathtt{y}  , \varphi )
  \cdot  e^{2 \pi i \mathrm{Tr}( T \mathtt{z} ) }  , 
\]
where  the coefficients are
\[
R_T( \mathtt{y}  , \varphi )   = 
\det( \mathtt{y} )^{- \frac{m}{2} + \frac{1}{4} \dim_\Q(V)  }  
  \sum_{  \substack{  v \in V^d \\ Q(v) = T } } \varphi_\R( v  \alpha)  \varphi_f(v)  \, e^{2\pi  \mathrm{Tr}( Q(v  \alpha )) } .
\]
Here  $\alpha \in \GL_d(\R)$ is any matrix of positive determinant satisfying   $\mathtt{y} =\alpha \cdot {}^t \alpha$.
The product $v \alpha \in V_\R^d$ is computed by viewing $v=(v_1,\ldots, v_d)\in V_\R^d$ as a row vector and multiplying it by the matrix $\alpha$ in the usual way.
\end{proposition}

\begin{proof}
Fix a   $\mathtt{z} \in \mathcal{H}_d$,  and a    factorization  $\mathtt{y} =\alpha \cdot {}^t \alpha$ as above.
By definition, 
\[
\theta_{V,\psi}( \mathtt{z} ,\varphi)  
 = 
\det( \mathtt{y} )^{- \frac{m}{2} }
\sum_{    v  \in V^d  }    \omega_{V,\psi}(  \tilde{g}_\mathtt{z}  ) \varphi  ( v )  ,
\]
where $g_\mathtt{z} \in \Sp(W_\R)$ and $\tilde{g}_\mathtt{z} \in \Mp(W_\R)_\psi^{(2)}$ are defined by  \eqref{natural g_z} and \eqref{g_z lift}.

As in the proof of Lemma \ref{lem:siegel section}, and  using the notation of Proposition \ref{prop:symplectic-orthogonal weil formulas},  the Rao cocycle is trivial on the subgroup
\[
\left\{  
m(A)n(B) : A \in \GL_d(\R) , B \in \Sym_d(\R) , \det(A) >0 
\right\} \subset \Sp_{2d}(\R) .
\]
Because of this, the factorization $g_{\mathtt{z}} = n(\mathtt{x}) \cdot  m(\alpha)$ from 
 \eqref{natural g_z} implies the analogous factorization 
 \[
 \tilde{g}_\mathtt{z}   = \tilde{n} ( \mathtt{x}) \cdot  \tilde{m}( \alpha)  
 \]
 of \eqref{g_z lift}, where we  are using the Rao coordinates \eqref{rao coords} to regard
\[
\tilde{n}(\mathtt{x}) = ( n ( \mathtt{x}) , 1 ) \qquad  \mbox{and}  \qquad \tilde{m}(\alpha) = ( m(\alpha)  , 1)
\]
 as elements of $\Mp(W_\R)_\psi^{(2)}$.  In particular,
 \[
 \omega_{V,\psi}( \tilde{g}_\mathtt{z}   ) =
 \omega_{V,\psi}( \tilde{n}(\mathtt{x})   )  \circ \omega_{V,\psi}( \tilde{m}(\alpha)   ).
 \]

Using the explicit formulas from Proposition \ref{prop:symplectic-orthogonal weil formulas}
(and also  Remark \ref{rem:beta on parabolics},  to see that the $\beta_v$ factors are all $1$), we find 
\begin{align*}
 \omega_{V,\psi}( \tilde{g}_\mathtt{z}   ) \varphi (v)   = 
e^{2\pi i \mathrm{Tr}( Q(v)  \mathtt{x} )}  \det( \mathtt{y} ) ^{  \frac{1}{4} \dim(V)  }   \varphi( v  \alpha)   ,
\end{align*}
for any $ v  \in V^d$.   Hence
\begin{equation}\label{coefficient computation 2}
\theta_{V,\psi}( \mathtt{z} ,\varphi)  
 = 
\det( \mathtt{y} )^{- \frac{m}{2 } + \frac{1}{4} \dim(V)  }  
\sum_{    v \in V^d  }    \varphi( v  \alpha)   \cdot  e^{2\pi i \mathrm{Tr}( Q(v)  \mathtt{x} )}   .
\end{equation}
To complete the proof,  note that 
\[
e^{2\pi  \mathrm{Tr}(   Q(v)  \mathtt{y}   ) }  =e^{2\pi  \mathrm{Tr}({}^t  \alpha  Q(v)  \alpha ) } =
 e^{2\pi  \mathrm{Tr}( Q(v  \alpha )) } ,
\]
which implies
\[
e^{2 \pi i \mathrm{Tr} ( Q(v) \mathtt{x}  ) }  = 
e^{2 \pi i \mathrm{Tr} ( Q(v) \mathtt{z} ) } e^{2\pi  \mathrm{Tr}(   Q(v \alpha)     ) }   .
\] 
This allows us to rewrite \eqref{coefficient computation 2}  as 
\begin{align*}
\theta_{V,\psi}( \mathtt{z} ,\varphi)  
& = 
\det( \mathtt{y} )^{- \frac{m}{2} + \frac{1}{4} \dim(V)  }  
\sum_{    v \in V^d  }        \varphi( v  \alpha)  e^{ 2\pi  \mathrm{Tr}(   Q(v \alpha)     ) }  e^{2 \pi i \mathrm{Tr} ( Q(v) \mathtt{z} ) }    ,
\end{align*}
as desired.
\end{proof}

By  the discussion above,  the problem of constructing theta functions on $\mathcal{H}_d$ is entirely reduced to the problem of constructing  Schwartz functions $\varphi_\R \in S(V_\R^d)$ that transform like   \eqref{orthogonal symplectic K weight} under the Weil representation.
  The following theorem shows how to do this in the most classical case, recovering the original theta functions of Jacobi and Siegel.

\begin{theorem}\label{thm:classical theta modularity}
Assume that $V$ is positive definite.  
For any $\varphi_f \in S(V^d_{\A_f})$, the function  
\[
\vartheta (\mathtt{z} , \varphi_f ) = 
\sum_{ T \in \Sym_d( \Q)}  \left(   \sum_{ \substack{ v \in V^d \\ Q(v) = T } } \varphi_f(v)   \right)  
e^{2 \pi i \mathrm{Tr}( T \mathtt{z} ) } 
\]
is a holomorphic Siegel modular form on $\mathcal{H}_d$ of weight  $2^{-1} \dim_\Q(V)$.  Its level, an arithmetic subgroup of $\Sp_{2d}(\Q)$,  depends on $\varphi_f$.
\end{theorem}

\begin{proof}
This is just a question of choosing the right  $\varphi_\R \in S( V_\R^d)$ in Proposition \ref{prop:classical theta expansion}, for which we take the standard Gaussian
\[
\varphi^\circ_\R(v)  = e^{-2\pi  \mathrm{Tr}( Q(v  )) }   =  \prod_{j=1}^d e^{- \pi   b (v_j , v_j  ) } .
\]

If we can show that  the standard Gaussian satisfies the transformation law \eqref{orthogonal symplectic K weight} with  
\[
m = 2^{-1} \dim(V),
\]
   then we are done:   Proposition \ref{prop:classical theta expansion} applies to  the Schwartz function $\varphi = \varphi^\circ_\R \otimes \varphi_f$, and shows that the theta function  $\theta_{V,\psi} (\mathtt{z} , \varphi )$,  which we already know transforms like a Siegel modular form  of weight $m$,  is none other than   the function $\vartheta (\mathtt{z} , \varphi_f )$ in the statement of the theorem.

To show that the standard Gaussian satisfies the correct transformation law, we return to the discussion  of \S \ref{s:infinitesimal weil}.
Define $J \in \Sp(W)$ as in  \eqref{standard J}, so that $K$ is its centralizer in $\Sp_{2d}(\R) = \Sp(W_\R)$,  and $\chi$ is the distinguished character from \eqref{chi_J}. 

Denote by 
$
\bm{J} = J  \otimes   \mathrm{id}_V
$
the induced automorphism of $\bm{W} = W \otimes_\Q V$, 
let $K_{\bm{J}} \subset \Sp(\bm{W}_\R)$ be its  stabilizer, and let 
$\chi_{\bm{J}}$ be its distinguished character as in  \eqref{chi_J}.
As in \eqref{coordinate vacuum},  the vacuum vector $\varphi^\circ_{\bm{J}} \in S(\bm{Y}_\R) =S(V_\R^d)$ of Definition \ref{def:vacuum}  agrees with  the standard Gaussian defined above.

Consider the diagram 
\[
\begin{tikzcd}
 {  \Mp(W_\R)_\psi^{(2)} }  \ar[d] \ar[r, " i_{V,\psi} "]    &  {  \Mp(\bm{W}_\R) _\psi   }  \ar[d]  \\
{   \Sp(W_\R) }  \ar[r , "i_V" ]   &   {   \Sp(\bm{W}_\R)  }   ,
\end{tikzcd}
\]
of real Lie groups, as in \eqref{symplectic-orthogonal local splitting diagram}.
The bottom horizontal arrow takes $K$ into $K_{\bm{J}}$, and  the
distinguished characters $\chi$ and $\chi_{\bm{J}}$ of these groups  are related by 
\begin{equation}\label{orthogonal chi pullback}
\chi_{\bm{J}} \circ i_V=   \chi^{ \dim(V) }    .
\end{equation}
Indeed, both are equal to  the inverse of the determinant of $K $ acting on $\bm{W}_{\R, \bm{J}} = W_{ \R, J}  \otimes_\Q V$, where the action is only on the first tensor factor.

 Proposition \ref{prop:vacuum eigenvector} tells us that  the preimage of $K_{\bm{J}}$ in  $\Mp(\bm{W}_\R)_\psi$ acts on the vacuum vector through a character. 
 It follows that $K^{(2)}$,   the preimage of $K$ under the left vertical arrow in the diagram,  acts  on the standard Gaussian via the pullback of this  character along $i_{V,\psi}$.

To see what this pullback is, denote by  $\mathfrak{k} \subset \mathfrak{sp}(W_\R)$  the Lie algebra of $K$, which we identify with the Lie algebra of $K^{(2)}$.  
Similarly, let $\mathfrak{k}_{\bm{J}} \subset \mathfrak{sp}( \bm{W}_\R)$ be the Lie algebra of $K_{\bm{J}}$.
Recalling \eqref{mp lie}, the above diagram of Lie groups  induces the bottom half of the diagram of real Lie algebras 
\begin{equation}\label{orthogonal-symplectic lie diagram}
\begin{tikzcd}
{  \mathfrak{k} }  \ar[r] \ar[d] &  {    \mathfrak{k}_{ \bm{J} }   \oplus i \R }  \ar[d]  \\
  { \mathfrak{sp}(W_\R) }  \ar[d, equal ] \ar[r, " i_{V,\psi} "]    &  {    \mathfrak{sp}(\bm{W}_\R)  \oplus i\R  }  \ar[d]  \\
 { \mathfrak{sp}(W_\R) } \ar[r , "i_V" ]   &   {   \mathfrak{sp}(\bm{W}_\R)  }    .
\end{tikzcd}
\end{equation}
Note that the  horizontal arrows are related by $i_{V,\psi}(Z) = (  i_V(Z) , 0 )$, simply because there are no nonzero Lie algebra maps $\mathfrak{sp}(W_\R) \to i \R$.

The character through which  $K^{(2)}$ acts on the standard Gaussian differentiates to the composition of the top horizontal arrow with the linear functional  on $ \mathfrak{k}_{\bm{J}} \oplus i\R$ appearing in  Proposition \ref{prop:vacuum eigenvector}.
Using \eqref{orthogonal chi pullback}, this composition is 
\[
\frac{1}{2}  \cdot \chi_{\bm{J}}^\mathrm{inf} \circ i_V 
= \frac{1}{2}  \cdot  \dim(V)  \cdot \chi^\mathrm{inf} = m \cdot \chi^\mathrm{inf} ,
\]
which is the infinitesimal form of the character $\chi^m : K^{(2)} \to \C^1$.
Because $K^{(2)}$ is connected,  it follows that the standard Gaussian satisfies   \eqref{orthogonal symplectic K weight}, as desired.
 \end{proof}


\subsection{Polynomial weights}
\label{ss:orthogonal poly weights}


With the goal of generalizing Theorem \ref{thm:classical theta modularity}, 
 assume  throughout \S \ref{ss:orthogonal poly weights}  that $(V,b)$ is a positive definite quadratic space over $\Q$.
Fix an orthonormal   basis $v_1,\ldots, v_n \in V_\R$, so that  $b(v_j,v_k) =\delta_{j k} $.

Writing vectors in $V_\R$ as $x_1 v_1+ \cdots + x_n v_n$,
   we regard functions on $V_\R$ as functions of the real variables $x_\alpha$ with $1\le \alpha \le n$.
   For an integer $1\le j \le d$, let $x_{\alpha j}$ be the composition
   \[
   V_\R^d \map{ \mathrm{pr}_j }    V_\R \map{x_\alpha}  \R,
   \]
   where the first arrow  is projection to the $j^\mathrm{th}$ factor.
In this way we regard functions on $V_\R^d$ as functions of the variables $x_{\alpha j}$.

For each $1 \le j \le d$, define  a differential operator 
\begin{align*}
\Delta_j   & = \sum_\alpha   \frac{\partial^2 }{ \partial x^2 _{\alpha j }}    
\end{align*}
on the space of ($\C$-valued)  smooth  functions on the real vector space $V_\R^d$. 
 In other words, $\Delta_j$ is the usual Laplace operator on the $j^\mathrm{th}$ copy of $V_\R$ in $V_\R^d$.
 It is independent of the choice of orthonormal basis used to define it, hence so is $\Delta= \Delta_1+\cdots+\Delta_d$.
 
Given a  polynomial function $P : V_\R^d \to \C$, meaning a polynomial in the variables $x_{\alpha j}$,   define a new polynomial
\begin{equation}\label{Psharp}
P^\sharp =  \exp\left( -\frac{\Delta}{8\pi} \right)  P  
\end{equation}
as in the introduction.
Clearly  $P\mapsto P^\sharp$ is an automorphism of the space of all polynomial functions on $V_\R^d$.
The definition is motivated by the following result of Roehrig \cite{Roehrig}, whose proof will occupy the remainder of this subsection.

\begin{theorem}\label{thm:main orthogonal}
Fix  an integer $N \in \Z$, and  suppose $P$ is a polynomial function on $V_\R^d$ satisfying
\begin{equation}\label{P homogeneity}
P( v g ) = \det(g)^N P(v) 
\end{equation}
 for  all $v\in V_\R^d$ and $g \in \GL_d(\R)$.    
 For any Schwartz function $\varphi_f \in S(V^d_{\A_f})$,   the  function 
\[
\vartheta_P (\mathtt{z} ,  \varphi_f ) = 
\det(\mathtt{y})^{-N/2} 
\sum_{ T \in \Sym_d( \Q)}  \left(   \sum_{ \substack{ v \in V^d \\ Q(v) = T } }   P^\sharp(v \alpha ) \varphi_f(v  )   \right)   \, e^{2 \pi i \mathrm{Tr}( T \mathtt{z} ) } 
\]
of the variable $\mathtt{z} =\mathtt{x} + i \mathtt{y}\in  \mathcal{H}_d$ transforms like a Siegel modular form of weight $N+2^{-1} \dim_\Q(V)$, under the action of some arithmetic subgroup of $\Sp_{2d}(\Q)$.
Here  $\alpha \in \GL_d(\R)$  is any matrix of positive determinant such that  $\mathtt{y}=\alpha \cdot {}^t \alpha$. 
\end{theorem}

\begin{remark}
The assumption  \eqref{P homogeneity} implies that $P$ is right invariant under the subgroup
$\mathrm{SO}(d) \subset \GL_d(\R)$.
The differential operator $\Delta$  commutes with the action of this subgroup, 
and hence  $P^\sharp$ is also  $\mathrm{SO}(d)$-invariant.  As in  Remark \ref{rem:alpha ambiguity}, it follows that 
$P^\sharp (v \alpha)$ is independent of the choice of $\alpha$.
\end{remark}

\begin{remark}
Theorem \ref{thm:intro main orthogonal} of the introduction is just the special case in which  $\varphi_f$ is the characteristic function of $\prod_p L_{\Z_p} \subset V_{\A_f}$.
\end{remark}

%
%

We   begin the  proof of Theorem \ref{thm:main orthogonal}.
Our calculations are  related to those of \cite[\S 9]{Mumford-3}, but we allow nonharmonic polynomials $P$, and carry out all calculations at the Lie algebra level, rather than working with the group version of the Weil representation.

For $j,k \in \{ 1, \ldots, d\}$, define differential operators  
\begin{equation*}
E_{jk}   \define  \sum_\alpha x_{\alpha j}  \frac{\partial}{ \partial x_{\alpha k}}  
\qquad \mbox{and} \qquad 
F_{jk}   \define \sum_\alpha \frac{\partial }{ \partial x _{\alpha j }}    \frac{\partial }{ \partial x _{\alpha k }}    \end{equation*}
on the space of smooth functions on $V_\R^d$.  The following is our version of \cite[Proposition 3.4]{Roehrig}.

\begin{lemma}\label{lem:homo to diffeq}
The transformation law \eqref{P homogeneity} for $P$  is equivalent to the system of differential equations 
\begin{equation}\label{homo equations}
E_{jk} P = N \delta_{jk} P .
\end{equation}
\end{lemma}

\begin{proof}
Define a representation  $r$ of $\GL_d(\R)$  on the space of  polynomial functions  on $V_\R^d$  by $r(g)P(v) = P(v  g)$.  Denote in the same way the   induced representation of the Lie algebra  $\mathfrak{gl}_d(\R)$.

The transformation law \eqref{P homogeneity} holds for all $g \in \GL_d(\R)$ if and only if it holds for all $g$ in the connected component of the identity (by a Zariski density argument), 
and by   basic representation theory this  is equivalent to $P$ satisfying
\begin{equation}\label{right regular trace}
r(X )P = N  \mathrm{Tr}(X)P .
\end{equation}
for all $X \in \mathfrak{gl}_d(\R)$.

Denote by  $X_{jk} \in \mathfrak{gl}_d(\R)$ the matrix with a $1$ in the $(j,k)$ entry,   and $0$'s elsewhere.
By differentiating the action $\GL_d(\R)$, one sees that
$r(X_{j k } ) =  E_{jk}$, and hence the collection of equalities  \eqref{right regular trace}  as  $X$ varies  is  equivalent to the system of differential equations  \eqref{homo equations} as $j,k\in \{1,\ldots, d\}$ vary.
\end{proof}

The following is our version of \cite[Lemma 3.7]{Roehrig}.

\begin{lemma}\label{lem:diffeq conversion}
The system of differential equations \eqref{homo equations} for $P$ is equivalent to the system of differential equations 
\[
\left( E_{jk}  -  \frac{1}{4\pi}  F_{j k} \right)  P^\sharp 
= N    \delta_{jk}   P^\sharp 
\]
for the polynomial $P^\sharp$ from  \eqref{Psharp}.
\end{lemma}

\begin{proof}
By direct calculation, one proves the commutator relation 
\[
 \Delta  E_{j k}     = E_{ j k } \Delta +    2  F_{ j k} .
\] 
Using induction and the observation that $\Delta$ and $F_{jk}$ commute, this can be  strengthened to 
\[
 \Delta^r    E_{j k}      = E_{j k}   \Delta^r    +   2 r F_{ j k} \Delta^{r-1}
\]
for all $r \ge 1$.  An elementary calculation then shows
\begin{align*}
\exp\left( -\frac{\Delta}{8\pi} \right) \circ  E_{jk}  
& =    \left(  E_{jk}    -      \frac{1}{4 \pi}    F_{ j k}   \right)  \circ \exp\left( -\frac{\Delta}{8\pi} \right)    ,
\end{align*}
which proves the equality 
\[
( E_{jk}  P)^\sharp =  \left(  E_{jk}    -      \frac{1}{4 \pi}    F_{ j k}   \right) P^\sharp.
\]
The claim follows immediately from this.
\end{proof}

This is the point where we diverge from the arguments of \cite{Roehrig}.
Let us compute the action of the vectors 
\begin{equation}\label{Z_jk}
Z_{jk}  \define    (e_j  -  i  f_j  )  \cdot ( e_k  +  i f_k   )  \in \Sym^2(W_\C) 
\stackrel{  \eqref{sym2sp} } { \iso }   \mathfrak{sp}(W_\C)
\end{equation}
under the composition
\[
 \mathfrak{sp}(W_\C)   \map{ i_V }  \mathfrak{sp}(\bm{W}_\C )
  \map{  \omega^\mathrm{inf}  }   \End_\C( S(V_\R^d ))
\]
of   \eqref{symplectic i_V} with the infinitesimal Weil representation of Definition \ref{def:Lie Schrodinger}.

\begin{lemma}\label{lem:inf Z_jk}
 The vectors defined above satisfy
\begin{align*}
\omega^\mathrm{inf}  (  i_V( Z_{jj} )    ) & =
 2 \pi i  \sum_\alpha  \left(  x_{\alpha j} ^2 
-  \frac{1}{ 4 \pi^2 } \frac{ \partial^2}{ \partial x_{ \alpha j}^2}  \right),
\end{align*}
and, for $j \neq k$, 
\begin{align*}
\omega^\mathrm{inf}  (  i_V(  Z_{jk}  ) ) 
& =
 2 \pi i   \sum_\alpha  \left(   x_{\alpha j } +  \frac{1}{2\pi}  \frac{\partial}{ \partial x_{\alpha j } }   \right)
 \left(   x_{\alpha k } -  \frac{1}{2\pi}  \frac{\partial}{ \partial x_{\alpha k } }   \right) .
  \end{align*}
\end{lemma}

\begin{proof}
Let $\mathcal{W}$ be the  Weyl algebra (Definition \ref{def:weyl algebra}) of the real symplectic space $\bm{W}_\R$.
Recalling the identification \eqref{sym2sp},  one sees that the composition
\[
\Sym^2(W_\R) \iso \mathfrak{sp}(W_\R) \map{i_V}  \mathfrak{sp}(\bm{W}_\R) \iso  \Sym^2(\bm{W}_\R)
\]
  sends
\[
w_1  w_2  \mapsto \sum_\alpha    ( w_1 \otimes v_\alpha ) (w_2 \otimes v_\alpha )  
\]
for any $ w_1, w_2  \in W_\R$.
Tracing through the constructions leading to Definition \ref{def:Lie Schrodinger}, it follows that 
\begin{align*}
 \omega^{\mathrm{inf}}(  i_V( w_1 w_2 ) ) = 
    \sum_\alpha 
\left(    \frac{1}{2\pi i}
    \rho  ( w_1 \otimes v_\alpha  ) \circ \rho ( w_2 \otimes v_\alpha ) 
    -  \frac{  1 }{2   }  \langle w_1 , w_2  \rangle   \right),
\end{align*}
where we use the composition 
$
\bm{W}_\R \to \bm{W}_\C^{\otimes}  \to \mathcal{W}  
$
to regard $w_1 \otimes v_\alpha$ and $w_2 \otimes v_\alpha$ as elements of the Weyl algebra,
and  $\rho$ is the action of the Weyl algebra  on $S(\bm{Y}_\R)=S(V_\R^d)$ from \eqref{Weyl module}.
 Remark \ref{rem:weyl action} shows that
 \[
\rho   (  e_j \otimes v_\alpha )    =   2 \pi i x_{\alpha j} 
\qquad \mbox{and} \qquad 
\rho   ( f_j \otimes v_\alpha  )   = - \frac{\partial}{ \partial x_{ \alpha j } }   ,
\]
and from these formulas the claim follows easily.
\end{proof}

\begin{proof}[Proof of Theorem \ref{thm:main orthogonal}]
In  the coordinates $x_{\alpha j}$ on $V_\R^d$ determined by our orthonormal basis of $V_\R$, the standard Gaussian $\varphi_\R^\circ \in S(V_\R^d)$ from the proof of Theorem \ref{thm:classical theta modularity} becomes 
\[
\varphi^\circ_\R(v)  =  e^{-2\pi   \mathrm{Tr}(  Q(v ) ) } 
= \prod_{ j =1 }^d  e^{ - \pi   \sum_\alpha x^2_{\alpha j} } .
\]
Exactly as in the proof of Theorem \ref{thm:classical theta modularity},
if we can  show that the Schwartz function  $\varphi_\R = P^\sharp \varphi_\R^\circ$ satisfies the transformation law \eqref{orthogonal symplectic K weight} with 
 \[
 m =  N + \frac{1}{2} \dim(V)  ,
 \]
 then we are done by  Proposition  \ref{prop:classical theta expansion}.

  Using  the explicit formulas of Lemma \ref{lem:inf Z_jk},  it is a calculus exercise to  verify  that
\[
\omega^\mathrm{inf}(  i_V(  Z_{k j }  ) )( f \varphi^\circ_\R)
=   2 i    \cdot    \left(   E_{jk} f  - \frac{1}{4\pi  } F_{jk} f+  \frac{ \delta_{jk}  }{2} \dim(V)  f \right) \cdot  \varphi^\circ_\R
\]
for any polynomial function $f$ on $V_\R^d$.   
Apply this with $f=P^\sharp$.  As we are assuming   $P$ satisfies  \eqref{P homogeneity},
 combining  Lemmas \ref{lem:homo to diffeq} and \ref{lem:diffeq conversion} shows that
 \[
\left( E_{jk}  -  \frac{1}{4\pi}  F_{j k} \right)  P^\sharp 
= N   \delta_{jk}   P^\sharp ,
\]
and hence
\[
\omega^\mathrm{inf} (   i_V( Z_{k j } )   ) ( P^\sharp    \varphi_\R^\circ) 
 =  2 i   \delta_{jk}   \cdot  \left( N + \frac{1}{2} \dim(V) \right)  \cdot  P^\sharp   \varphi_\R^\circ .
\]

Let $ \mathfrak{k} \subset \mathfrak{sp}(W_\R)$ be   the  Lie algebra of  $K\subset \Sp(W_\R)$, which is canonically identified with the Lie algebra of $K^{(2)} \subset \Mp(W_\R)_\psi^{(2)}$.
By Lemma \ref{lem:basic Z_jk action},  the vectors \eqref{Z_jk}  span the subalgebra  $\mathfrak{k}_\C \subset \mathfrak{sp}(W_\C)$, and  the linear functional  $\chi^\mathrm{inf} : \mathfrak{k} \to i \R$
 obtained by differentiating the character $\chi : K \to \C^1$ (extended $\C$-linearly to $\mathfrak{k}_\C \to \C$)  satisfies
$
\chi^\mathrm{inf} (Z_{k j })   = 2 i \delta_{jk} .
$
It follows that for all $Z \in \mathfrak{k}$, we have
\begin{equation}\label{final sharp inf action}
\omega^\mathrm{inf} (   i_V(  Z )    ) ( P^\sharp    \varphi_\R^\circ) 
 =      \left( N + \frac{1}{2} \dim(V) \right)   \cdot   \chi^\mathrm{inf} (Z)   \cdot  P^\sharp   \varphi_\R^\circ .
\end{equation}

Now consider the action  of $K^{(2)}$ on $S(V_\R^d)$ through the Weil representation  $\omega_{V,\psi} = \omega_\psi \circ i_{V,\psi}$ of Definition \ref{def:symplectic-orthogonal weil}. 
We  may differentiate the action of the Weil representation to an action of the Lie algebra $\mathfrak{k}$  on $S(V_\R^d)$.
By Theorem  \ref{thm:inf compatibility} and the relation between $i_V$ and $i_{V,\psi}$ in the diagram \eqref{orthogonal-symplectic lie diagram}, the resulting  Lie algebra action of $\mathfrak{k}$ on $S(V_\R^d)$  is precisely the composition $\omega^\mathrm{inf}\circ i_V$ appearing on the left hand side of \eqref{final sharp inf action}.  
Thus we may conclude from \eqref{final sharp inf action} that
\[
\omega_{ V , \psi} (   k    ) ( P^\sharp    \varphi_\R^\circ) 
 =      \chi^m   (k)    P^\sharp   \varphi_\R^\circ 
\]
for all $k \in K^{(2)}$, completing the proof of Theorem \ref{thm:main orthogonal}.
\end{proof}


\section{Unitary-unitary dual pairs}
\label{s:unitary-unitary}


In this section we use the theory of unitary-unitary reductive dual pairs in the sense of Howe \cite{Howe-theta} to embed a  quasi-split unitary group into a metaplectic group.   Using the restriction of   the Weil representation of the larger group to the smaller one, we  construct  theta functions on the adelic unitary group, and then convert them into classical theta functions on the Hermitian half-space. 


\subsection{The setup}
\label{ss:uu setup}


Let $E/F$ be a quadratic extension of global fields of characteristic $\mathrm{char}(F) \neq 2$, and denote by $\sigma \in \Gal(E/F)$ the nontrivial Galois automorphism.  We write $\A_E$ for the ring of adeles of $E$, but continue to abbreviate $\A=\A_F$ as before.
For a matrix $A \in M_{m\times n}(E)$ of any size, write
\[
{}^\dagger A = {}^t A^\sigma
\]
for its conjugate transpose.
Fix an additive character  $\psi : F \backslash \A \to \C^1$.

Let $V$ be a finite dimensional vector space over $E$, endowed with a nondegenerate  Hermitian form 
\[
h : V \times V \to E .
\]
 Thus $h$ is $E$-linear in the first variable, and satisfies 
 $
 h(v_1,v_2) = h(v_2,v_1)^\sigma . 
 $
  For a $d$-tuple $v=(v_1,\ldots, v_d) \in V^d$, denote by 
\begin{equation}\label{hermitian gram}
H(v) = \left(  h(v_k  ,v_j)   \right)_{ 1 \le j,k \le d } \in \Herm_d(E)
\end{equation}
 the matrix of Hermitian  inner products of the components of $v$.
The order of indices matters here: the entry in the $j^\mathrm{th}$ row and $k^\mathrm{th}$ column is   $h( v_k , v_j )$.

Let $W$ be a vector space over $E$ of dimension $2d$,  endowed with a skew-Hermitian form $\langle-,-\rangle : W \times W \to E$.  
This  pairing  is $E$-linear in the \emph{second} variable, and satisfies $\langle w_1,w_2 \rangle = - \langle w_2 ,w_1\rangle^\sigma$. 
Assume one can choose  basis vectors  $e_1,\ldots, e_d, f_1,\ldots, f_d \in W$ in such a way  that $\langle e_j,f_k\rangle = \delta_{jk}$, and the subspaces
\[
X=\mathrm{Span} \{ e_1,\ldots, e_d \}   \qquad \mbox{and} \qquad Y=\mathrm{Span} \{ f_1,\ldots, f_d \} 
\]
are totally isotropic. 
The isometry group $\mathrm{U}(W)$ can then be  identified with the matrix group 
  \begin{equation}\label{quasi-split}
\mathrm{U}_{d,d}(F) = \left\{ 
g \in \GL_{2d}(E) :  {}^\dagger  g    \begin{pmatrix}  & I_d \\  - I_d \end{pmatrix}  g = \begin{pmatrix}  & I_d \\  - I_d \end{pmatrix} 
\right\} 
\end{equation}
 in the following way: Using the basis to regard $W = E^{2d}$ as a space of \emph{row} vectors,  a matrix $g \in \mathrm{U}_{d,d}(F)$ is identified with the  isometry of $W$ defined by $g( w) = w \cdot {}^\dagger g$, where the $\cdot$ on the right is usual matrix multiplication.

\begin{remark}\label{rem:banana peel}
Any $g \in \mathrm{U}_{d,d}(F)$ has a determinant $\det(g)$, computed as a matrix using the usual rules of undergraduate linear algebra.  
Any  $g \in \mathrm{U}(W)$ has a determinant $\det_W(g)$, computed as an $E$-linear automorphism of $W$.  
Under our isomorphism, these two determinants are related by 
\[
\det\nolimits_W(g) = \det({}^\dagger g ) = \det(g)^\sigma,
\]
and similarly for the traces $\mathrm{Tr}_W(g)$ and  $\mathrm{Tr}(g)$.  It is helpful  to maintain a notational and mental distinction between the group of matrices $\mathrm{U}_{d,d}(F)$ and the  group of isometries $\mathrm{U}(W)$, and to agree that an undecorated $\det$ or $\mathrm{Tr}$ is understood in the matrix sense.
\end{remark}

From $V$ and $W$ we can construct  a polarized symplectic space  by regarding
$
\bm{W}= W \otimes_E  V
$
as a vector space over $F$, endowed  with the symplectic form
 \begin{equation}\label{big unitary symplectic}
\langle  w_1 \otimes v_1 , w_2 \otimes v_2  \rangle_{\bm{W}} = 
 \mathrm{Tr}_{E/F} \big(   \langle w_1,w_2 \rangle^\sigma \cdot  h(v_1,v_2)    \big),
 \end{equation}
 and the polarization $\bm{W} = \bm{X} \oplus \bm{Y}$ defined by  $ \bm{X} = X \otimes_E V$ and  $\bm{Y} = Y \otimes_E V$.  We use $(v_1,\ldots, v_d ) \mapsto f_1 \otimes v_1 + \cdots + f_d \otimes v_d$ to identify 
\[
V^d \iso  \bm{Y} .
\]
 There is a natural homomorphism 
 $
 \mathrm{U}(W) \times \mathrm{U}(V) \to \Sp(\bm{W}),
 $
and in particular an injective homomorphism
\begin{equation}\label{unitary i_V}
i_V : \mathrm{U}(W) \to \Sp(\bm{W}) .
\end{equation}

Our normalization of \eqref{big unitary symplectic} follows \cite[Chapter 10.3.3]{ThetaBook}.   It is more common in the literature to follow \cite[(0.9)]{KudlaSplitting} and define the symplectic form on $\bm{W}$  to be one-half of ours.  In \S  \ref{ss:old normalization} we explain how the results of \S \ref{s:unitary-unitary} would change if we adopted this other convention.


\subsection{Lifting the embedding}


Following Kudla \cite{KudlaSplitting}, but using the notation and conventions of \cite{ThetaBook}, 
we will upgrade   \eqref{unitary i_V}  to an embedding of the  unitary group into the metaplectic group.
We will then restrict  the Weil representation of the metaplectic group to the unitary group, and use this to construct theta functions on the unitary group.
The first step is described in this subsection, and the second is described in the next.

Denote by 
\[
\eta_{E/F} : F^\times \backslash \A^\times \to \{\pm 1\}
\]
 the quadratic character determined by the extension $E/F$, and fix an auxiliary  character 
\begin{equation}\label{eta def}
\eta : E^\times \backslash \A_E^\times \to \C^1
\end{equation}
whose restriction to $\A^\times$ satisfies  $\eta |_{\A^\times } = \eta_{E/F}^{ \dim_E(V) }$.

For every place $v$ of $F$ we have the local  Leray cocycle $\bm{c}_{\psi,v}$ on $\Sp(\bm{W}_{F_v})$  from  \eqref{leray cocycle}.
  Using the choice of auxiliary character $\eta$,  Kudla \cite{KudlaSplitting} defines  an explicit function 
 \[
 \beta_v=\beta_{V, \psi ,v }^\eta  : \mathrm{U}(W_{F_v}) \to \C^1 
 \]
  in such a way that 
 \begin{equation}\label{unitary kudla coboundary}
 \bm{c}_{\psi , v}   \big(  i_V(g_1)  ,  i_V(g_2) \big)   \cdot  
    \frac{ \beta_v (g_1)   \cdot   \beta_v (g_2)  } {  \beta_v ( g_1g_2)  }  =1 
 \end{equation}
 for all $g_1,g_2 \in \mathrm{U}(W_{F_v})$.    
 The precise definition is  Case $3_{+}$ of  \cite[(11.3)]{ThetaBook}, keeping Remark \ref{rem:ThetaBook caveat} in mind.

 As in  \eqref{symplectic-orthogonal local splitting diagram},  there is an associated diagram
  \begin{equation}\label{unitary-unitary local splitting diagram}
\begin{tikzcd}
{ \mathrm{U}(W_{F_v})   }  \ar[r]   \ar[d,  equal ]     &  {  \Sp( \bm{W}_{F_v})  \times \C^1 }  \ar[d,  " \eqref{leray coords} " ]    \\ 
 {  \mathrm{U}(W_{F_v}) }  \ar[d, equal ] \ar[r, " i^\eta_{V,\psi} "]&  {  \Mp(\bm{W}_{F_v})_\psi  }  \ar[d] \\
{   \mathrm{U}(W_{F_v}) }  \ar[r , "i_V" ]   &  {   \Sp(\bm{W}_{F_v}) , } 
\end{tikzcd}
\end{equation}
in which the top horizontal arrow is 
$
g \mapsto (  i_V( g ) ,  \beta_v(g) ) . 
$
The arrow labeled $ i^\eta_{V,\psi}$ is defined by the commutativity of the diagram, and is  a group homomorphism by \eqref{unitary kudla coboundary}.

These  local homomorphisms  can be collected together into a homomorphism 
 \begin{equation}\label{unitary-unitary adelic lift}
i_{V,\psi}^\eta :  \mathrm{U}( W_\A )    \to   \Mp(\bm{W}_\A)_\psi .
 \end{equation}
 This adelization process  is  more direct in the  unitary-unitary setting than it was in the symplectic-orthogonal setting of  Proposition \ref{prop:so adelization}, for which we  simply cited Sweet's thesis.
 In this simpler  unitary-unitary setting, we at least indicate for the reader the basic outline of the proof.

\begin{proposition}
There is a unique continuous homomorphism \eqref{unitary-unitary adelic lift}  such that, for every finite set of places $\Sigma$ of $F$, the diagram 
 \[
 \begin{tikzcd}
 {  \prod_{v\in\Sigma} \mathrm{U}(W_{F_v})   }    \ar[d]   \ar[rr]      &  &  { \left( \prod_{v\in\Sigma} \Sp( \bm{W}_{F_v})  \right) \times \C^1 }  \ar[d , " \eqref{semi leray coordinates} " ]  \\ 
 { \mathrm{U}( W_\A )  } \ar[rr ,  " i^\eta_{V,\psi} "  ' ]    &   & {    \Mp(\bm{W}_\A)_\psi   }
 \end{tikzcd}
 \]
 commutes.  Here the top horizontal arrow is defined by 
 \[
g \mapsto \Big(  i_V(g) , \prod_{v\in \Sigma} \beta_v (g_v)  \Big) .
 \]
 Moreover,  for this homomorphism we have a commutative diagram
 \[
 \begin{tikzcd}
 {  \mathrm{U}(W) }    \ar[d ]   \ar[r , " i_V" ]        &  { \Sp( \bm{W}  )  }  \ar[d , " \eqref{canonical splitting}  " ]  \\ 
 { \mathrm{U}( W_\A )  } \ar[r ,  " i^\eta_{V,\psi} "  ' ]    &    {    \Mp(\bm{W}_\A)_\psi   .}
 \end{tikzcd}
 \]
 \end{proposition}

 \begin{proof}
 First, one verifies that for any $g \in \mathrm{U}(W_\A)$ the product   $\prod_v \beta_v(g)$
over all places  is well-defined, in the sense that all but finitely many factors are equal to $1$.   This follows from the explicit formula for $\beta_v$ found in  \cite[(11.3)]{ThetaBook}.  
As a consequence, one deduces the analogous statement for the product 
$\prod_v \bm{c}_{\psi,v}( i_V ( g_{1,v} ) ,  i_V ( g_{2,v} ) )$ of local Leray cocycles.  

Next,  one considers the local operator 
\begin{equation}\label{local unitary weil}
\omega_{V,\psi}^\eta(g_v) \define  \beta_v(g_v)   r_{\psi,v}( i_V(g_v) )  : S(\bm{Y}_{F_v}) \to S(\bm{Y}_{F_v})
\end{equation}
at each place $v$, where $r_{\psi,v}$ is defined by \eqref{big weil}.   
The key fact is that at all but finitely many places $v$, this operator fixes the distinguished vector $\varphi^\circ_v \in S(\bm{Y}_{F_v})$ from \S \ref{ss:adelization}. This follows from the formula for the operator found in \cite[Proposition 12.27]{ThetaBook}, at least in the case where  $g$ is one of a particular set of generators for  $\mathrm{U}(W_\A)$.
One then uses the previous paragraph to extend this to  general $g$.

 This key fact about  the local operators  is what allows us to make sense of the infinite tensor product
 \[
 \omega_{V,\psi}^\eta(g)  \varphi  \define 
    \otimes_v \left(  \omega_{V,\psi}^\eta(g_v) \varphi_v \right)
 \]
for any factorizable $\varphi = \otimes_v \varphi_v \in S(\bm{Y}_\A)$, and therefore define \eqref{unitary-unitary adelic lift} by 
\[
g \mapsto \left(  i_V(g) ,  \omega_{V,\psi}^\eta(g)  \right) \in   \Mp(\bm{W}_\A)_\psi \subset    \Sp(\bm{W}_\A) \times \Aut (S (\bm{Y}_\A )) .
\]
With this definition in hand, the proposition follows easily.  
The commutativity of the final diagram in the statement  boils down to checking that $\beta(g)=1$ for any rational point $g\in \mathrm{U}(W)$, which is a consequence of formulas for $\beta_v$ cited above, together with Weil's product formula \cite[\S 30]{Weil1}  for local Weil indices.
  \end{proof}

 \begin{remark}
At a place $v$ of $F$ that is nonsplit in $E$, the groups $\mathrm{U}(W_{F_v})$ and $\Sp( \bm{W}_{F_v})$ form a dual pair of Type I,  Case $3_{+}$, in the terminology of  \cite{ThetaBook}. 
 At a split place they form a dual  pair of Type II. 
In the discussion above we have implicitly suppressed the distinction, and our citations to  \cite{ThetaBook} above are to the results relevant to nonsplit places.   
At a split place the entire   discussion of Type II dual pairs (including Leray cocycles, the splitting function $\beta_v$, the local operators \eqref{local unitary weil}) found  in  \cite{ThetaBook} applies, but with a caveat: in the Type II case  results and formulas are often expressed using a different polarization of  $\bm{W}_{F_v}$  than the one obtained by localizing our global polarization $\bm{W} = \bm{X} \oplus \bm{Y}$.  
Thus at a split place one needs the extra step of keeping  track of how formulas  change under a change of polarization, in order to convert the relevant results from \cite{ThetaBook} to our setting.  Fortunately, \cite{ThetaBook} systematically keeps track of such things.
 In the end, one finds that the  local results and formulas that we have cited from \cite{ThetaBook} in the nonsplit case also hold in the split case.
  \end{remark}


\subsection{Restriction of the Weil representation}


The adelic metaplectic group $\Mp(\bm{W}_\A)_\psi$ acts on   $S(V^d_\A)= S(\bm{Y}_\A)$ via   the  Weil representation  $\omega_\psi$ of Definition \ref{def:adelic metaplectic}, and restricting this   along  the lower horizontal arrow in the diagram yields a representation of $\mathrm{U}(W_\A)$ on the same space.   We record this construction as a definition, the Hermitian analogue of Definition \ref{def:symplectic-orthogonal weil}.

\begin{definition}\label{def:unitary-unitary weil}
The \emph{Weil representation}  $\omega^\eta_{V,\psi}$ is the composition
\[
\omega^\eta_{V,\psi}  =  \omega_\psi  \circ  i^\eta_{V,\psi}  : \mathrm{U}(W_\A) \to \Aut( S(V^d_\A)  ) .
\] 
The \emph{adelic theta function} associated to a Schwartz function  $\varphi \in S(V_\A^d)$ is the continuous function  
\begin{equation}\label{adelic hermitian theta}
\theta^\eta_{V,\psi}( g ,  \varphi)
=   \sum_{ v \in V^d }  \left(  \omega^\eta_{V,\psi}(g) \varphi  \right) (v) 
\end{equation}
of the variable $ g \in \mathrm{U}(W) \backslash \mathrm{U}(W_\A)$.
\end{definition}

\begin{remark}
Exactly as in the symplectic-orthogonal case, the convergence and left  $\mathrm{U}(W)$-invariance of the adelic theta function  are immediate from  Theorem \ref{thm:big theta}, as 
\[
\theta^\eta_{V,\psi}( g ,  \varphi)  =  \Theta_\psi(    i^\eta_{V,\psi} (g)    , \varphi   ) .
\]
\end{remark}

\begin{proposition}\label{prop:unitary-unitary weil formulas}
For $A \in \GL_d(\A_E)$ and $B \in \Herm_d(\A_E)$, define elements of $\mathrm{U}(W_\A)$ by 
\[
m(A)  = \begin{pmatrix}  A  &  \\  &   {}^\dagger A^{-1}  \end{pmatrix} 
\qquad \mbox{and} \qquad 
n(B) = \begin{pmatrix}   I_d  & B  \\ &  I_d  \end{pmatrix} .
\]
Under the Weil representation on $S(V_\A^d)$, these satisfy
\[
\omega^\eta_{V,\psi} (m(A)  )   \varphi (v) 
=   \eta( \det({}^\dagger A) )   \cdot   |  \det(A)  |_E^{ \frac{1}{2} \dim_E(V) }   \cdot  \varphi(   v  A ) 
\]
and 
\[
\omega^\eta_{V,\psi} (n(B) ) \varphi (v) 
=    \psi \left(    \mathrm{Tr} \big(  H(v)  B \big)  \right) \cdot  \varphi(  v )  . 
\]
Here we have normalized the modulus function $| \cdot |_E$ on $\A_E^\times$   by 
\[
\int_{\A_E} f(ax) \, dx = |a|_E^{-1}  \int_{\A_E} f(x)\, dx,
\]
and the product $vA$ for $v \in V_\A^d$ has the same meaning (row vector times matrix) as in Proposition \ref{prop:symplectic-orthogonal weil formulas}.
\end{proposition}

\begin{proof}
The proof is virtually identical to that of Proposition \ref{prop:symplectic-orthogonal weil formulas}, given the description of $\omega^\eta_{V,\psi}(g)$ as a tensor product of  local operators \eqref{local unitary weil}.
One also uses the explicit formulas 
\begin{equation}\label{unitary parabolic beta}
\beta_v  \big( m(A_v)  \big )  = \eta_v( \det({}^\dagger A_v))  
 \quad \mbox{and} \quad
 \beta_v  \big( n(B_v)  \big )  = 1 
\end{equation}
from  \cite[(11.3)]{ThetaBook}.  Here one must remember  Remark \ref{rem:banana peel}; the matrix determinant $\det({}^\dagger A_v)$ agrees with the determinant of the restriction of $m(A_v) \in \mathrm{U}(W_{F_v})$ to $X_{F_v} \subset W_{F_v}$, which is the determinant relevant to Case $3_+$ of \cite[(11.3)]{ThetaBook}.
\end{proof}


\subsection{Generalities on Hermitian modular forms}


For the remainder of \S \ref{s:unitary-unitary} we work over the global field $F=\Q$, and assume that $E/\Q$ is quadratic imaginary.  Fix an embedding $E \to \C$, and use this to identify $E_\R \iso \C$.
We recall the passage  from  automorphic forms on the matrix group   $\mathrm{U}_{d,d}(\A)$  to functions  on the Hermitian half-space.   

Regard the quasi-split unitary group $\mathrm{U}_{d,d}$ from  \eqref{quasi-split} as a reductive group over $\Q$.
The group $\mathrm{U}_{d,d}(\R) \subset \GL_{2d}(\C)$ acts  on the Hermitian half-space
\begin{align*}
\mathcal{H}_d 
 &  = \left\{ \mathtt{z} =  \mathtt{x} + i \mathtt{y}  \in M_d(\C) :  
{ \begin{array}{c}    
\mathtt{x}, \mathtt{y} \in \Herm_d(\C) \\  \mathtt{y} \mbox{ is positive definite}    \end{array} }  \right\}  \\
 & =   \left\{  \mathtt{z} \in M_d(\C) :   \frac{  \mathtt{z} -  {}^\dagger \mathtt{z}  }{2i}      \mbox{ is positive definite}   \right\}
\end{align*}
by the same formula as in the Siegel case: if 
\begin{equation*}
\gamma = \begin{pmatrix}  A & B \\ C & D  \end{pmatrix}  \in \mathrm{U}_{d,d}(\R),
\end{equation*}
then  $\gamma \mathtt{z} = (A\mathtt{z}+B)(C\mathtt{z}+D)^{-1}.$  
Define a function 
\[
j  ( \gamma , \mathtt{z})  =  \det(C \mathtt{z} + D )  
\]
of $\gamma \in \mathrm{U}_{d,d}(\R)$ and $\mathtt{z}\in \mathcal{H}_d$.
This  satisfies the same cocycle relation \eqref{siegel automorphy} as in the Siegel case,
and is the factor of automorphy used to define Hermitian modular forms.

There is an injection $\mathrm{U}(d)  \times \mathrm{U}(d)  \to \mathrm{U}_{d,d}(\R)$ sending a pair $(k_1,k_2)$ to 
\[
  [k_1 , k_2] \define   \frac{1}{2}  \begin{pmatrix} k_1 + k_2  & -i k_1 + i k_2  \\ ik_1 - i k_2 & k_1 +k_2  \end{pmatrix}  \in  \mathrm{U}_{d,d}(\R) ,
\]
where $\mathrm{U}(d) = \{ g \in  \GL_d(\C) : {}^\dagger g = g ^{-1}  \}$ is the usual  unitary group.
The stabilizer   of   $i I_d \in \mathcal{H}_d$ is  the maximal compact subgroup 
\begin{equation}\label{unitary max compact}
K  =  \big\{  [ k_1 , k_2 ] : k_1,k_2 \in \mathrm{U}(d) \big\}  \subset \mathrm{U}_{d,d}(\R).
\end{equation}

Fix integers $m_1$ and $m_2$, and 
suppose $F$ is a function on   $\mathrm{U}_{d,d}(\Q) \backslash \mathrm{U}_{d,d}(\A)$  satisfying 
\[
F(g k ) =     \det(k_1)^{  m_1 }   \det(k_2)^{ m_2   }    F(g)  
\]
for all $g \in \mathrm{U}_{d,d}(\A)$ and $k=[k_1,k_2]  \in K$ (regarded as an element of $\mathrm{U}_{d,d}(\A)$ with trivial nonarchimedean components).
Suppose also that $F$ is right invariant under some compact open subgroup $U \subset  \mathrm{U}_{d,d} ( \A_f)$, and set 
\[
\Gamma =  \{ \gamma \in \mathrm{U}_{d,d}(\R) :  \exists u \in U \mbox{ for which }
\gamma u \in \mathrm{U}_{d,d}(\Q)  \} .
\]
The product $\gamma u$ on the right hand side is understood inside $\mathrm{U}_{d,d}(\A)$.

We  convert $F$ to a function on the Hermitian half-space as follows: 
Given   $\mathtt{z} = \mathtt{x} + i \mathtt{y}  \in \mathcal{H}_d$,  choose a    factorization  
\[
\mathtt{y} =\alpha \cdot {}^\dagger \alpha
\]
in such a way that  $\alpha \in \GL_d(\C)$ has positive real determinant,   set 
\[
g_\mathtt{z}  \define   \begin{pmatrix} I_d & \mathtt{x} \\ & I_d \end{pmatrix} 
\begin{pmatrix} \alpha &  \\ & {}^\dagger \alpha^{-1}  \end{pmatrix}    \in  \mathrm{U}_{d,d}(\R)  ,
\]
and define 
  \begin{equation}\label{classical unitary}
F(\mathtt{z}) \define    \det(\mathtt{y})^{  \frac{m_2-m_1}{2}  } F( g_\mathtt{z} ) .
\end{equation}
On the right hand side  we are viewing 
 $g_\mathtt{z} \in \mathrm{U}_{d,d}(\R)$ as an element of  $\mathrm{U}_{d,d}(\A)$ with trivial nonarchimedean components.

\begin{remark}\label{rem:su g_z}
Our assumption that $\alpha$ has positive real determinant implies that $\det(g_\mathtt{z})=1$.
\end{remark}

\begin{remark}\label{rem:unitary alpha ambiguity}
Any other choice of $\alpha$ differs from the one we already have by right multiplication by some $k \in \mathrm{SU}(d)$.
This has the effect of multiplying $g_\mathtt{z}$ on the right by $[k,k] \in K$,
and it follows that  $F(\mathtt{z})$  does not depend on the choice of $\alpha$ used  in the definition of $g_\mathtt{z}$.
\end{remark}

\begin{proposition}\label{prop:hermitian transformation}
 For all $\gamma \in  \Gamma  \cap \mathrm{SU}_{d,d}(\R)$, the function \eqref{classical unitary} satisfies  the transformation law
\[
F(\gamma \mathtt{z} ) =    j  ( \gamma ,\mathtt{z})^{m_1-m_2}  F(\mathtt{z})
\]
of a Hermitian modular form of weight $m_1 - m_2$.
\end{proposition}

\begin{proof}
Fix  $\gamma \in \mathrm{U}_{d,d}(\R)$.
We must have
\[
g_\mathtt{z}^{-1}  \gamma^{-1} g_{\gamma \mathtt{z}} = [ k_1 , k_2] 
\]
for some  $k_1,k_2 \in \mathrm{U}(d)$, because the left hand side fixes $i I_d \in \mathcal{H}_d$, hence lies in $K$. 
 If  we define 
\[
c(\gamma , \mathtt{z}) = \det(k_1)^{  m_1 } \det(k_2)^{ m_2 } , 
\]  
then  the same reasoning as in the proof of  Proposition \ref{prop:siegel transformation} shows that
\begin{equation}\label{unitary circle transform}
F( g_{ \gamma \mathtt{z} } ) = c(\gamma,\mathtt{z}) F(g_\mathtt{z}) .
\end{equation}

To relate $c(\gamma,\mathtt{z})$  to $j(\gamma, \mathtt{z})$, first note that
\begin{equation}\label{C conjugation}
[k_1,k_2] = C \begin{pmatrix} k_1 \\ & k_2 \end{pmatrix} C^{-1}
\quad \mbox{in which} \quad C = \frac{1}{\sqrt{2}} \begin{pmatrix}   I_d & I_d \\ i I_d & -i I_d  \end{pmatrix}   ,
\end{equation}
and hence    $\det( [ k_1,k_2] ) = \det(k_1) \det(k_2)$.
Combining this with Remark  \ref{rem:su g_z} shows that
\[
\det(\gamma)^{-1} = \det( [ k_1,k_2] ) = \det(k_1) \det(k_2) ,
\]
and we deduce
\begin{equation}\label{alt c}
c(\gamma , \mathtt{z}) = \det(\gamma)^{ - m_1 } \det(k_2)^{ m_2 -m_1 } .
\end{equation}

Now write  $\mathtt{y}(\mathtt{z}) = (2i)^{-1} ( \mathtt{z} - {}^\dagger\mathtt{z} )$ for the $\mathtt{y}$-part of $\mathtt{z}$.
Using the factorization $g_{ \gamma \mathtt{z}} = \gamma g_\mathtt{z} [ k_1, k_2 ]$ and the cocycle relation for $j$, one finds 
\[
j( g_{ \gamma \mathtt{z}}  , i I_d)  = j( \gamma ,\mathtt{z} ) \cdot  j( g_\mathtt{z} , i I_d )   \cdot j( [ k_1,k_2] , i I_d ).
\]
Directly from the definition of $j$, we have the equalities 
\begin{align*}
j( g_{ \gamma \mathtt{z}}   ,  i I_d)  & = \det( \mathtt{y}(\gamma \mathtt{z} ))^{ -1/2}  \\
j( g_\mathtt{z} , i I_d ) &  = \det( \mathtt{y}(\mathtt{z}) )^{-1/2}  \\
j( [ k_1,k_2] , i I_d )  & = \det(k_2) ,
\end{align*}
from which we deduce
\begin{align*}
\det(k_2)     =  j( \gamma ,  \mathtt{z} )^{-1}  
\left( \frac{ \det  \mathtt{y} (\mathtt{z})  }{ \det  \mathtt{y} ( \gamma \mathtt{z} ) }   \right)^{1/2} .
\end{align*} 

Plugging this last expression for $\det(k_2)$ into \eqref{alt c} shows that
\[
j( \gamma ,  \mathtt{z} )^{m_1-m_2}  
= c(\gamma,\mathtt{z})   \det(\gamma)^{  m_1 }  
 \left( \frac{ \det  \mathtt{y} (\mathtt{z})  }{ \det  \mathtt{y} ( \gamma \mathtt{z} ) }   \right)^{ \frac{m_1-m_2}{2} }   ,
\]
and from this and \eqref{unitary circle transform} we deduce that   \eqref{classical unitary} satisfies  
\[
F(\gamma \mathtt{z} ) =  \det( \gamma )^{-m_1}   j  ( \gamma ,\mathtt{z})^{m_1-m_2}  F(\mathtt{z})
\]
for all $\gamma \in  \Gamma$, proving (a stronger version of) the claim.
\end{proof}


\subsection{Hermitian theta functions}


We continue to work  with a quadratic imaginary extension $E$  of  the global field $F=\Q$, and for our additive character
\[
\psi : \Q \backslash \A \to \C^1
\]
take the unique one with $\psi_\infty(x) = e^{2\pi i x}$.
Identify $\mathrm{U}_{d,d}(\Q) \iso \mathrm{U}(W)$ in the way described  in \S \ref{ss:uu setup}.

Recalling the character $\eta$ chosen in  \eqref{eta def},  define an integer $r_\eta \in \Z$  by 
\begin{equation}\label{r_eta}
\eta_\infty (z)  = ( z/ |z| )^{r_\eta} 
\end{equation}
for all $z \in E_\R^\times = \C^\times$.
It satisfies  $r_\eta  \equiv \dim_E(V) \pmod{2}$.

Fix  another   integer $m \equiv \dim_E(V)  \pmod{2}$, and  suppose  
\[
\varphi = \varphi_\R \otimes \varphi_f \in S(V_\A^d)
\]
 is a Schwartz function whose archimedean component  transforms under the Weil representation (Definition \ref{def:unitary-unitary weil}) as  
\begin{equation}\label{unitary-unitary K weight}
\omega^\eta_{V,\psi}( k) \varphi_\R 
=   \det(k_1)^{  \frac{ m  - r_\eta}{2}  }  \cdot \det(k_2)^{ \frac{-m - r_\eta}{2} }  \cdot  \varphi_\R,
\end{equation}
for all $k=[k_1,k_2]$ in the  subgroup  $K \subset  \mathrm{U}_{d,d}(\R)$.

The adelic theta function on $\mathrm{U}_{d,d}(\Q) \backslash \mathrm{U}_{d,d}(\A)$ from \eqref{adelic hermitian theta}  satisfies
\[
\theta^\eta_{V,\psi}( g k , \varphi) =
\det(k_1)^{  \frac{ m  - r_\eta}{2}  }  
\cdot  \det(k_2)^{ \frac{-m - r_\eta}{2} }  
 \cdot \theta^\eta_{V,\psi}( g   ,  \varphi)
\]
for all $k  \in K$, so we may apply  to it the de-adelization construction \eqref{classical unitary}  to obtain  a function 
\begin{equation}\label{classical hermitian theta}
\theta_{V,\psi}(\mathtt{z} , \varphi) \define  
  \det(\mathtt{y})^{- m/2 }  \theta^\eta_{V,\psi}( g_\mathtt{z} , \varphi) 
\end{equation}
of the variable $\mathtt{z}  = \mathtt{x} + i \mathtt{y} \in \mathcal{H}_d$, satisfying
(by Proposition \ref{prop:hermitian transformation}) the transformation law
\[
\theta_{V,\psi}(\gamma \mathtt{z} ,  \varphi) 
=  j( \gamma ,\mathtt{z})^m \theta_{V,\psi} (\mathtt{z} ,  \varphi)
\]
of a Hermitian modular form of weight $m$, 
   for all $\gamma$ in some  arithmetic  subgroup of $\mathrm{SU}_{d,d}(\Q)$.    
The following calculation of its Fourier expansion shows that this function is actually independent of the choice of character $\eta$ from \eqref{eta def}, explaining its omission  from the notation in \eqref{classical hermitian theta}.

\begin{proposition}\label{prop:hermitian theta expansion}
The Fourier expansion of  \eqref{classical hermitian theta}  is
\[
\theta_{V,\psi} (\mathtt{z} , \varphi ) 
= \sum_{ T \in \Herm_d( E )} R_T( \mathtt{y}  , \varphi )
  \cdot  e^{2 \pi i \mathrm{Tr}( T \mathtt{z} ) }  , 
\]
where  the coefficients are
\[
R_T( \mathtt{y}  , \varphi )   = 
\det( \mathtt{y} )^{- \frac{m}{2} + \frac{1}{2} \dim_E(V)  }  
  \sum_{  \substack{  v \in V^d \\ H (v) = T } } \varphi_\R( v  \alpha)  \varphi_f(v)  \, e^{2\pi  \mathrm{Tr}( H(v  \alpha )) } 
\]
for any $\alpha \in \GL_d(\C)$ of positive real determinant satisfying   $\mathtt{y} =\alpha \cdot {}^\dagger \alpha$. 
In forming the product $v \alpha$, we are viewing $v \in V_\R^d$ as a row vector, and viewing $V_\R$ as a vector space over $E_\R = \C$.
\end{proposition}

\begin{proof}
This is entirely similar  to the proof of  Proposition \ref{prop:classical theta expansion}, using the formulas of Proposition \ref{prop:unitary-unitary weil formulas} for the Weil representation  in place of those of Proposition \ref{prop:symplectic-orthogonal weil formulas}.
\end{proof}

We now use the discussion above to  construct  Hermitian analogues of the classical theta functions of Jacobi and Siegel.  The following result can be found in, for example,  \cite{CoRes} and \cite[Appendix A7]{ShimuraEisenstein}.

\begin{theorem}\label{thm:hermitian theta modularity}
Assume that $V$ is positive definite.  
For any $\varphi_f \in S(V^d_{\A_f})$, the function  
\[
\vartheta (\mathtt{z} , \varphi_f ) = 
\sum_{ T \in \Herm_d( E )}  \left(   \sum_{ \substack{ v \in V^d \\ H(v) = T } } \varphi_f(v)   \right)  
e^{ 2  \pi i \mathrm{Tr}( T \mathtt{z} ) } 
\]
is a holomorphic Hermitian modular form on $\mathcal{H}_d$ of weight  $ \dim_E(V)$.
Its level, an arithmetic subgroup  of $\mathrm{SU}_{d,d}(\Q)$, depends on  $\varphi_f$.
\end{theorem}

\begin{proof}
By the same reasoning as in the proof of Theorem \ref{thm:classical theta modularity}, this follows from Proposition \ref{prop:hermitian theta expansion} once we  show that the  standard Gaussian
\[
\varphi^\circ_\R(v)  
\define
 e^{-2\pi  \mathrm{Tr}( H (v  )) }   =  \prod_{j=1}^d e^{- 2\pi   h(v_j  , v_j  ) } 
\]
on $V^d_\R$  satisfies the transformation law \eqref{unitary-unitary K weight} with  
$m =  \dim_E(V)$.

Using  the basis $e_1,\ldots, e_d,f_1,\ldots, f_d \in W$, define  $J \in \mathrm{U}(W_\R)$  exactly as in  \eqref{standard J}.   
The centralizer of $J$ is the maximal compact subgroup $K$ from \eqref{unitary max compact}, which is also the stabilizer of the  eigenspace decomposition 
\[
W_\R  = W_\R^{J = i }  \oplus  W_\R^{J = - i }  .
\]
The two eigenspaces are, recalling \eqref{C conjugation}, 
\begin{align*}
W_\R^{J=i}  & = \mathrm{Span}_\C\{ e_\ell - i f_\ell  : 1 \le \ell \le d \} 
= C \cdot \mathrm{Span}_\C\{ f_\ell : 1 \le \ell \le d   \} \\
W_\R^{J=- i}  & = \mathrm{Span}_\C\{ e_\ell + i f_\ell : 1 \le \ell \le d\} 
= C \cdot \mathrm{Span}_\C\{ e_\ell  : 1 \le \ell \le d  \} , 
\end{align*} 
and   $K$ is the direct product of its two subgroups
\begin{align}\label{K1K2}
K_1  &\define \{ [ k_1 , I_d ]  \in K \}   = \mathrm{ker} \big(  K \to \GL( W_\R^{J= - i}  )  \big) \\
K_2 & \define \{ [  I_d , k_2 ]  \in K \}   = \mathrm{ker} \big(  K  \to \GL( W_\R^{J=   i}  )  \big). \nonumber
\end{align}

The above $J$ determines a complex structure 
$
\bm{J} = J  \otimes   \mathrm{id}_V
$
 on $\bm{W}=W\otimes_E V$.  Let  $K_{\bm{J}}  \subset  \Sp(\bm{W}_\R)$ be its  centralizer, which  is also equal to the stabilizer of the eigenspace decomposition
\begin{equation}\label{big eigenspaces}
\bm{W}_ \R =   \bm{W}_\R^{ \bm{J} = i } \oplus \bm{W}_\R^{ \bm{J} = -i } = 
 ( W_\R^{ J =  i } \otimes_E V   )  \oplus ( W_\R^{ J = - i }   \otimes_E V )   .
\end{equation}
The vacuum vector  $\varphi_{\bm{J}}^\circ \in S(\bm{Y}_\R) \iso S(V_\R^d)$  from Definition \ref{def:vacuum}  associated to this choice of $\bm{J}$ is precisely the standard Gaussian defined above.

In the  commutative diagram
\[
\begin{tikzcd}
&  {  \Mp(\bm{W}_\R)_\psi  }  \ar[d] \\
{   \mathrm{U} (W_\R) }  \ar[r , "i_V" ' ]  \ar[ur  ,   "i^\eta_{V,\psi}"  ]    &  {   \Sp(\bm{W}_\R)  } 
\end{tikzcd}
\]
from  \eqref{unitary-unitary local splitting diagram},   the  horizontal arrow takes $K$ into $K_{\bm{J}}$.  
  By Proposition \ref{prop:vacuum eigenvector}, the preimage of $K_{\bm{J}}$ under the right vertical arrow acts on the vacuum vector through a character, and hence the compact subgroup $K \subset \mathrm{U}(W_\R)$ acts on the standard Gaussian via the pullback of this character along $i^\eta_{V,\psi}$.
This pullback  is what we must compute.

\begin{lemma}
The restriction to $K$ of the character $\chi_{\bm{J}} :  K_{\bm{J}}\to \C^1 $ from  \eqref{chi_J}  is 
\begin{equation}\label{unitary pullback character}
 \left( \chi_{\bm{J}}  \circ  i_V \right) ( [ k_1,k_2] ) =   \det\nolimits_W([ k_1, I_d ] )^{-m}  \cdot \det\nolimits_W( [ I_d , k_2])^{m} .
\end{equation}
Equivalently, recalling the distinction between matrix determinants and determinants of operators on $W$, 
\[
 \left( \chi_{\bm{J}}  \circ  i_V \right) ( [ k_1,k_2] ) =   \det(k_1)^{m}  \cdot \det(k_2)^{-m} .
\]
\end{lemma} 

\begin{proof}
Recall that in  \eqref{chi_J} we defined $\chi_{\bm{J}} : K_{\bm{J}} \to \C^1$ as the composition
\[
K_{\bm{J}} \to \GL( \bm{W}_ \R  ) \map{ \det_{\bm{J}}^{-1} } \C^\times .
\]
The notation $\det_{\bm{J}}$ means the determinant is being computed  using  the complex structure $\bm{J}$ (and \emph{not} the action of $E_\R=\C$) to regard $\bm{W}_\R$ as a complex vector space.
Using   \eqref{K1K2} and \eqref{big eigenspaces}, we see that  the character  $\chi_{\bm{J}} \circ i_V$ is equal to the composition 
\[
K \to   \GL( \bm{W}_\R^{ \bm{J} = i} ) \times \GL(\bm{W}_\R^{ \bm{J} =- i} ) \map{  ( k_1,k_2) \mapsto    \det_{\C}( k_1)^{-1}      \overline{ \det_{ \C} ( k_2) ^{-1}  }  }  \C^\times,
\]
where $m=\dim_E(V)$, and the determinants in the definition of the second arrow are with respect to the actions of $k_1$ and $k_2$ on 
\[
\bm{W}_\R^{ \bm{J} = \pm i } = W_\R^{ J = \pm i } \otimes_E V,
\]
which is regarded as a complex vector space using the action of $E_\R = \C$  (and \emph{not}   the action of $J$). 
This change of complex structure is why there is a complex conjugation on the determinant of $k_2$.  
The claim follows immediately.
 \end{proof}

Let $\mathfrak{k} \subset \mathfrak{u}(W_\R)$ be the Lie algebra of $K   \subset \mathrm{U}(W_\R)$, and let 
$\mathfrak{k}_{\bm{J}} \subset \mathfrak{sp}(\bm{W}_\R)$ be the Lie algebra of  $K_{\bm{J}}  \subset  \Sp(\bm{W}_\R) $.
As in the proof of Theorem \ref{thm:classical theta modularity}, we have a diagram
\begin{equation}\label{unitary-unitary lie diagram}
\begin{tikzcd}
{  \mathfrak{k} }  \ar[r] \ar[d] &  {    \mathfrak{k}_{ \bm{J} }   \oplus    i \R }  \ar[d]  \\
  { \mathfrak{u}(W_\R) }  \ar[d, equal ] \ar[r, " i^\eta_{V,\psi} "]    &  {   \mathfrak{sp}(\bm{W}_\R)  \oplus  i\R  }  \ar[d]  \\
 { \mathfrak{u}(W_\R) } \ar[r , "i_V" ]   &   {   \mathfrak{sp}(\bm{W}_\R)  }     ,
\end{tikzcd}
\end{equation}
but  the relation between the  horizontal arrows is not so simple as in \eqref{orthogonal-symplectic lie diagram}.

  \begin{lemma}\label{lem:kinf correction}
  The horizontal arrows in \eqref{unitary-unitary lie diagram} are related by 
  \[
i^\eta_{V,\psi}(Z) = \left(   i_V(Z) ,  \frac{ r_\eta }{2} \cdot  \mathrm{Tr}_W(Z)  \right),
\]
where  $\mathrm{Tr}_W(Z) \in i \R$ is the trace of $Z \in \mathfrak{u}(W_\R)$ acting on $W_\R$, 
regarded as a vector space over $E_\R=\C$, and $r_\eta$ is the integer defined by \eqref{r_eta}.
  \end{lemma}
  
  \begin{proof}
Every Lie algebra map  $\mathfrak{u}(W_\R) \to i \R $ is a scalar multiple of the trace map $\mathrm{Tr}_W : \mathfrak{u}(W_\R) \to i \R$, and it follows that the horizontal arrows are related by 
   \[
i^\eta_{V,\psi}(Z) = \left(   i_V(Z)  , c    \mathrm{Tr}_W(Z)   \right) 
\]
for some constant $c\in \R$.  

 To compute the constant, we may consider only those $Z$ in the Lie algebra of the center of $\mathrm{U}(W_\R)$.  The center is the circle group  $\C^1 \subset \C^\times = E_\R^\times$ acting on $W_\R$ by scalar multiplication.
 It follows from   \eqref{unitary parabolic beta} that for any $z \in \C^1$, the element $z \cdot  \mathrm{Id}_W  \in   \mathrm{U}(W_\R)$ satisfies 
\[
\beta_\R( z \cdot   \mathrm{Id}_W ) = \eta_\infty ( z )^d  = z^{  d \cdot  r_\eta} .
\]
Going back to \eqref{unitary-unitary local splitting diagram}, we see that the homomorphism
\[
 \mathrm{U}(W_\R)   \map{   i^\eta_{V,\psi}  }  \Mp(\bm{W}_\R)_\psi  
\]
satisfies
\[
i^\eta_{V,\psi}(z  \cdot \mathrm{Id}_W ) = \big(  i_V(z \cdot  \mathrm{Id}_W) , z^{   d \cdot r_\eta}  \big), 
\]
and differentiating this shows that  $2 c= r_\eta$.
  \end{proof}

   The character through which  $K$ acts on the standard Gaussian differentiates to the composition of the top horizontal arrow  in \eqref{unitary-unitary lie diagram}, which is known by  Lemma \ref{lem:kinf correction},   
   with the linear functional  on $ \mathfrak{k}_{\bm{J}}  \oplus i \R$ appearing in  Proposition \ref{prop:vacuum eigenvector}.  
More precisely,  $K=K_1 \times K_2$ acts on the standard Gaussian through the character whose infinitesimal form is the linear functional   $\mathfrak{k}= \mathfrak{k}_1 \oplus \mathfrak{k}_2 \to i\R $ given by 
   \begin{equation}\label{unitary vacuum inf character}
Z_1+Z_2  \mapsto \frac{1}{2}   \cdot   ( \chi^\mathrm{inf} _{\bm{J}}  \circ  i_V) ( Z_1+Z_2 ) + \frac{r_\eta}{2}   \cdot  \mathrm{Tr}_W(Z_1+Z_2) .
\end{equation}

Differentiating  \eqref{unitary pullback character} shows that 
\[
( \chi^\mathrm{inf}_{\bm{J}}  \circ  i_V ) (Z_1 + Z_2 ) = -m \mathrm{Tr}_W(Z_1) +  m \mathrm{Tr}_W(Z_2),
\]
which  allows us to rewrite \eqref{unitary vacuum inf character} as 
\[
(Z_1 , Z_2)  \mapsto   \left(  \frac{ - m + r_\eta }{2}   \right) \mathrm{Tr}_W(Z_1)   +  \left(  \frac{ m + r_\eta }{2}   \right) \mathrm{Tr}_W(Z_2) . 
\]
Recalling  Remark \ref{rem:banana peel}, this  is the infinitesimal form of the character 
\[
[ k_1 , k_2 ] \mapsto \det(k_1)^{ \frac{ m -  r_\eta }{2} } \det(k_2)^{ \frac{  - m -  r_\eta }{2} } 
\]
of $K$.  This last character  is therefore the character by which $K$ acts on the standard Gaussian, completing the proof of Theorem \ref{thm:hermitian theta modularity}. 
\end{proof}


\subsection{Polynomial weights}
\label{ss:unitary poly weights}


With the goal of generalizing Theorem \ref{thm:hermitian theta modularity}, we assume throughout \S \ref{ss:unitary poly weights} that $(V,h)$ is a positive definite Hermitian space over our quadratic imaginary field $E$.  Regard $V_\R$ as a vector space over $E_\R = \C$, and fix a basis 
 $
 v_1,\ldots, v_n \in V_\R
 $
satisfying  $h(v_j,v_k) =\delta_{jk}$.

Writing vectors in $V_\R$ as $z_1 v_1+ \cdots + z_n v_n$ with $z_1,\ldots, z_n \in \C$, 
   we regard functions on the real vector space $V_\R$ as functions of the  variables 
   \[
   z_\alpha=x_\alpha+ i y_\alpha 
   \qquad \mbox{and} \qquad  
   \overline{z}_\alpha=x_\alpha -  i y_\alpha
   \]
    with $1\le \alpha \le n$.
   For an integer $1\le j \le d$, define a coordinate function  $z_{\alpha j}$ as  the composition
   \[
   V_\R^d \to   V_\R \map{z_\alpha}  \C,
   \]
   where the first arrow  is projection to the $j^\mathrm{th}$ factor.
In this way we regard functions on $V_\R^d$ as functions of the  variables 
\[
z_{\alpha j} = x_{\alpha j}+i y_{\alpha j}
\qquad \mbox{and} \qquad 
\overline{z}_{\alpha j} = x_{\alpha j} - i y_{\alpha j} .
\] 
Following the usual conventions,  define differential operators
\[
 \frac{\partial  }{ \partial z _{\alpha j }}  = 
\frac{1}{2} \left(   \frac{\partial  }{ \partial x _{\alpha j }}  - i  \frac{\partial }{ \partial y _{\alpha j }}  \right)
  \qquad   \mbox{and}  \qquad 
   \frac{\partial  }{ \partial \overline{z} _{\alpha j }}  = 
   \frac{1}{2} \left(     \frac{\partial  }{ \partial x _{\alpha j }}  +  i \frac{\partial }{ \partial y _{\alpha j }}  \right) 
\]
on the space of smooth functions on $V_\R^d$.

For each $1 \le j \le d$, define  a differential operator 
\[
\Delta_j    \define \sum_\alpha  \left(   \frac{\partial^2 }{ \partial x^2 _{\alpha j }}  +   \frac{\partial^2 }{ \partial y^2 _{\alpha j }}   \right) 
=  4  \sum_\alpha \frac{\partial^2}{\partial z_{\alpha j }    \partial \overline{z}_{\alpha j } } 
\]
on the space of polynomial functions on the real vector space $V^d_\R$, by which we always mean  $\C$-valued polynomials  in the variables $z_{\alpha j}$ and $\overline{z}_{\alpha j}$.
As in \eqref{Psharp}, set $\Delta = \Delta_1+\cdots+\Delta_d$, and  for any such polynomial   $P$ define
\[
P^\sharp =  \exp\left( -\frac{\Delta}{16\pi} \right)  P  .
\]

\begin{theorem}\label{thm:main unitary}
Fix  an integer $N \in \Z$, and  suppose $P$ is a polynomial function on the real vector space  $V^d_\R$ satisfying 
\begin{equation}\label{unitary P homogeneity}
P( v g ) = | \det(g)| ^{2N} P(v) 
\end{equation}
 for  all $v\in V_\R^d$ and $g \in \GL_d(\C)$.    
 For any Schwartz function $\varphi_f \in S(V^d_{\A_f})$,   the  function 
\[
\vartheta_P (\mathtt{z} ,  \varphi_f ) = 
\det(\mathtt{y})^{-N} 
\sum_{ T \in \Herm_d( E )}  \left(   \sum_{ \substack{ v \in V^d \\ H(v) = T } }   P^\sharp(v \alpha ) \varphi_f(v  )   \right)   \, e^{2 \pi i \mathrm{Tr}( T \mathtt{z} ) } 
\]
of the variable $\mathtt{z} =\mathtt{x} + i \mathtt{y}\in  \mathcal{H}_d$ transforms like a Hermitian  modular form of weight $2N+ \dim_E(V)$, under the action of some arithmetic subgroup of $\mathrm{SU}_{d,d}(\Q)$.
Here  $\alpha \in \GL_d(\C)$  is any matrix of positive real determinant such that  $\mathtt{y}=\alpha \cdot {}^\dagger \alpha$.   
\end{theorem}


\begin{remark}
The assumption  \eqref{unitary P homogeneity} implies that $P$ is right  invariant under  the subgroup
$\mathrm{SU}(d) \subset \GL_d(\C)$.
As the differential operator $\Delta$  commutes with the action of this subgroup, 
 the polynomial $P^\sharp$ is also  $\mathrm{SU}(d)$-invariant.
 Combining this with  Remark \ref{rem:unitary alpha ambiguity} shows that 
$P^\sharp (v \alpha)$ is independent of the choice of $\alpha$.
\end{remark}

\begin{remark}\label{rem:shimura polynomials}
If $\Delta P =0$,  then $P^\sharp =P$, and Theorem \ref{thm:main unitary} shows  that 
\[
\vartheta_P (\mathtt{z} , \varphi_f ) = 
\sum_{ T \in \Herm_d( E )}  \left(   \sum_{ \substack{ v \in V^d \\ H(v) = T } }   P(v ) \varphi_f(v  )   \right)   \, e^{2 \pi i \mathrm{Tr}( T \mathtt{z} ) } 
\]
is a holomorphic Hermitian modular form of weight $2N+  \dim_E(V)$.
Similar results have been proved by Shimura  in \cite[Appendix A7]{ShimuraEisenstein} and \cite[Appendix A5]{Shimura2000}, but this precise result is not among them.  
Shimura  works with theta series for Hermitian spaces of arbitrary signature.  
When specialized to the positive definite case that we work in, Shimura's results prove the modularity  only of  theta functions weighted by polynomial functions on $V_\R$ that are either holomorphic or anti-holomorphic.   
Nonconstant polynomials satisfying   \eqref{unitary P homogeneity} do not fall into either class, so it seems that  even the special case $\Delta P=0$ of Theorem \ref{thm:main unitary} has not previously appeared explicitly in the literature.
\end{remark}

We now  begin the  proof of Theorem \ref{thm:main unitary}.
For  $j,k \in \{1,\ldots, d\}$, define differential operators on $S(V_\R^d)$ by 
\begin{align*}
   E^{\FirstIndex}_{jk}   &  \define  \sum_\alpha   z_{\alpha j}     \frac{\partial  }{\partial z_{\alpha k } }  
    & \qquad F_{jk}  &  \define    \frac{1}{2}  \sum_\alpha   \left(    \frac{\partial^2  }{ \partial x_{\alpha j } \partial x_{\alpha k} }    +       \frac{\partial^2  }{  \partial y_{\alpha j } \partial y_{\alpha k}}     \right)  \\ 
    E^{\SecondIndex}_{jk}   &  \define  \sum_\alpha    \overline{z}_{\alpha j}     \frac{\partial  }{\partial  \overline{z}_{\alpha k } }   
& \qquad 
G_{jk}  &  \define   \frac{1}{2}   \sum_\alpha
\left(    \frac{\partial^2  }{ \partial x_{\alpha j } \partial y_{\alpha k} }    -       \frac{\partial^2  }{  \partial y_{\alpha j } \partial x_{\alpha k}}   
\right) .
\end{align*}

\begin{lemma}
The transformation law \eqref{unitary P homogeneity} is equivalent to the system of differential equations
\begin{equation}\label{unitary homo equations}
 E^{\FirstIndex}_{jk}   P  =   N  \delta_{jk} P 
  \qquad \mbox{and} \qquad 
  E^{\SecondIndex}_{jk}  P   =   N  \delta_{jk} P   
\end{equation}
for all $j,k \in \{1,\ldots, d\}$.
\end{lemma}

\begin{proof}
Define a representation   of  $\GL_d(\C)$  on the space of  polynomial functions  on $V_\R^d$  by $( r(g)P)(v) = P(v  g)$, where we regard $V_\R$ as a vector space over $E_\R = \C$.
 Note that this is a representation of $\GL_d(\C)$ regarded as a \emph{real} Lie group; because our polynomial functions are polynomials in both $z_{\alpha j}$ and $\overline{z}_{\alpha j}$, the automorphism $r(g)$  varies smoothly, but not holomorphically, as $g \in \GL_d(\C)$ varies.

 Under the induced  representation of the \emph{real} Lie algebra $\mathfrak{gl}_d(\C)$, the  matrix $X_{jk}\in \mathfrak{gl}_d(\C)$ with a $1$ in the $(j,k)$ entry and $0$'s elsewhere satisfies 
\begin{align*}
r(X_{j k} ) 
&  =  \sum_\alpha \left(  x_{\alpha j} \frac{\partial}{\partial x_{\alpha k}} + y_{\alpha j} \frac{\partial}{\partial y_{\alpha k}}  \right)  \\
r( i X_{jk} ) 
&  =  \sum_\alpha \left(  -y_{\alpha j} \frac{\partial}{\partial x_{\alpha k }} + x_{\alpha j} \frac{\partial}{\partial y_{\alpha k}}  \right) .
\end{align*}
In particular, note that 
\begin{equation}\label{X_jk}
2 E^{\FirstIndex}_{jk}   = r(X_{jk}) - i r( i X_{jk} ) 
\qquad \mbox{and}\qquad 
2 E^{\SecondIndex}_{jk}   = r(X_{jk}) +  i r( i X_{jk} ) .
\end{equation}

Elementary representation theory shows that the transformation law \eqref{unitary P homogeneity} under the Lie group $\GL_d(\C)$ is equivalent to the Lie algebra relations 
\begin{align*}
r(X_{jk}) P  = 2N  \delta_{jk} P  \qquad \mbox{and} \qquad 
r( i X_{j k }) P  =  0  
\end{align*}
for all $j,k \in \{ 1,\ldots, d\}$.  By  \eqref{X_jk} these are equivalent to 
\begin{align*}
 ( E^{\SecondIndex}_{jk}    + E^{\FirstIndex}_{jk} ) P  = 2 N  \delta_{jk} P \qquad \mbox{and} \qquad 
(  E^{\SecondIndex}_{jk}  -  E^{\FirstIndex}_{jk} ) P   =  0  ,
\end{align*}
which are equivalent to \eqref{unitary homo equations}.
\end{proof}

\begin{lemma}\label{lem:unitary commutator}
The system of differential equations \eqref{unitary homo equations} for $P$ is equivalent to the system of differential equations 
\begin{align*}
  \left(         E^{\FirstIndex}_{jk}  +  E^{\FirstIndex}_{kj}    -     \frac{ 1 }{4\pi }      F_{jk}    \right) P^\sharp  
 & =     2 N     \delta_{jk}     P^\sharp        \\ 
 \left(         E^{\SecondIndex}_{jk}   +  E^{\SecondIndex}_{kj}    -    \frac{ 1 }{4\pi }     F_{jk}    \right) P^\sharp  
 & =     2  N   \delta_{jk}    P^\sharp      \\  
\left(       E^{\FirstIndex}_{jk}  -    E^{\FirstIndex}_{kj}          -    \frac{  1    }{4 \pi i }   G_{jk}   \right) P^\sharp   & = 0   \\
\left(       E^{\SecondIndex}_{jk}  -    E^{\SecondIndex}_{kj}         +    \frac{  1    }{4 \pi i }   G_{jk}   \right) P^\sharp   & = 0 \\
\end{align*}
for the polynomial $P^\sharp$.
\end{lemma}

\begin{proof}
Of course this is similar to the proof of Lemma \ref{lem:diffeq conversion}.
Direct calculation proves the commutator relations 
\begin{align*}
\Delta  E^{\FirstIndex}_{jk}  & = E^{\FirstIndex}_{jk} \Delta + 4 \sum_{\alpha } \frac{\partial^2}{ \partial \overline{z} _{\alpha j}   \partial z_{\alpha k } }  \\
\Delta  E^{\SecondIndex}_{jk}   & = E^{\SecondIndex}_{jk} \Delta + 4 \sum_{\alpha } \frac{\partial^2}{  \partial   z_{\alpha j}   \partial  \overline{z}_{\alpha k }   } ,
\end{align*}
from which one deduces
\begin{align*}
\Delta  ( E^{\FirstIndex}_{jk}    +  E^{\FirstIndex}_{kj}  ) 
& =  ( E^{\FirstIndex}_{jk}   +  E^{\FirstIndex}_{kj} )  \Delta + 4 F_{jk}   \\
\Delta  ( E^{\SecondIndex} _{jk}   +  E^{\SecondIndex} _{kj}  )  
& = ( E^{\SecondIndex} _{jk}   +  E^{\SecondIndex} _{kj}  )   \Delta + 4 F_{jk}   
\end{align*}
and
\begin{align*}
\Delta  ( E^{\FirstIndex}_{jk}    -   E^{\FirstIndex}_{kj} ) 
& =  ( E^{\FirstIndex}_{jk}  -  E^{\FirstIndex}_{kj} )  \Delta - 4 i  G_{jk}    \\
\Delta  ( E^{\SecondIndex}_{jk}   -   E^{\SecondIndex}_{kj}  )  & = ( E^{\SecondIndex}_{jk}   -  E^{\SecondIndex}_{kj}  )   \Delta + 4 i  G_{jk} .
\end{align*}
These imply the relations
\begin{align*}
\exp \left( - \frac{ \Delta }{ 16 \pi } \right) \circ   ( E^{\FirstIndex}_{jk}  +  E^{\FirstIndex}_{kj} ) 
& =  \left( E^{\FirstIndex}_{jk}  +  E^{\FirstIndex}_{kj}    - \frac{1}{4\pi}  F_{jk}  \right)  \circ \exp \left( - \frac{ \Delta }{16 \pi} \right)     \\
\exp \left( - \frac{\Delta}{16 \pi} \right) \circ    ( E^{\SecondIndex}_{jk}   +  E^{\SecondIndex}_{kj}  )  & = 
\left( E^{\SecondIndex}_{jk}   +  E^{\SecondIndex}_{kj}     - \frac{1}{4\pi}  F_{jk}  \right)  \circ \exp \left( - \frac{ \Delta }{16 \pi} \right) 
\end{align*}
and
\begin{align*}
\exp \left( - \frac{ \Delta }{ 16 \pi } \right) \circ   ( E^{\FirstIndex}_{jk}  -  E^{\FirstIndex}_{kj} ) 
& =  \left( E^{\FirstIndex}_{jk}  -  E^{\FirstIndex}_{kj}    -  \frac{1}{4\pi i }    G_{jk}  \right)  \circ \exp \left( - \frac{ \Delta }{ 16 \pi} \right)    \\
\exp \left( - \frac{\Delta}{ 16 \pi} \right) \circ    ( E^{\SecondIndex}_{jk}   -  E^{\SecondIndex}_{kj}  )  & = 
\left( E^{\SecondIndex}_{jk}   -  E^{\SecondIndex}_{kj}     +  \frac{1}{4\pi i }   G_{jk}  \right)  \circ \exp \left( - \frac{ \Delta }{16 \pi} \right) , 
\end{align*}
from which  the  claim follows easily.
\end{proof}

Define an  $\R$-linear map
 \begin{equation}\label{herm2u}
W_\R\otimes_\R W_\R \to \mathfrak{u}( W_\R ) 
\end{equation}
 by sending  $w_1  \otimes w_2$ to the endomorphism 
$
w  \mapsto   \langle w_1 , w  \rangle w_2 + \langle w_2   ,w  \rangle  w_1   
$
of $W_\R$.   
  If we define a real vector space  $\Herm^2(W_\R)$  as the quotient of $W_\R\otimes_\R W_\R$ by the span of all tensors of the form $w_1 \otimes w_2 - w_2 \otimes w_1$ and $(\alpha w_1) \otimes w_2 - w_1\otimes (\overline{\alpha} w_2)$
with $w_1 , w_2 \in W_\R$ and $\alpha \in E_\R$, then the above map factors through an  $\R$-linear isomorphism
\[
 \Herm^2(  W_\R )  \iso \mathfrak{u}( W_\R ) .
\]
Compare with  \eqref{sym2sp}.

Using \eqref{herm2u} and the action of $i \in \C=E_\R$ on  $W_\R$, we regard the tensors 
\begin{align} \label{XsYs}
X^{\FirstIndex}_{ j k }  & = (  e_j -i f_j  ) \otimes (  e_k  -i f_k  )  \\
 Y^{\FirstIndex}_{ j k }   & = (  e_j -i f_j  )   \otimes   (  i e_k +  f_k  )   \nonumber  \\
 X^{\SecondIndex}_{ j k }  & =  ( e_j + i f_j  )  \otimes ( e_k + i f_k )   \nonumber    \\
  Y^{\SecondIndex}_{ j k }  & = ( e_j + i f_j  )  \otimes  ( i e_k  -  f_k) ,  \nonumber 
\end{align}
 as elements of $\mathfrak{u}( W_\R )$,  for all $j,k \in \{ 1, \ldots, d\}$.
Let us compute their images under the composition 
 \[
 \mathfrak{u}(W_\R) \map{ i_V }    \mathfrak{sp}(\bm{W}_\R )  
  \map{\omega^\mathrm{inf}   }  \End_\C(S(V_\R^d ))
 \]
of   \eqref{unitary i_V} with the infinitesimal Weil representation of  Definition \ref{def:Lie Schrodinger}.

\begin{lemma}\label{lem:unitary XY operators}
 The above vectors act on $S(V_\R^d)$  as 
\begin{align*}
 \omega^\mathrm{inf}  ( i_V(  X^{\FirstIndex}_{jk})) 
   & = 
    \frac{1}{4 \pi i } 
 \sum_\alpha   \left(4\pi i x_{\alpha j} +  \frac{\partial}{\partial y_{\alpha j } }   \right)   
 \left(4\pi i x_{\alpha k} +  \frac{\partial}{\partial y_{\alpha k } }   \right)     \\
 & \qquad  +
 \frac{1}{4 \pi i } 
   \sum_\alpha    \left(4\pi i y_{\alpha j} -   \frac{\partial}{\partial x_{\alpha j } }   \right)   
 \left(4\pi i y_{\alpha k} -  \frac{\partial}{\partial x_{\alpha k } }   \right)  \\
 \omega^\mathrm{inf}  (  i_V( X_{jk}^{\SecondIndex} ) ) 
   & = 
    \frac{1}{4 \pi i } 
 \sum_\alpha   \left(4\pi i x_{\alpha j} -  \frac{\partial}{\partial y_{\alpha j } }   \right)   
 \left(4\pi i x_{\alpha k} -  \frac{\partial}{\partial y_{\alpha k } }   \right)     \\
 & \qquad  +
 \frac{1}{4 \pi i } 
   \sum_\alpha    \left(4\pi i y_{\alpha j} +   \frac{\partial}{\partial x_{\alpha j } }   \right)  
 \left(4\pi i y_{\alpha k}  +   \frac{\partial}{\partial x_{\alpha k } }   \right)  
\end{align*}
and 
\begin{align*}
 \omega^\mathrm{inf}   (  i_V(  Y^{\FirstIndex}_{jk} ) ) 
 &=   
  \frac{1}{4 \pi i } 
 \sum_\alpha   \left(4\pi i x_{\alpha j} +  \frac{\partial}{\partial y_{\alpha j } }   \right)   
  \left(4\pi i y_{\alpha k} -  \frac{\partial}{\partial x_{\alpha k } }   \right)    \\
 & \quad  -
 \frac{1}{4 \pi i } 
   \sum_\alpha    \left(4\pi i y_{\alpha j} -   \frac{\partial}{\partial x_{\alpha j } }   \right) 
    \left(4\pi i x_{\alpha k} +  \frac{\partial}{\partial y_{\alpha k } }   \right)   \\
    & \quad - 2  \dim_E(V)       \delta_{jk} \\
     \omega^\mathrm{inf}(  i_V( Y^{\SecondIndex}_{jk} ) ) 
 &=  
  \frac{1}{4 \pi i } 
 \sum_\alpha   \left(4\pi i x_{\alpha j} -  \frac{\partial}{\partial y_{\alpha j } }   \right)   
  \left(4\pi i y_{\alpha k} +  \frac{\partial}{\partial x_{\alpha k } }   \right)    \\
 & \quad  -
 \frac{1}{4 \pi i } 
   \sum_\alpha    \left(4\pi i y_{\alpha j} +    \frac{\partial}{\partial x_{\alpha j } }   \right)  
    \left(4\pi i x_{\alpha k} -  \frac{\partial}{\partial y_{\alpha k } }   \right)  \\
   & \quad +  2   \dim_E(V)      \delta_{jk}   .
 \end{align*}
\end{lemma}

\begin{proof}
Elementary linear algebra shows that the composition
\[
W_\R \otimes_\R W_\R \map{\eqref{herm2u}}  
 \mathfrak{u}(W_\R) \map{i_V} \mathfrak{sp}(\bm{W}_\R) 
  \stackrel{ \eqref{sym2sp} }{ \iso }  \Sym^2 (  \bm{W}_\R ) 
\]
 sends
\[
w_1   \otimes w_2  \mapsto 
 \frac{1}{2} \sum_\alpha   (w_1  \otimes  v_\alpha ) ( w_2\otimes v_\alpha)   
+ 
 \frac{1}{2}  \sum_\alpha     (i w_1 \otimes v_\alpha  ) (i w_2  \otimes v_\alpha  )  .
\]
Tracing through the constructions leading to Definition \ref{def:Lie Schrodinger}, it follows that if $Z \in \mathfrak{u}(W_\R)$ is the image of  $w_1 \otimes w_2$ under \eqref{herm2u}, then 
\begin{align*}
\omega^\mathrm{inf}  ( i_V( Z ) )  &=     \sum_\alpha 
  \left(  \frac{1}{4 \pi i }
   \rho   ( w_1   \otimes  v_\alpha) \circ\rho ( w_2  \otimes v_\alpha  )  - \frac{1}{4}  \mathrm{Tr}_{\C/\R}     \langle w_1 ,  w_2  \rangle
  \right)    \\
  & \quad  +
     \sum_\alpha  \left(  \frac{1}{4 \pi i }
   \rho  ( i w_1  \otimes  v_\alpha ) \circ \rho ( i w_2\otimes v_\alpha )    - \frac{1}{4}  \mathrm{Tr}_{\C/\R}     \langle w_1 ,  w_2  \rangle
  \right)      .
 \end{align*}
In this last expression $\mathcal{W}$ is the  Weyl algebra (Definition \ref{def:weyl algebra}) of the real symplectic space $\bm{W}_\R$, 
we are identifying  the vectors $w_1   \otimes  v_\alpha$,  $i w_2   \otimes  v_\alpha$, and so on in $\bm{W}_\R = W_\R \otimes_{E_\R}  V_\R$ with   their images under the composition
 \begin{equation}\label{screwy unitary composition}
  \bm{W}_\R \to \bm{W}_\C^{\otimes} \to \mathcal{W}. 
  \end{equation}
  The action  $\rho$  of the Weyl algebra on $S(\bm{Y}_\R)=S( V_\R^d  )$ was defined in 
   \eqref{Weyl module}.

Caution is needed, as the real vector space $\bm{W}_\C = \bm{W}_\R\otimes_\R \C$ has two complex structures: the action of $\C  = E_\R$ on the first tensor factor, and the natural action of $\C$ on the second  factor.
The tensor power in $\bm{W}_\C^\otimes$ is over $\C$, acting on the second tensor factor.
   The first arrow in \eqref{screwy unitary composition} is $\C$-linear with respect to the action on the first tensor factor.  The second is $\C$-linear with respect to the action on the second tensor factor.     
   The composition \eqref{screwy unitary composition}  is only $\R$-linear, and  
 $\rho  (   iw  \otimes  v_\alpha ) \neq i \cdot \rho  (   w  \otimes  v_\alpha )$.

Using the above expression for $\omega^\mathrm{inf}  ( i_V( Z ) )$, the rest of the proof is an explicit calculation using the  formulas
\begin{align*}
\rho   ( e_j \otimes v_\alpha )     & =   4 \pi i  \cdot x_{\alpha j}  & \qquad 
\rho  ( i  e_j \otimes v_\alpha    )    & =   4 \pi i  \cdot  y_{\alpha j} \\
\rho  ( f_j  \otimes v_\alpha )   & = - \frac{\partial}{\partial x_{\alpha j } }    & \qquad 
\rho (  i  f_j   \otimes v_\alpha )   & = - \frac{\partial}{\partial y_{\alpha j } } 
 \end{align*}
from Remark \ref{rem:weyl action}.  (The change from $2\pi i $ to $4\pi i $ is because our normalization of the symplectic form \eqref{big unitary symplectic} implies that 
$\langle e_j \otimes v_\alpha ,f_k \otimes v_\alpha  \rangle_{\bm{W}}$ is equal to $2 \delta_{jk}$, rather than $\delta_{jk}$.)
\end{proof}

\begin{proof}[Proof of Theorem \ref{thm:main unitary}]
In  the coordinates $x_{\alpha j}$ and $y_{\alpha j}$ on $V_\R^d$ determined by our orthonormal basis of $V_\R$, the standard Gaussian $\varphi_\R^\circ \in S(V_\R^d)$ from the proof of Theorem \ref{thm:hermitian theta modularity} becomes 
\[
\varphi^\circ_\R(v)  =  e^{-2\pi   \mathrm{Tr}(  H(v ) ) } 
= \prod_{ j =1}^d  e^{ - 2 \pi   \sum_\alpha  | z_{\alpha j} |^2     } .
\]
As in earlier arguments, if we can show that $\varphi_\R=P^\sharp \varphi^\circ_\R$ satisfies  the transformation law \eqref{unitary-unitary K weight} with 
\[
m = 2 N + \dim_E(V),
\]
then we are done by  Proposition \ref{prop:hermitian theta expansion}.

 Return to the setup of the proof of Theorem \ref{thm:hermitian theta modularity}, so that $K=K_1\times K_2$ is the stabilizer in $\mathrm{U}(W_\R)$ of the eigenspace  decomposition 
\[
W_\R  = W_\R^{J = i }  \oplus W_\R^{J = - i }  .
\]
Consider the corresponding Lie algebra decomposition 
$\mathfrak{k} = \mathfrak{k}_1\oplus \mathfrak{k}_2 \subset \mathfrak{u}(W_\R)$.

By directly computing the images of the vectors \eqref{XsYs} under \eqref{herm2u}, we see that they act on $W_\R$ as
\begin{align*}
X^{\FirstIndex}_{jk} \cdot ( e_\ell - i f_\ell  )  
& = 
- 2i \delta_{ j \ell }( e_k - i f_k )  - 2i \delta_{ k  \ell }( e_j - i f_j )   \\
X^{\FirstIndex}_{jk} \cdot ( e_\ell + i f_\ell  )  
& = 0 \\
Y^{\FirstIndex}_{jk} \cdot ( e_\ell - i f_\ell  )  
& = 
 2  \delta_{ j \ell }( e_k - i f_k )  - 2 \delta_{ k  \ell }( e_j - i f_j )   \\
Y^{\FirstIndex}_{jk} \cdot ( e_\ell + i f_\ell  )  
& = 0 
\end{align*}
and
 \begin{align*}
X^{\SecondIndex}_{jk} \cdot ( e_\ell - i f_\ell  )  
& =   0  \\
X^{\SecondIndex}_{jk} \cdot ( e_\ell + i f_\ell  )  
& =    2i \delta_{ j \ell }( e_k + i f_k )  + 2i \delta_{ k  \ell }( e_j + i f_j )     \\
Y^{\SecondIndex}_{jk} \cdot ( e_\ell - i f_\ell  )  
& =  0   \\
Y^{\SecondIndex}_{jk} \cdot ( e_\ell + i f_\ell  )  
& =  - 2  \delta_{ j \ell }( e_k + i f_k )  + 2 \delta_{ k  \ell }( e_j + i f_j )   .
\end{align*}
From these formulas and \eqref{K1K2}, one sees that 
\[
\mathfrak{k}_1  = \mathrm{Span}_\R\{ X'_{jk} , Y'_{jk} \}  \qquad \mbox{and}\qquad 
\mathfrak{k}_2  = \mathrm{Span}_\R\{ X''_{jk} , Y''_{jk} \}   ,
\]
and that  the traces of these vectors acting on $W_\R$, regarded as a vector space over $E_\R=\C$, are given by 
\begin{align}\label{XsYs trace}
 \mathrm{Tr}_W ( X^{\FirstIndex}_{jk}  ) & = - 4 i \delta_{jk} 
&
 \mathrm{Tr}_W ( X^{\SecondIndex}_{jk}  ) & =  4 i \delta_{jk}     \\ 
 \mathrm{Tr}_W ( Y^{\FirstIndex}_{jk}  ) & = 0
& 
\mathrm{Tr}_W ( Y^{\SecondIndex}_{jk}  ) & = 0   . \nonumber 
\end{align}

Using the formulas of Lemma \ref{lem:unitary XY operators}, a tedious calculus exercise shows that 
\begin{align*}
   \omega^\mathrm{inf}(  i_V  ( X^{\FirstIndex}_{jk}) )  (f \varphi^\circ_\R )
  &  =  2 i \cdot 
\left(        E^{\FirstIndex}_{jk} f  + E^{\FirstIndex}_{kj}f     -     \frac{ 1 }{4\pi }      F_{jk} f  +    \dim_E(V)  \delta_{jk}    f    \right)  \cdot    \varphi^\circ_\R     \\
  \omega^\mathrm{inf}(  i_V  ( X_{jk}^{\SecondIndex} ) )   ( f  \varphi^\circ_\R ) 
  &  =  
  2 i \cdot 
  \left(        E^{\SecondIndex}_{jk}  f + E^{\SecondIndex}_{kj}f    -    \frac{ 1 }{4\pi }     F_{jk} f  +  
   \dim_E(V) \delta_{jk}    f   \right)  \cdot     \varphi^\circ_\R    \\
   \omega^\mathrm{inf}(  i_V  ( Y^{\FirstIndex}_{jk}  )  ) (f \varphi^\circ_\R )  
   & =   
  2 \cdot   \left(     -   E^{\FirstIndex}_{ jk }  f      +   E^{\FirstIndex}_{k j } f +   \frac{  1    }{4 \pi i }   G_{jk}  f     \right)  \cdot    \varphi^\circ_\R  \\
  \omega^\mathrm{inf}(  i_V  ( Y^{\SecondIndex}_{jk}  )  )   (f \varphi^\circ_\R ) 
  & =   
   2 \cdot    \left(       E^{\SecondIndex}_{jk } f -   E^{\SecondIndex}_{kj} f       +   \frac{  1    }{4 \pi i }   G_{jk}f   \right) \cdot     \varphi^\circ_\R  ,
 \end{align*}
 for any polynomial function $f$ on $V_\R^d$.
 Taking $f=P^\sharp$  and using Lemma \ref{lem:unitary commutator}, we find 
\begin{align*}
  \omega^\mathrm{inf}(  i_V  ( X^{\FirstIndex}_{jk}) )  (P^\sharp \varphi^\circ_\R )
  &  =   2 i \delta_{jk}  \cdot \big( 2N +  \dim_E(V) \big)  \cdot P^\sharp \varphi^\circ_\R   \\
   \omega^\mathrm{inf}(  i_V  ( X_{jk}^{\SecondIndex} ) )   ( P^\sharp \varphi^\circ_\R ) 
  &  =   2 i \delta_{jk}  \cdot  \big( 2N +  \dim_E(V) \big)   \cdot P^\sharp \varphi^\circ_\R     \\
\omega^\mathrm{inf}(  i_V  ( Y^{\FirstIndex}_{jk}  )  ) (P^\sharp \varphi^\circ_\R )     & =    0  \\
    \omega^\mathrm{inf}(  i_V  ( Y^{\SecondIndex}_{jk}  )  )   (P^\sharp \varphi^\circ_\R )    & =     0  .
 \end{align*}
Comparing   these formulas with \eqref{XsYs trace}, we see that 
\begin{align}\label{Z1Z2 action}
\omega^\mathrm{inf}( i_V(Z_1 ) )  ( P^\sharp \varphi_\R^\circ )  
 & = -  \frac{ m  }{2}   \cdot  \mathrm{Tr}_W(Z_1)  \cdot P^\sharp \varphi_\R^\circ    \\
\omega^\mathrm{inf}( i_V(Z_2 ) )  ( P^\sharp \varphi_\R^\circ )  
 & =   \frac{ m  }{2}  \cdot   \mathrm{Tr}_W(Z_2)  \cdot  P^\sharp \varphi_\R^\circ   \nonumber
\end{align}
for all $Z_1 \in \mathfrak{k}_1$ and $Z_2 \in \mathfrak{k}_2$.

Now consider the action of $K=K_1\times K_2$ on $S(V_\R^d)$ through the Weil representation  $\omega^\eta_{V,\psi} = \omega_\psi \circ i^\eta_{V,\psi}$ of Definition \ref{def:unitary-unitary weil}.
Unlike in the proof of Theorem \ref{thm:main orthogonal}, 
the action of $\mathfrak{k}=\mathfrak{k}_1 \oplus \mathfrak{k}_2$  on $S(V_\R^d)$ obtained by differentiating  $\omega^\eta_{V,\psi}$  is \emph{not} the naive composition $\omega^\mathrm{inf}\circ i_V$, but rather is 
\[
Z \mapsto \frac{ r_\eta }{2}  \cdot  \mathrm{Tr}_W(Z) 
+ (  \omega^\mathrm{inf} \circ i_V) (Z)  .
\]
This is a consequence of  Theorem \ref{thm:inf compatibility}  and Lemma \ref{lem:kinf correction}.    
Taking  this correction  into account, it follows from  \eqref{Z1Z2 action} that  $\mathfrak{k}=\mathfrak{k}_1 \oplus \mathfrak{k}_2$ acts on  $P^\sharp \varphi^\circ_\R$   through the linear functional
\[
Z_1 +  Z_2   \mapsto  \left(  \frac{ - m + r_\eta }{2}   \right) \mathrm{Tr}_W(Z_1)   +  \left(  \frac{ m + r_\eta }{2}   \right) \mathrm{Tr}_W(Z_2).
\]
Recalling Remark \ref{rem:banana peel},  this shows that $K=K_1\times K_2$ acts on $P^\sharp \varphi^\circ_\R$ through the character in  \eqref{unitary-unitary K weight}, as was needed to prove the theorem.
\end{proof}


\subsection{Changing the conventions}
\label{ss:old normalization}


Suppose we were to change the definition of the symplectic form \eqref{big unitary symplectic} to 
\[
\langle  w_1 \otimes v_1 , w_2 \otimes v_2  \rangle_{\bm{W}} = 
\frac{1}{2} \cdot  \mathrm{Tr}_{E/F} \big(   \langle w_1,w_2 \rangle^\sigma \cdot  h(v_1,v_2)    \big).  
 \]
 Although rescaling the symplectic form on $\bm{W}$ changes the Weil representation $\omega_\psi$,  if we were to simultaneously change the definition of the matrix \eqref{hermitian gram} to 
 \[
 H(v) = \left(   \frac{ h(v_k  ,v_j) }{2}    \right)_{ 1 \le j,k \le d } \in \Herm_d(E) 
 \]
 then much of \S \ref{s:unitary-unitary} would  continue to hold as currently stated.

Under these new conventions the current statements of Propositions \ref{prop:unitary-unitary weil formulas} and \ref{prop:hermitian theta expansion}  would remain true (although the  function $\theta_{V,\psi}( \mathtt{z} ,\varphi)$ would have a different   meaning).
Theorem  \ref{thm:hermitian theta modularity} would continue to hold as written, but because rescaling the symplectic form on $\bm{W}$ would change the vacuum vector, in the proof  one must change the definition of the standard Gaussian to 
\[
\varphi^\circ_\R(v)  
= e^{-2\pi  \mathrm{Tr}( H (v  )) }   =  \prod_{j=1}^d e^{- \pi   h(v_j  , v_j  ) } .
\]

Although many formulas in the proof of Theorem \ref{thm:main unitary} would change, the statement of that theorem would continue to hold with a single modification:  the definition of $P^\sharp$ must  be replaced by 
\[
P^\sharp =  \exp\left( -\frac{\Delta}{8\pi} \right)  P  .
\]

\bibliographystyle{amsalpha}


\providecommand{\bysame}{\leavevmode\hbox to3em{\hrulefill}\thinspace}
\providecommand{\MR}{\relax\ifhmode\unskip\space\fi MR }
\providecommand{\MRhref}[2]{%
  \href{http://www.ams.org/mathscinet-getitem?mr=#1}{#2}
}
\providecommand{\href}[2]{#2}

\end{document}